\documentclass{article}
\usepackage{amsmath, amsthm, amssymb}
\usepackage{epsfig}
\usepackage{url}
\usepackage{latexsym}
\usepackage[colorlinks, bookmarks=true]{hyperref}
\usepackage{amsmath}
\usepackage{bm}           

\def\R{\mathbb R}

\def\Ca{\mathcal C}

\def\tr{\mbox{tr}}
\def\ts{\times}

\def\R{{\mathbb R}}
\def\ra{\rightarrow}

\usepackage{mathrsfs}

\usepackage{color}

\makeatletter
\def\thm@space@setup{%
  \thm@preskip=\parskip \thm@postskip=0pt
}
\makeatother

\numberwithin{equation}{section}

\renewcommand{\cal}{\mathcal}
\newcommand\cA{{\mathcal A}}

\newcommand{\cF}{{\cal F}}

\newcommand\cH{{\mathcal H}}
\newcommand\cI{{\mathcal I}}

\newcommand{\cL}{{\cal L}}
\newcommand{\cM}{{\cal M}}

\newcommand{\fc}{{\frak c}}

\newcommand{\fK}{{\frak K}}
\newcommand{\fR}{{\frak R}}
\newcommand{\bma}{{\bm{a}}}

\newcommand{\bmx}{{\bm{x}}}
\newcommand{\bmy}{{\bm{y}}}

\newcommand{\rd}{{\rm d}}

\newcommand{\ri}{\mathrm{i}}

\newcommand{\bB}{{\mathbb B}}
\newcommand{\bC}{{\mathbb C}}
\newcommand{\bD}{\mathbb{D}}
\newcommand{\bE}{\mathbb{E}}

\newcommand{\bP}{\mathbb{P}}

\newcommand{\bR}{{\mathbb R}}

\newcommand{\bY}{\mathbb{Y}}

\newcommand{\sfa}{{\sf a}}
\newcommand{\sfb}{{\sf b}}

\newcommand{\sfc}{{\sf c}}
\newcommand{\sfy}{{\sf y}}
\newcommand{\al}{\alpha}

\newcommand{\la}{\lambda}

\DeclareMathOperator{\Tr}{Tr}

\DeclareMathOperator{\supp}{supp}

\DeclareMathOperator{\diag}{diag}

\DeclareMathOperator{\OO}{O}
\DeclareMathOperator{\oo}{o}

\renewcommand{\Re}{\mathop{\mathrm{Re}}}
\renewcommand{\Im}{\mathop{\mathrm{Im}}}

\def\tr{{\rm Tr}}

\newcommand{\deq}{\mathrel{\mathop:}=} 
\newcommand{\eqd}{=\mathrel{\mathop:}} 
\renewcommand{\leq}{\leqslant}
\renewcommand{\geq}{\geqslant}

\newcommand{\del}{\partial}

\newcommand{\qq}[1]{[\![{#1}]\!]}

\newcommand{\beq}{\begin{equation}}
\newcommand{\eeq}{\end{equation}}

\theoremstyle{plain} 
\newtheorem{theorem}{Theorem}[section]
\newtheorem*{theorem*}{Theorem}
\newtheorem{lemma}[theorem]{Lemma}
\newtheorem*{lemma*}{Lemma}
\newtheorem{corollary}[theorem]{Corollary}
\newtheorem*{corollary*}{Corollary}
\newtheorem{proposition}[theorem]{Proposition}
\newtheorem*{proposition*}{Proposition}

\newtheorem*{assumption*}{Assumption}
\newtheorem{claim}[theorem]{Claim}

\newtheorem*{definition*}{Definition}

\newtheorem*{example*}{Example}
\newtheorem{remark}[theorem]{Remark}

\newtheorem*{remark*}{Remark}
\newtheorem*{remarks*}{Remarks}

\def\author#1{\par
    {\centering{\authorfont#1}\par\vspace*{0.05in}}
}

\def\titlefont{\fontsize{13}{15}\bfseries\boldmath\selectfont\centering{}}
\def\authorfont{\fontsize{13}{15}}

\let\affiliationfont\rhfont

\def\address#1{\par
    {\centering{\affiliationfont#1\par}}\par\vspace*{11pt}
}

\def\body{
\setcounter{footnote}{0}
\def\thefootnote{\alph{footnote}}
\def\@makefnmark{{$^{\rm \@thefnmark}$}}
}

\def\title#1{
    \thispagestyle{plain}
    \vspace*{-14pt}
    \vskip 79pt
    {\centering{\titlefont #1\par}}%
    \vskip 1em
}

\newcommand{\bmla}{{\bm{\lambda}}}
\newcommand{\bmmu}{{\bm{\mu}}}
\newcommand{\bmeta}{{\bm{\eta}}}
\newcommand{\bmkappa}{{\bm{\kappa}}}

\newcommand{\unif}{{\rm unif}}

\begin{document}
\title{Large Deviation Principles via Spherical Integrals}

\vspace{1.2cm}

 \begin{minipage}[c]{0.31\textwidth}
 \author{Serban Belinschi}
\address{CNRS-Universit\'e de Toulouse\\
 Serban.Belinschi@math.univ-toulouse.fr}
    \end{minipage}
 \begin{minipage}[c]{0.31\textwidth}
 \author{Alice Guionnet}
\address{CNRS-ENS Lyon\\
aguionne@ens-lyon.fr}
      \end{minipage}
 \begin{minipage}[c]{0.31\textwidth}
 \author{Jiaoyang Huang}
\address{NYU\\
 jh4427@nyu.edu}
    \end{minipage}

~\vspace{0.3cm}

\abstract
In this article, we develop a framework to study the large deviation principle for matrix models and their quantized versions, by tilting the measures using the limits of  spherical integrals obtained in \cite{GZ3,GZei1add}. As examples, we obtain 
\begin{enumerate}
\item the large deviation principle for the empirical distribution of the diagonal entries of $UB_NU^*$, for a sequence of $N\times N$ diagonal matrices $B_N$ and unitary or orthogonal  Haar distributed matrices $U$,
\item a large deviation upper bound for the empirical eigenvalue distribution of $A_N+UB_NU^*$, for two sequences of $N\times N$ diagonal matrices $A_N, B_N$, and their complementary lower bounds at measures which are described by the free product with amalgamation,
\item a large deviation principle for the Kostka number $K_{\bm\la_N \bmeta_N}$, for two sequences of partitions $\bm\la_N, \bmeta_N$ with at most $N$ rows,
\item  a large deviation upper bound for the Littlewood-Richardson coefficients $c_{\bm\la_N \bmeta_N}^{\bmkappa_N}$, for three sequences of partitions $\bm\la_N, \bmeta_N, \bmkappa_N$ with at most $N$ rows, and their complementary lower bounds at nice  measures.
\end{enumerate} 
\date{\today}

\section{Introduction}
During the last thirty years, 
  random matrix theory  has  
  grown  into a sophisticated branch of mathematics, interacting profoundly with physics  and other areas of mathematics such as statistics,  probability and operator algebra. Following the initial breakthroughs by
 Wishart \cite{wishart} and Wigner \cite{Wig58},  the convergence of the empirical measure of the eigenvalues and the extreme eigenvalues could be established for many  models of random matrices. The 
  fluctuations of the eigenvalues both in the local and global scale  were investigated. However, the understanding of the probabilities that the spectrum has an unlikely behavior, as measured  by large deviations principles,  is still very scarce, even at a conjectural level.

  For instance, let us consider  Wigner matrices , that are  self-adjoint matrices with independent  (modulo the symmetry constraint) centered entries with covariance given  by the inverse of the dimension. It has been shown that the empirical measure of such random matrices converges towards a non-random limit and that the extreme eigenvalues  ``stick'' to the bulk in the sense that they converge towards the boundary of the support of this limiting distribution as soon as their fourth moment is finite \cite{BZ,BY,furedi}. Fluctuations around this limit could be studied. It was shown that generically, the empirical eigenvalue distribution has small fluctuations. The central limit theorem for the empirical measure holds without the celebrated normalization by the square root of the dimension required for sums of independent random variables \cite{SS98, DE01, BS04, PaLy08}. Local fluctuations of the eigenvalues were first proven for Gaussian ensembles \cite{TW,TW94b} and more recently generalized to Wigner matrices \cite{MR1727234, erdy,ESY11,ESY,EPRY, MR3914908,MR3541852, LY,TV, TVun}. Large deviation principles, which allow to estimate the probability to deviate from the almost sure asymptotic behavior, are still much less understood. They were first derived for random matrices with Gaussian entries thanks to the explicit joint law of the eigenvalues \cite{BAG, BADG01}.  Large deviation principles were then obtained  for random matrices with entries with tails  heavier than the Gaussians \cite{BCC,fanny}, by using the fact that the matrix can more easily create deviations by having a few large entries. The case of sub-Gaussian entries thus stayed open until recently.
 F. Augeri, J. Husson and one of the author tackled the large deviations for the largest eigenvalue \cite{HuGu,AGH}. They showed that the  large deviation rate function is the same as in the Gaussian case if the entries have a Laplace transform which is bounded above by the Laplace transform of the centered Gaussian variable with the same variance, but is otherwise different. The large deviations for the eigenvalue empirical measure is still open in the general sub-Gaussian case.

 The proof of the large deviation estimates  in \cite{HuGu,AGH}  is  based on a tilt of the distribution by spherical integrals. Let us remind the reader that one 
 approach to the large deviation principles is to tilt the measure by an exponential moment generating function. This technique for instance allows to prove the celebrated Cramer's theorem for the empirical distribution of independent variables taking their values in Polish spaces \cite{DZ,DS}. For instance, consider 
 a sequence of vector-valued random variables $\{X_N\}_{N\geq 1}$ in $\bR^d$, and  let $\bP_\theta$ be the tilted measure by  $e^{a_{N}\langle \theta, X_N\rangle}$ if $a_{N}\rightarrow \infty$ is the expected speed for the large deviation principle. The idea is that  rare events for the original measure become typical events under the tilted measure for an appropriate choice of $\theta$, and can be studied through the following relation
\begin{align*}
\bP(X_N\in x+\rd x)=\bP_\theta (X_N\in  x+\rd x) \exp\left\{-a_{N}\{\langle \theta, x\rangle+\frac{1}{a_{N}}\log \bE\left[\exp\{a_{N}\langle \theta, X_N\rangle\}\right]\}\right\},
\end{align*}
where the asymptotics of the last term on the right hand side can eventually be computed (for instance if it is the sum of independent i.i.d random vectors). 
Unfortunately, if $X_N$ is a sequence of random matrices and one considers the law of its empirical measure or its largest eigenvalue $\la_{X_N}$, it is not clear how to make such a computation. Because  in general, either the tilt is a function of the largest eigenvalue and then we do not know how to compute  its Laplace transform, or it is for instance a tilt on each individual entry but then we do not know how it is related with the largest eigenvalue. The idea in \cite{HuGu,AGH} to study the large deviation for the largest eigenvalue was to tilt the probability measure by spherical integrals, because we know \cite{GuMa05} it becomes a function of the largest eigenvalue when the dimension goes to infinity, but also it produces independent (but random) tilt on each entry with a Laplace transform which can eventually be evaluated. It amounts to tilt the law of the matrix in 
 a random direction to make the desired deviation more likely. The tilted measure $\mathbb P_\theta$ is given by
$$\mathbb P_\theta \left(\lambda_{X_N}\in \la+\rd \la \right) =\frac{1}{\int I_N(\theta,X_N) \rd\mathbb P(X_N)}\int {\bf 1}_{\{\lambda_{X_N}\in \la+\rd \la\}} I_N(\theta, X_N) \rd\mathbb P(X_N),$$
where $I_N$ is the one dimensional spherical integral
$$I_N(\theta,X_N)=\mathbb E_u[e^{\theta N\langle u, X_N u\rangle}],$$
and $u$ follows the uniform law on the unit sphere in $\mathbb R^N$. It was proven that given a real number $\la$, one can often (for instance when $\la$ is sufficiently large)  find a tilt $\theta_\la$ so that under
$\mathbb P_{\theta _\la}$, the largest eigenvalue $\lambda_{X_N}$ is close to $\la$, whereas the cost of this tilt can be  computed thanks to \cite{GuMa05} and the expectation of the spherical integral over $X_{N}$ can be estimated \cite{HuGu,AGH}. 
Yet, deriving  the large deviations for the distribution of the empirical eigenvalue density of Wigner matrices is still an open problem.  Even though the large dimensional limit of spherical integrals can also be computed \cite{GZ3}, its average over matrices was not yet estimated.

In this article, we study the large deviations for the empirical measures of different models of random matrices. The first model concerns the diagonal entries of Hermitian matrices with given eigenvalues.
Let $B_N$ be an $N\times N$ diagonal matrix with eigenvalues 
$b_1\geq b_2\geq \cdots\geq  b_N$. The classical Schur-Horn theorem \cite{MR0063336} states that
the diagonal entries of $UB_NU^*$ as a vector is in the permutation polytope  generated by $(b_1,b_2,\cdots, b_N)$. More precisely, if we arrange the diagonal entries of $UB_NU^*$ in nonincreasing order $(UB_NU^*)_{\sigma(1)\sigma(1)}\geq (UB_NU^*)_{\sigma(2)\sigma(2)}\geq \cdots\geq (UB_NU^*)_{\sigma(N)\sigma(N)}$ for some permutation of $\sigma$, then
\begin{align}\label{e:schur-horndis}
b_1+b_2+\cdots+b_k\geq (UB_NU^*)_{\sigma(1)\sigma(1)}+ (UB_NU^*)_{\sigma(2)\sigma(2)}+\cdots+ (UB_NU^*)_{\sigma(k)\sigma(k)},\quad 1\leq k\leq N.
\end{align}
A more challenging problem of the same flavor is Horn's problem: Given two $N\times N$ diagonal matrices $A_N, B_N$ with eigenvalues 
$a_1\geq a_2\geq\cdots\geq a_N$ and $b_1\geq b_2,\cdots\geq b_N$, what can be said about the eigenvalues of $A_N+UB_NU^*$? Besides the trivial relation $\Tr A_N +\Tr B_N =\Tr (A_N+UB_NU^*)$, Horn \cite{MR0140521} had conjectured the form of a set of necessary and sufficient inequalities to be satisfied for the eigenvalues of $A_N+UB_N U^*$. After contributions by several authors, see in particular
\cite{MR1654578}, these conjectures were proven by Knutson and Tao \cite{MR1671451,MR1811121}. See \cite{MR1754641} for a nice survey of this problem. This result however  do not say anything about the probability that the spectrum of $A_N+UB_NU^*$ has some given distribution, if $U$ is random.

Let $U$ follow the Haar measure on the orthogonal group when $\beta=1$ and on the unitary group when $\beta=2$. The randomized Schur-Horn theorem and Horn's problem ask the distribution of the diagonal entries of $UB_NU^*$ and the empirical eigenvalue distribution of $A_N+UB_NU^*$:
\begin{equation}\label{em}
\mu_N=\frac{1}{N}\sum_{i=1}^N \delta_{(UB_NU^*)_{ii}},\quad 
\tilde \mu_N=\frac{1}{N}\sum_{i=1}^N \delta_{\lambda_i(A_N+UB_NU^*)}\,.
\end{equation}
The randomized Schur-Horn theorem is equivalent to computing the
Duistermaat--Heckman measures of the coadjoint orbit $UB_NU^*$. More generally, they are defined using the push-forward of the Liouville measure on a symplectic manifold along the moment map \cite{MR696693,MR674406,MR664117,MR364552,MR642416}.
The randomized version of Horn's problem has
been first studied in \cite{MR1090007} for $N=3$. Recent years, there has seen a surge of interest
in this problem: \cite{MR3939342, MR3835548, MR3892253, MR4009590,Zhang_2019} for general dimension; \cite{CMZ} for an extension to other Lie groups; \cite{FZ} for the multiplicative version of randomized Horn's problem. However, those distribution densities of randomized Schur-Horn theorem and Horn's problem are complicated and hard to analyze directly. 

For the large $N$ limit, we will assume that the spectral measures of $A_N,B_N$ converge towards $\mu_A$ and $\mu_B$ respectively. It is well known that $\mu_N$ converges weakly almost surely towards a delta mass at $\int x\rd\mu_B(x)$. On the other hand, Voiculescu \cite{BV93, voi91} proved that $\tilde\mu_N$ converges towards the free convolution of $\mu_A$ and $\mu_B$. For the second model, fluctuations of the empirical measure where studied in \cite{CGM} and large deviations for the largest eigenvalue in \cite{GuMa20}.
To study the large deviations for $\mu_{N}$ and $\tilde \mu_{N}$ from this asymptotic behavior,  we propose to tilt the original measure using the spherical integrals, which are Fourier transforms over Unitary/Orthogonal groups. We recall that 
given two sequences $A_{N},B_{N}$ of deterministic self-adjoint matrices, the $N$ dimensional spherical integral is defined as
\begin{align*}
I_N(A_N,B_N)=\int e^{\frac{\beta N}{2}\Tr(A_NUB_NU^{*})} \rd U,
 \end{align*}
where $U$ follows the Haar measure on the unitary group (resp. orthogonal group) when $\beta=2$ (resp. $\beta=1$). If the spectral measures $\mu_{A_N},\mu_{B_N}$ of $A_N, B_N$ converge weakly towards $\mu_{A}$ and $\mu_{B}$ respectively, under mild assumptions, it was proven in \cite{GZ3, GZei1add}, see also \cite{guionnet2004large,MR2095566}, that the spherical integral converges
\begin{align}
 \label{e:defI}I(\mu_{A},\mu_{B})=\lim_{N\rightarrow\infty}\frac{1}{\beta N^{2}}\log\int e^{\frac{\beta N}{2}\Tr(A_NUB_NU^{*})} \rd U.
 \end{align}
The density of the spectrum of the eigenvalues of $A+UBU^*$ and of the diagonal entries of $UBU^*$ was expressed in terms of spherical integrals in \cite{zuber,CoZu,MR3939342,CoMZu} by using Fourier analysis. This however requires to take  spherical integrals evaluated at complex entries and use Fourier analysis. It is hard to see how to use such an approach to obtain asymptotics,  as spherical integrals need to be evaluated at complex matrices
 for which asymptotics are only known when the norms $A$ and $B$ are sufficiently small \cite{novak2020complex}.
Our approach by tilting the original measure using the spherical integrals, requires only the derivatives for the limiting spherical integral, which we derive in Section \ref{sec:sp}. As a consequence, it gives new understandings of the Schur-Horn theorem and Horn's problem, as well as the evaluation of the asymptotics of Kostka numbers and Littlewood-Richardson coefficients.

As the first application of our spherical integral approach, we study large deviations of the randomized Schur-Horn theorem and Horn's problem. We obtain the large deviation principle for the empirical distribution of the diagonal entries of $UB_NU^*$ and
the large deviation upper bound for the empirical eigenvalue distribution of $A_N+UB_NU^*$. 
In a recent work \cite{narayanan2021large} by Narayanan and Sheffield, a large deviation for the empirical measure of $A_N+UB_NU^*$ is proven, under the assumption that the partition function of the limiting empirical measures of $A_N, B_N$  are strongly concave. Their approach is totally different, being based on random hives, and understanding the relation between their rate function and the one of the present article should be enlightening.

The quantized versions of the Horn-Schur theorem and Horn's problem ask when a Kostka number or a Littlewood-Richardson coefficient is non-zero. We denote by $\bY_N$ the set of partitions with at most $N$ rows. We recall that given a partition $\bmla_N\in \bY_N$, the Kostka numbers $K_{\bmla_N \bmeta_N}$ are the coefficients that arise when one expresses the Schur symmetric polynomial $S_{\bmla_N}$ as a linear combination of monomial symmetric functions $m_{\bmeta_N}$:
\begin{align}\label{e:defKostka}
S_{\bmla_N}(x_1, x_2, \cdots, x_N)=\sum_{\bmmu_N\in \bY_N}K_{\bmla_N\bmeta_N}m_{\bmeta_N}(x_1, x_2, \cdots, x_N).
\end{align}
The Kostka numbers $K_{\bmla_N \bmmu_N}$ are nonnegative, and are positive if and only if $\bmla_N$ and $\bmeta_N$ are of the same size (have the same number of boxes), and $\bmla_N$ is larger than $\bmeta_N$ in the dominance order:
\begin{align*}
\lambda_1+\la_2+\cdots+\la_i\geq \eta_1+\eta_2+\cdots+\eta_i, \quad 1\leq i\leq N.
\end{align*}
Given a pair of  partitions $\bmla_N, \bmeta_N\in \bY_N$, the Littlewood-Richardson coefficients $c_{\bmla_N \bmeta_N}^{\bm\kappa_N}$ are the coefficients that arise when one expresses the product of the Schur symmetric polynomials $S_{\bmla_N}S_{\bmeta_N}$ as a linear combination of Schur symmetric polynomials $S_{\bm\kappa_N}$:
\begin{align}\label{e:defLR}
S_{\bmla_N}(x_1, x_2, \cdots, x_N)S_{\bmmu_N}(x_1, x_2, \cdots, x_N)=\sum_{\bmkappa_N\in \bY_N}c_{\bmla_N\bmeta_N}^{\bmkappa_N}S_{\bmkappa_N}(x_1, x_2, \cdots, x_N).
\end{align}
The Littlewood-Richardson coefficients $c_{\bmla_N, \bmeta_N}^{\bmkappa_N}$ are nonnegative. 
Horn's problem is equivalent to deciding the conditions on the triples $(\bmla_N, \bmeta_N, \bmkappa_N)$, such that the Littlewood-Richardson coefficient $c_{\bmla_N, \bmeta_N}^{\bmkappa_N}$ is positive.  This result has previously been obtained by Klyachko using geometric invariant
theory \cite{MR1654578}.

As the second application of our spherical integral approach, we derive the large deviation principle of Kostka numbers and the large deviation upper bound for the Littlewood-Richardson coefficients.
The asymptotics of certain extreme Kostka numbers and the Littlewood-Richardson coefficients were derived in \cite{MR3906648,MR1676282,MR2868112}. 

\smallskip
\noindent \textbf{Acknowledgement.} The authors heartily thank Denis Serre for his help to guess  the derivative of spherical  integrals  at the early stage of this work. A.G. thanks Ofer Zeitouni  for enlightening discussions. H.J. wants to thank Vadim Gorin for the references on Duistermaat--Heckman measures, and  the Institute for Advanced Study for its support. The authors are very grateful to Mikael de la Salle for providing a shorter proof of Proposition \ref{c:continuityfa}. Part of this work was done while STB was supported by the program Simons CRM Scholar-in-Residence during the program {\em New Developments in Free Probability and Applications} at the Centre de Recherches Math\'ematiques in Montreal, March 2019. 
A.G. was partially  supported by the LABEX MILYON (ANR-10-LABX-0070) of Universit\'e de Lyon, within the program ``Investissements d'Avenir" (ANR-11-IDEX- 0007) operated by the French National Research Agency (ANR) and by the ERC Project LDRAM : ERC-2019-ADG Project 884584.

\subsection{Main Results}


Before stating our main results, we need to introduce some notations and definitions. 
In this paper, we fix a large constant $\fK>0$. We denote by  $\cM$(resp. $\cM([-\fK,\fK])$) the space of probability measures on $\bR$ (resp. $[-\mathfrak K,\mathfrak K]$), with bounded first moment:
\begin{align*}
\cM=\{\nu:\supp \nu \subset \bR, \nu(|x|)<\infty\},\quad \cM([-\fK, \fK])=\{\nu:\supp \nu \subset [-\fK,\fK]\}.
\end{align*}
We equip $\cM$ and  $\cM([-\fK,\fK])$, with the weak topology. The weak topology is compatible with the following distance
\begin{align}\label{e:dd}
\rd(\mu, \nu)\deq\sup_{\|f\|_\infty+\|f\|_\cL\leq 1}\left|\int f(x) \rd\mu-\int f(x) \rd\nu\right|,
\end{align}
where $\|f\|_\cL$ is the Lipschitz constant of $f$, see \cite[Appendix C]{AGZ}.
For probability measures with bounded first moment, a more natural distance is the Wasserstein distance:
\begin{align}\label{e:wd}
\rd_W(\mu, \nu)\deq\sup_{\|f\|_\cL\leq 1}\left|\int f(x) \rd\mu-\int f(x) \rd\nu\right|.
\end{align}
We remark the convergence in the Wasserstein distance $\rd_W$ is equivalent to weak convergence and the convergence of the first moment, see \cite{clement2008elementary}. On the set of uniformly compactly supported measures, i.e. $\cM([-\fK, \fK])$, convergences in the Wasserstein distance and weak convergence are equivalent.
We also denote by $\cM^{\rm b}([0,\fK])$ the set of probability measures on $[0,\fK]$ with density bounded by $1$, i.e. the set of probabilities $\rd\nu=\rho(x)\rd x$, such that $\rho(x)$ is supported on $[0,\fK]$ and $\rho(x)\leq 1$.

Given a probability measure $\mu$, let $T_\mu\colon(0,1)\mapsto (-\infty, \infty)$  be the right continuous increasing function, such that $\mu$ is the push-forward of the uniform distribution on $(0,1)$ by $T_\mu$. In other words, for all bounded continuous function $f$, we have 
\begin{align}\label{e:defT}
\int_{0}^{1} f(T_{\mu}(x))\rd x=\int f(x)\,{\rm d}\mu(x).
\end{align}
More explicitly, $T_\mu$ is the functional inverse of the cumulative density function $F_\mu$ of $\mu$. With this notation, we can rewrite the Wasserstein distance \eqref{e:wd} as
\begin{align}\label{e:wd2}
\rd_W(\mu, \nu)\deq\sup_{\|f\|_\cL\leq 1}\left|\int f(x) \rd\mu-\int f(x) \rd\nu\right|
=\int |T_\mu(x)-T_{\nu}(x)|\rd x.
\end{align}

\begin{theorem}\label{t:UBU}
Let $B_N$ be a sequence of deterministic self-adjoint matrices such that the spectral measures $\mu_{B_N}$ of $B_N$ converge weakly towards $\mu_B$ as $N\to\infty$. Assume there exists a constant $\fK>0$, such that $\supp \mu_{B_N}\subset [-\fK,\fK]$. 
\begin{enumerate}
\item Let $\mu\in \cM([-\fK,\fK])$, $\nu\in \cM$ and set
\begin{align}\label{e:defHD}
H^{\mathsf D}_\mu(\nu)=\frac{1}{2}\int T_\nu(x) T_\mu(x) \rd x-I(\nu,\mu_{B})\,,
\end{align}
where {$I(\cdot, \cdot)$ is defined in \eqref{e:defI} and} $T_\mu, T_\nu$ are as defined in \eqref{e:defT}. 
Then, the functional $\mathcal I^{\mathsf D}(\cdot)$
 \begin{align} \label{e:defID}
 \mathcal I^{\mathsf D}(\mu):=\sup_{\nu\in\mathcal M}H^{\mathsf D}_{\mu}(\nu),
 \end{align}
  is non-negative, lower semicontinuous on $\cM([-\fK,\fK])$ and vanishes only at the Dirac mass at $\int x\rd\mu_B$. $\mathcal I^{\mathsf D}(\mu)=+\infty$ unless $\int x\rd\mu=\int x\rd \mu_B$ and $\mu$ satisfies the limiting Schur-Horn inequalities:
\begin{equation}\label{limSH} \int_y^1 (T_{\mu}(x)-T_{\mu_B}(x))\rd x\leq 0\quad\text{for all }y\in [0,1],\end{equation}
which is the continuum limit of \eqref{e:schur-horndis}.
\item 
The distribution of the empirical measure of the diagonal entries of $UB_N U^*$,
\begin{align*}
\mu_N=\frac{1}{N}\sum_{i=1}^N \delta_{(UB_NU^*)_{ii}},
\end{align*}
satisfies a large deviation principle with good rate function $\mathcal I^{\mathsf D}$.  In other words, for any $\mu\in \cM([-\fK, \fK])$, if $ \bB_{\delta}(\mu)$ denotes the open ball $\{\nu\in \cM([-\fK,\fK]): \rd(\nu,\mu)<\delta\}$,
\begin{align}\label{e:UBU}
\lim_{\delta\rightarrow 0}\liminf_{N\rightarrow \infty}\frac{1}{\beta N^2}\log \bP(\mu_N\in \bB_{\delta}(\mu))=\lim_{\delta\rightarrow 0}\limsup_{N\rightarrow \infty}\frac{1}{\beta N^2}\log \bP(\mu_N\in \bB_{\delta}(\mu))
= - \mathcal I^{\mathsf D}(\mu).
\end{align}
\end{enumerate}
\end{theorem}

\begin{remark}
Since $\mu_{B_{N}}$ is supported on $[-\fK, \fK]$, then deterministically we have
$|(UB_NU^*)_{ii}|\leq \fK$ for all $1\leq i\leq N$. Therefore, for our study of the large deviation principle of the empirical measure $\mu_N$ in Theorem \ref{t:UBU},  we have restricted ourselves in the set of measures supported on $[-\fK, \fK]$, which is compact in the weak topology. If we do not restrict ourselves to the set of measures supported on $[-\fK, \fK]$, the function $H_\mu^{\mathsf D}(\cdot)$ as in \eqref{e:defHD} may not be well-defined. However, if one simply set $H_\mu^{\mathsf D}(\cdot)$ to be $+\infty$ when it is not well-defined, then the large deviation principle \eqref{e:UBU} holds for any probability $\mu$.
We also notice that since $\cM([-\fK,\fK])$ is compact, the above weak large deviation principle is equivalent to a full large deviations principle.
\end{remark}
%
%
%
%
Before stating the large deviation results for the empirical eigenvalue distributions of $A_N+UB_NU^*$, we need to introduce a notation. 
Take $\beta=2$. Let $C_{N},Y_{N}$ be two sequences of deterministic self-adjoint matrices, such that their spectral measures $\mu_{C_N}$ and $\mu_{Y_N}$ converge in Wasserstein distance \eqref{e:wd} towards $\mu_{C}$ and $\mu_{Y}$ respectively. We assume that  there exists a constant $\fK>0$, such that $ \supp\mu_{C_N}\subset [-\fK,\fK]$ and $\mu_{Y_N}(|x|)\leq \fK$.  
Then the joint distribution of $(Y_N,UC_N U^*)$ converges towards a non-commutative law $\tau_{\mu_C,\mu_Y}$  under the tilted measure :
\begin{align}
\rd\mu_{C_N, Y_{N}}(U)=\frac{e^{\frac{\beta N }{2}\Tr (Y_N U C_N U^*)}\rd U}{\int e^{\frac{\beta N}{2}\Tr (Y_N V C_N V^*)}\rd V}\label{e:lawU0}.
\end{align}
In other words, for any $k\geq 0$, any choice of $z_1,\ldots,z_k\in\mathbb C\backslash\mathbb R$ and integer numbers $n_1,\ldots,n_k\geq 0$, let  
 $F(\sfy,\sfc)=\frac{1}{z_1-\mathsf y}\sfc^{n_1} \frac{1}{z_2-\mathsf y}\sfc^{n_2} \cdots \frac{1}{z_k-\mathsf y} \sfc^{n_k}$, we have 
\begin{align}\label{e:tauF0}
\lim_{N\rightarrow \infty}\int \frac{1}{N}\Tr(F(Y_N,UC_NU^{*}))\rd\mu_{C_N, Y_N}(U)=\tau_{\mu_C, \mu_Y}(F(\mathsf y,\mathsf c)).
\end{align}
A proof of \eqref{e:tauF0} will be given in Theorem \ref{t:convl1} statement \eqref{e:tauF}. Although, the convergence holds for $\beta=2$, the limiting object $\tau_{\mu_C, \mu_Y}$ is independent of $\beta$.

As many other results from random matrix theory, the above statement and (parts of) the following theorem benefit from notions and 
results from noncommutative, and especially free, probability. In order to keep this section concise, we postpone the introduction of notions and results from 
noncommutative probability to Sections \ref{s:prel} (noncommutative probability spaces and freeness) and \ref{Subsection:4.3} (freeness with amalgamation).
\begin{theorem}\label{t:AUBU}
Let $A_N, B_N$ be a sequence of deterministic self-adjoint matrices, such that their spectral measures $\mu_{A_N}, \mu_{B_N}$ converge weakly towards $\mu_A, \mu_B$ respectively. Assume there exists a constant $\fK$, such that $\supp\mu_{A_N}, \supp\mu_{B_N}\subset [-\fK,\fK]$. Then
\begin{enumerate}
\item
Let  $\mu\in \cM([-2\fK,2\fK])$, $\nu\in \mathcal M$ and set
\begin{align}\label{e:defHA+B}
H^{\mathsf{A+B}}_\mu(\nu)=I(\nu, \mu)-I(\nu,\mu_A)-I(\nu,\mu_B)\,,
\end{align}
where $I(\cdot, \cdot)$ is defined in \eqref{e:defI}. 
The functional $\mathcal I^{\mathsf{A+B}}(\cdot)$
\begin{align}\label{e:defIAB}
\mathcal I^{\mathsf{A+B}}(\mu):=\sup_{\nu\in\mathcal M}H^{\mathsf{A+B}}_{\mu}(\nu)
\end{align}
 is non-negative and lower semicontinuous on $\cM([-2\fK,2\fK])$. $\mathcal I^{\mathsf{A+B}}(\mu)=\infty$ 
unless $\int x\rd \mu=\int x\rd \mu_A+\int x\rd \mu_B$ and the limiting Ky Fan inequalities hold:
\begin{align*}
\int_y^1  T_{\mu_A}(x)\rd x
+\int_y^1 T_{\mu_B}(x)\rd x
\geq 
\int_y^1 T_{\mu}(x)\rd x,\quad \text{for all $y\in [0,1]$},
\end{align*}
where is the continuum limit of the Ky Fan inequality \eqref{e:KyFan}.
\item 
Let $ \bB_{\delta}(\mu)$ denote the open ball $\{\nu\in \cM([-2\fK,2\fK]): \rd(\nu,\mu)<\delta\}$.
The empirical eigenvalue distribution $\mu_N$ of $A_N+UB_NU^{*}$ satisfies the large deviation upper bound 
\begin{align}\begin{split}\label{e:AUBU}
&\limsup_{\delta\rightarrow 0}\limsup_{N\rightarrow \infty}\frac{1}{\beta N^2}\log \bP(\mu_N\in \bB_{\delta}(\mu))
\leq - \mathcal I^{\mathsf{A+B}}(\mu).
\end{split}\end{align}
and the complementary lower bound  holds  if $\mu$ is in  the set  $\cH^{\mathsf{A+B}}$ of measures obtained by free product 
with amalgamation defined as follows: $\cH^{\sf A+B}$ is the set of probability measures given as the law of $\sfa+\sfb$
\begin{align}\label{e:defHAB}
\cH^{\sf A+B}=\{\text{law of } \sfa+\sfb\}
\end{align}
where $\sfa, \sfb$ are non-commutative variables with laws $\mu_A, \mu_B$, and such that  there exists another non-commutative variable $\sfy\in L^1$ with law $\mu_Y$ and $\mu_{Y}(|x|)<\infty$, so that $\sfa$ and $\sfb$ are free with amalgamation over $\sfy$ with marginal distributions $\tau_{\mu_A,\mu_Y}$ and $\tau_{\mu_B,\mu_Y}$ (as in \eqref{e:tauF0} by taking $\mu_C=\mu_A, \mu_B$).
\end{enumerate}
\end{theorem}

Freeness with amalgamation is described by \eqref{freeam}. If we take $\mu_Y$ to be the delta mass at $0$, i.e. $\mu_Y=\delta_0$, then the law of $\sfa+\sfb$ is the free convolution of  $\mu_A$ and $\mu_B$.
The set 
 $\cH^{\mathsf{A+B}}$ is closed for the weak topology and probably not dense 
  in the set of measures which are characterized by the limiting Schur-Horn inequalities. When the supremum defining $\mathcal I^{\mathsf{A+B}}(\mu)$ in \eqref{e:defIAB} is achieved at a probability measure $\nu$ which is compactly supported
and such that 
all connected components of $\supp\nu$ are infinite sets (that is, $\nu$ has no isolated point masses), one can see as in Theorem \ref{t:LR} that the large deviation lower bound holds at $\mu$ and in fact that then $\mu\in \cH^{\sf A+B}$. 

\begin{remark}
If $\mu_{A_N}$ and $\mu_{B_N}$ are supported on $[-\fK, \fK]$, then deterministically the spectral measure of $A_N+UB_NU^*$ is supported on $[-2\fK, 2\fK]$. Therefore, for our study of the large deviation principle of the empirical measure of $A_N+UB_NU^*$,  we can restrict ourselves to the set of measures supported on $[-2\fK, 2\fK]$, which is compact in the weak topology. 

\end{remark}

The Harish-Chandra-Itzykson-Zuber integral formula, originated from the work of Harish-Chandra \cite{Harish}, and Itzykson and Zuber \cite{BIPZ,Itzyksonzuber},  exactly computes the spherical integral for $\beta=2$
\begin{align}\label{e:HCIZ}
\int e^{N\Tr(A_NUB_NU^{*})} \rd U
=\left(\prod_{i=1}^{N-1}i! \right)\frac{\det [e^{a_i b_j}]_{1\leq i,j\leq N}}{\prod_{1\leq i<j\leq N}(a_i-a_j)(b_i-b_j)}.
\end{align}
The determinantal structure on the right hand side of \eqref{e:HCIZ} resembles the formula for the Schur symmetric polynomials. In fact, we can use the Harish-Chandra-Itzykson-Zuber integral formula to give an integral representation of Schur symmetric polynomials. Before stating the formula, we need to introduce more notations. Given a partition $\bmla_N=(\la_1\geq \la_2\geq \cdots\geq \la_N)$, we encode it through the counting measure $m[\bmla_N]$ as
\begin{align}\label{e:defm}
m[\bmla_N]=\frac{1}{N}\sum_{i=1}^N \delta\left(\frac{\lambda_i +N-i}{N}\right).
\end{align}
To get the Schur symmetric polynomial $S_{\bmla_N}$ parametrized by the partition $\bmla_N$ from the Harish-Chandra-Itzykson-Zuber integral formula \eqref{e:HCIZ}, we take
\begin{align*}
D_N=\diag\left\{\frac{\la_1+N-1}{N}, \frac{\la_2+N-2}{N}, \cdots, \frac{\la_N}{N}\right\},
\quad Y_N=\diag\{y_1,y_2,\cdots, y_N\},
\end{align*}
and denote $e^{Y_N}=(e^{y_1}, e^{y_2}, \cdots, e^{y_N})$, then it follows
\begin{align}\begin{split}\label{e:schur}
&\phantom{{}={}}S_{\bmla_N}(e^{Y_N})
=\frac{\det [e^{y_i (\la_j+N-1)}]_{1\leq i,j\leq N}}{\prod_{1\leq i<j\leq N}(e^{y_i}-e^{y_j})}\\
&=\left(\prod_{i=1}^{N-1}i! \right)^{-1}
\frac{\prod_{1\leq i,j\leq N}(y_i-y_j)(\lambda_i-\lambda_j-i+j)}{\prod_{1\leq i,j\leq N}(e^{y_i}-e^{y_j}) }\int e^{N\Tr(Y_NUD_NU^{*})} \rd U.
\end{split}\end{align}
Let $\bmla_N\in \bY_N$ be a sequence of deterministic partitions, such that its counting measure $m[\bmla_N]$ as defined in \eqref{e:defm} converges weakly to $m_{\bmla}$, and the spectral measure $\mu_{Y_N}$ of $Y_N$ converges weakly towards $\mu_{Y}$, then \eqref{e:defI} implies the following asymptotics of Schur symmetric polynomials
\begin{align}\begin{split}\label{e:logS}
&\lim_{N\rightarrow\infty}\frac{1}{N^2}\log S_{\bmla_N}(e^{Y_N})
=J(\mu_Y, m_\bmla),\\
&J( \mu_Y,m_\bmla)= 2I(\mu_Y, m_\bmla)+\frac{1}{2}\int \log|x-y|\rd m_\bmla(x)\rd m_{\bmla}(y)\\
&\phantom{J(\mu_Y,m_\bmla)=}-\frac{1}{2}\int\int \log \left(\frac{e^x-e^y}{x-y}\right)\rd \mu_Y(x)\rd \mu_Y(y){ +\frac{3}{4}}.
\end{split}\end{align}

\begin{theorem}\label{t:Kostka}
Let $\bmla_N\in \bY_N$ be a sequence of deterministic partitions, such that its counting measure $m[\bmla_N]$ as defined in \eqref{e:defm}, converges weakly towards $m_{\bmla}$. Assume there exists a constant $\fK>0$, such that $\supp m[\bmla_N]\subset [0,\fK]$ for all $N\in\mathbb N$.
\begin{enumerate}
\item Let $\mu\in \cM^{\rm b}([0,\fK])$, $\nu\in \cM$ and set
\begin{align}\label{e:defHK}
H^{\mathsf{K}}_\mu(\nu):=\int (T_\mu(x)-x) T_\nu(x) \rd x-J(\nu, m_\bmla)\,,
\end{align}
where the functional $J(\cdot, \cdot)$ has been defined in \eqref{e:logS} and $T_\mu, T_\nu$ are as defined in \eqref{e:defT}.
The functional $\mathcal I^{\mathsf K}(\cdot)$
 \begin{align} \label{e:defIK}
 \mathcal I^{\mathsf K}(\mu):=\sup_{\nu\in\mathcal M}H^{\mathsf K}_{\mu}(\nu),
 \end{align}
is lower semicontinuous on $\cM^{\rm b}([0, \fK])$ and achieves its maximum only at the uniform measure $\unif[\int x\rd m_\bmla-1/2, \int x\rd m_\bmla+1/2]$. $\mathcal I^{\mathsf K}(\mu)=+\infty$ unless $\int x\rd\mu=\int x\rd m_{\bmla}$ and $\mu$ satisfies the following inequalities :
\begin{equation}\label{limSHK} \int_y^1 (T_{\mu}(x)-T_{m_\bmla}(x))\rd x\leq 0\quad\text{for all }y\in [0,1]\,.\end{equation}

\item
The Kostka numbers $K_{\bmla_N\bmeta_N}$ in \eqref{e:defKostka} satisfy, for any $\mu\in \cM^{\rm b}([0,\fK])$,
\begin{align}\label{e:Kostka}
\lim_{\delta\rightarrow 0}\limsup_{N\rightarrow \infty}\frac{1}{N^2}\log \sup_{m[\bmeta_N]\in \bB_\delta(\mu)} K_{\bmla_N\bmeta_N}=\lim_{\delta\rightarrow 0}\liminf_{N\rightarrow \infty}\frac{1}{N^2}\log \sup_{m[\bmeta_N]\in \bB_\delta(\mu)} K_{\bmla_N\bmeta_N}
=-\cI^{\mathsf{K}}(\mu),
\end{align}
where $ \bB_{\delta}(\mu)$ is the ball $\{\nu\in \cM^{\rm b}([0,\fK]): \rd(\nu,\mu)<\delta\}$.
\end{enumerate}

\end{theorem}
\begin{remark}
If $m[\bmla_N]$ are supported on $[0, \fK]$, and  $K_{\bmla_N, \bmeta_N}\neq 0$, then deterministically
$m[\bmeta_N]$ is supported on $[0,\fK]$.  Moreover, from our construction of $m[\bmeta_N]$ as in \eqref{e:defm}, the limit of $m[\bmeta_N]$ necessarily has a density bounded by $1$. Therefore in Theorem \ref{t:LR},  we have restricted ourselves in the set of measures $\cM^{\rm b}([0,\fK])$: supported on $[0, \fK]$ with density bounded by $1$. 
\end{remark}

We can also derive the following asymptotic formulas for the Littlewood-Richardson coefficients.
\begin{theorem}\label{t:LR}
Let $\bmla_N,\bmeta_N\in \bY_N$ be two sequences of deterministic partitions, such that their counting measures $m[\bmla_N], m[\bmeta_N]$, as defined in \eqref{e:defm}, converge weakly towards $m_{\bmla}, m_{\bmeta}$ respectively. Assume there exists a constant $\fK>0$, such that $\supp m[\bmla_N], \supp m[\bmeta_N]\subset [0,\fK]$. Then

\begin{enumerate}
\item
Let  $\mu\in \cM^{\rm b}([0,2\fK])$, $\nu\in \mathcal M$ and set
$$H^{\mathsf{LR}}_\mu(\nu)=J(\nu, \mu)-J( \nu,m_\bmla)-J( \nu,m_\bmeta)\,,$$
where the functional $J(\cdot, \cdot)$ has been defined in \eqref{e:logS} .
The functional $\mathcal I^{\mathsf{LR}}(\cdot)$
\begin{align}\label{e:defILR}
\mathcal I^{\mathsf{LR}}(\mu):=\sup_{\nu\in\mathcal M}H^{\mathsf{LR}}_{\mu}(\nu)
\end{align}
 is lower semicontinuous on $\cM^{\rm b}([0,2\fK])$. $\mathcal I^{\mathsf{LR}}(\mu)=+\infty$ 
unless $\int x\rd \mu=\int x\rd m_{\bmla}+\int x\rd m_{\bmeta}$ and the following inequalities hold:
\begin{align*}
\int_y^1  T_{m_\bmla}(x)\rd x
+\int_y^1 T_{m_\bmeta}(x)\rd x
\geq 
\int_y^1 T_{\mu}(x)\rd x,\quad \text{for all $y\in [0,1]$}.
\end{align*}

\item
Let  $ \bB_{\delta}(\mu)$ denote the ball $\{\nu\in \cM^{\rm b}([0,2\fK]): \rd(\nu,\mu)<\delta\}$.
The Littlewood-Richardson coefficients $c_{\bmla_N\bmeta_N}^{\bmkappa_N}$ \eqref{e:defLR} are asymptotically bounded above as follows
\begin{equation}\label{e:LR}
\limsup_{\delta\rightarrow 0}\limsup_{N\rightarrow \infty}\frac{1}{N^2}\log \sup_{m[\bmkappa_N]\in \bB_\delta(\mu)} c_{\bmla_N\bmeta_N}^{\bmkappa_N}
\leq -\cI^{\mathsf{LR}}(\mu),\end{equation}
and the complementary lower bound holds if the $\sup$ in \eqref{e:defILR} is achieved at a probability measure $\nu$ which is compactly supported
and
all components of $\supp\nu$ are infinite sets.
\end{enumerate}
\end{theorem}
We would have liked to describe the set of measures at which we can get the large deviation lower bound in more explicit way, similar to the set $\cH^{\sf A+B}$ of Theorem \ref{t:AUBU}. We believe that this could be done with an appropriate  quantized version for the freeness with amalgamation which we hope to investigate in future work.

\begin{remark}
If $m[\bmla_N]$ and $m[\bmeta_N]$ are supported on $[0, \fK]$, and  $c^{\bmkappa_N}_{\bmla_N, \bmeta_N}\neq 0$, then deterministically
$m[\bmkappa_N]$ is supported on $[0,2\fK]$.  Moreover, from our construction of $m[\bmeta_N]$ as in \eqref{e:defm}, the limit of $m[\bmkappa_N]$ necessarily has a density bounded by $1$. Therefore in Theorem \ref{t:LR},  we have restricted ourselves in the set of measures supported on $[0, 2\fK]$ with density bounded by $1$. 
\end{remark}

\section{Spherical Integral}
\label{sec:sp}

In this section we study the spherical integral and the limit function $I(\mu_A, \mu_B)$ as in \eqref{e:defI}. In Section \ref{s:prel}, we collect some estimates of the spherical integral and its limit $I(\mu_A, \mu_B)$ from \cite{GZ3,GZei1add, guionnet2004large,MR2095566}, where it was shown that $I(\mu_A, \mu_B)$ is related to a variational problem. The results in \cite{GZ3,GZei1add} requires that one of $\{\mu_A, \mu_B\}$ is compactly supported and the other has bounded second moment and free energy. However, by a continuity argument, it is easy to see that $I(\mu_A, \mu_B)$ is well-defined when one of $\{\mu_A, \mu_B\}$ is compactly supported and another has bounded first moment.
In Section \ref{s:L1spherical}, we extend the results in \cite{GZ3} for the setting that one measure is compactly supported and the other has bounded first moment. We remark this is the largest possible set where $I(\mu_A, \mu_B)$ is well defined. In this setting we show that the solution of the variational principle converges to the free Brownian bridge. Moreover, we characterize the limiting joint law of $(A_N, UB_NU)$ under the tilted measure $d\mu_{A_{N},B_{N}}(U)$ defined in \eqref{e:lawU0}.

Using the joint law of $(A_N, UB_NU)$ under $\mu_{A_{N},B_{N}}$ as input, in Section \ref{s:derSI}, we compute the derivatives of the limit function $I(\cdot,\cdot)$.  In Section \ref{s:continuity}, we give a more precise description of the solutions of the variational problem characterizing $I(\cdot, \cdot)$, by transforming the equations for the solution into a Beltrami equation.

\subsection{Preliminaries}\label{s:prel}

For any probability measure $\mu\in \cM(\bR)$, we denote $\Sigma(\mu)$ the energy of its logarithmic potential, or  its non-commutative entropy, 
\begin{align*}
\Sigma(\mu)=\int\int \log |x-y|\rd \mu(x)\rd \mu(y).
\end{align*}
We recall the following Theorem from \cite{GZ3}, where it is  proven that the limit of the spherical integral exists, provided one measure has bounded $L^2$ moment and logarithmic potential, another measure is compactly supported.
\begin{theorem}[{\cite[Theorem 1.1]{GZ3}}]\label{t:exist}
Let $A_{N},B_{N}$ be two sequences of deterministic self-adjoint matrices, such that their spectral measures $\mu_{A_N}$ and $\mu_{B_N}$ converge weakly to $\mu_{A}$ and $\mu_{B}$ respectively. If there exists a constant $\fK>0$, such that $\supp \mu_{A_N}\subset [-\fK, \fK]$ and $ \mu_{B_N}(|x|^2)\leq \fK$, $\Sigma(\mu_B)\geq -\fK$, then the spherical integral converges
\begin{align}\label{e:Sph}
\lim_{N\rightarrow\infty}I_{N}(A_{N},B_{N}):=\lim_{N\rightarrow\infty}\frac{1}{\beta N^{2}}\int e^{\frac{\beta N}{2}\Tr(A_NUB_NU^{*})} \rd U=I(\mu_{A},\mu_{B}),
 \end{align}
where $U$ follows the Haar measure on the unitary group (resp. orthogonal group) when $\beta=2$ (resp. $\beta=1$). 
\end{theorem}

The proof of Theorem \ref{t:exist} is intimately related with the following large deviation principle based on the Hermitian (resp. symmetric) matrix Brownian motion $H_N(t)$. It is the process of $N\times N$ matrices filled with independent Brownian motion entries above the diagonal, 
with $(i,j)$-th entry given by
\begin{align}\label{e:defBDM}
(H_N(t))_{ij}=
\left\{
\begin{array}{ll}
\frac{1}{\sqrt \beta N}(B_{i,j}(t)+\sqrt{-1}(\beta-1)\tilde B_{i,j}(t)), & \text{if } i<j,\\
\frac{\sqrt 2}{\sqrt \beta N}B_{i,i}(t), & \text{if } i=j,
\end{array}
\right.
\end{align}
where $B_{i,j}(t)$ are independent standard Brownian motions.

\begin{theorem}[{\cite[Theorem 3.2 and Theorem 3.3]{GZei1add}}]\label{t:NBMldp}
Let $A_{N}$ be a sequence of deterministic diagonal matrices, with diagonal entries $a_1\leq a_2\leq \cdots\leq a_N$ 
whose spectral measures converge towards $\mu_{A}$. Assume that there exist a constant $\fK>0$ and $\varepsilon>0$ such that
$$\mu_{A_N}(|x|^2)\leq \fK,\quad \mu_{A}(|x|^{5+\varepsilon})\leq \fK.$$
Let $H_{N}(t)$ be a Hermitian (resp. Symmetric) matrix Brownian motion (recall from \eqref{e:defBDM}). Let $(\lambda_{i}(t))_{1\leq i\leq N}$ be the eigenvalues of the  self-adjoint matrix
\begin{align}\label{e:defX}
X_{N}(t)=A_{N}+H_N(t), \quad t\in [0,1],
\end{align}
and denote by $(\mu^{N}_{t})_{t\in [0,1]}$  the empirical measure of these eigenvalues.
Then,  the law of $(\mu^{N}_{t})_{t\in [0,1]}$, seen as a continuous process with values  in the space $\mathcal P(\mathbb R)$ of probability measures, satisfies a large deviation principle with speed $N^{2}$ and good rate function which is infinite if $\mu_{0}\neq \mu_{A}$ and otherwise given by $\beta S_{\mu_{A}}$,
\begin{align}\label{e:defSA}
S_{\mu_{A}}(\mu_{.})=\sup_{f\in C^{2,1}(\mathbb R, [0,1])}\left(S^{0,1}(f,\mu_{.})-\frac{1}{2}\langle f,f\rangle_{\mu}\right),
\end{align}
where $C^{2,1}(\bR, [0,1])$ is the set of functions $f(x,t)$ which is twice differentiable in $x$ and differentiable in $t$, $\langle f,g\rangle_{\mu}=\int_0^{1}\int \partial_{x}f(s,x)\partial_{x}g(s,x)\rd\mu_{s}(x) \rd s$ and 
$$S^{0,1}(f,\nu)=\int f(1,x)\rd\nu_{1}(x)-\int f(0,x)\rd\mu_{A}(x)-\int_{0}^{1}\int \partial_{t} f(t,x)\rd\nu_{t}(x) \rd t\qquad\qquad$$
$$\qquad \qquad \qquad\qquad\qquad -\frac{1}{2}\int_{0}^{1}\int \frac{\partial_{x}f(s,x)-\partial_{x}f(s,y)}{x-y} \rd\nu_{s}(x)\rd \nu_{s}(y)\rd s\,.$$
As a consequence, the law of $\mu^{N}_{1}$ satisfies a large deviation principle in the scale $N^{2}$ with the rate function
\begin{align}\label{e:infJ}
J_{\beta}(\mu_{A},\mu):=\beta \inf\{S_{\mu_{A}}(\mu_{.}):\mu_{1}=\mu\}\,.
\end{align}
\end{theorem}

From formula \eqref{t:NBMldp}, if we condition on that the eigenvalues of $X_N(1)$ are given by $B_N$, i.e. $X_N(1)=UB_NU^*$, then the law for the eigenvectors $U$ of $X_N(1)$ is given by the integrand of the spherical integral:
\begin{align}\label{e:lawU}
d\mu_{A_{N},B_{N}}(U)=\frac{1}{Z_N}e^{\frac{\beta N}{2} \Tr(B_NUA_NU^{*})}\rd U.
\end{align}
Theorem \ref{t:exist} can then be deduced from Theorem \ref{t:NBMldp} by noticing that the law for the eigenvalues of $X_N(1)$ is given by
\begin{align}\label{e:lawX1}
\frac{e^{-\frac{\beta N}{4}\Tr(A_N^{2}) }}{Z_{N}} \prod_{i<j}|x_i-x_j|^\beta e^{-\frac{\beta N}{4}\Tr X_N(1)^{2}}\prod_{i=1}^N\rd x_i
\int e^{\frac{\beta N}{2} \Tr(X_N(1)UA_NU^{*})}\rd U,
\end{align}
where $x_1, x_2, \cdots, x_N$ are the eigenvalues of $X_N(1)$. Indeed, this formula \eqref{e:lawX1} asymptotically yields 
\begin{equation}\label{formulaJ}
J_{\beta}(\mu_{A},\mu)=\frac{\beta}{4}\left(\int x^{2}d\mu_{A}(x)+\int x^{2}d\mu(x)\right)-\frac{\beta}{2}\Sigma(\mu)-\beta I(\mu_{A},\mu)+const.\,,
\end{equation}
where $const.$ is the  finite constant coming from the partition function $Z_{N}$.

One can directly study the optimizing problem \eqref{e:infJ}. The optimizer is characterized via solutions of an Euler equation with negative pressure.
\begin{theorem}[{\cite[Theorem]{Gurig}}]\label{theoCMP}
We assume that there exists a constant $\fK>0$, such that $\supp\mu_A\subset[-\fK, \fK]$,  $\mu_B(|x|^2)\leq \fK$ and $\Sigma(\mu_B)\geq -\fK$, $\Sigma(\mu_A)\geq -\fK$. 
Then $I(\mu_A, \mu_B)$ is given by
\begin{align}\label{e:Iexp}
I(\mu_A,\mu_B)=-\frac{1}{4}\inf S(u,\rho)-\frac{1}{4}\left(\Sigma(\mu_A)+\Sigma(\mu_B)\right)+\frac{1}{4}\left(\int x^2 \rd \mu_A(x)+\int x^2 \rd \mu_B(x)\right)-\text{const.}
\end{align}
where 
\begin{align}\label{e:funcS}
S(u,\rho)=\int_0^1 \int_\bR \left(\frac{\pi^2}{3}\rho_t^3 +u_t^2 \rho_t \right)\rd x \rd t;
\end{align}
the $\inf$ is taken over all the pairs $(u_t,\rho_t)$ such that $\del_t \rho_t+\del_x(\rho_tu_t)=0$ in the sense of distributions, $\rho_t\geq 0$ almost surely w.r.t. the Lebesgue measure, $\int \rho_t \rd x=1$, and with initial and terminal data for $\rho$ given by 
\begin{align*}
\lim_{t\rightarrow 0+}\rho_t(x)\rd x=\mu_A,\quad \lim_{t\rightarrow 1-}\rho_t(x)\rd x=\mu_B,
\end{align*}
where convergence holds in the weak sense.

$S$ is strictly convex in $(\rho,u\rho)$. The infimum in \eqref{e:Iexp} is reached at
a unique probability measure-valued path
$\rho_t^*\rd x\in\Ca([0,1],M_1(\R))$
such that for $t\in (0,1)$, $\rho_t^* \rd x$ is absolutely continuous with respect to Lebesgue measure. 
%
The pair $(\rho^*,u^*)$ satisfies the Euler equation for isentropic 
flow described, for $t\in (0,1)$, by the equations
\begin{align}
\partial_t \rho^*_t(x)&=-\partial_x( \rho^*_t(x)u^*_t(x))
\label{eulpr10}\\
\partial_t(\rho^*_t(x)u^*_t(x))&=-\partial_x\left(\rho^*_t(x)
u^*_t(x)^2-{\pi^2\over 3}\rho^*_t(x)^3\right),\label{eulpr20}
\end{align}
in the sense of distributions:
for all $\varphi\in \Ca_c^{\infty,\infty}({\R}\ts [0,1])$,
$$\int_0^1\int\partial_t \varphi(t,x)\rho^*_t(x)\rd x 
\rd t +\int_0^1\int\partial_x  \varphi(t,x) u^*_t(x) \rho^*_t(x)
\rd x\rd t=0,$$
and, for 
any $\varphi\in\Ca^{\infty,\infty}_c(\Omega)$ with
\begin{align}\label{Omega}
& \Omega:=\{(x,t)\in\mathbb R\times(0,1)\colon\rho^*_t(x)>0\},\\
&\label{distf}\int \left(u^*_t(x)\partial_t \varphi(x,t)+\left( u^*_t(x)^2
-\frac{\pi^2}{3} \rho^*_t(x)^2\right) \partial_x \varphi(x,t)\right)\rho^*_t(x) \rd x \rd t=0.
\end{align}
The infimum $(\rho^*,u^*)$ are smooth in the interior of $\Omega$,
which guarantees that (\ref{eulpr10}) and (\ref{eulpr20})
hold everywhere in the interior of $\Omega$. 
Moreover, $\Omega$ is bounded in $\R\ts [0,1]$.
\end{theorem}

Moreover,
Theorem \ref{theoCMP} and the large deviation principle of \cite[Theorem 3.2 and Property 2.2]{GZei1add} implies that the spectral measure of $\{X_N(t), t\in [0,1]\}$ (as in \eqref{e:defX})  conditioned at time $1$ to be equal to $\mu_{B_N}$  converges towards the minimizer $\{\rho_t^*\}_{0\leq t\leq 1}$ of  $S$ in \eqref{e:funcS}.

\begin{theorem}[{\cite[Theorem 3.3]{GZei1add}}]\label{convrho}
Let $A_{N},B_{N}$ be two sequences of deterministic self-adjoint matrices, such that their spectral measures $\mu_{A_N}$ and $\mu_{B_N}$ converge weakly to $\mu_{A}$ and $\mu_{B}$ respectively. If there exists a constant $\fK>0$, such that $\supp \mu_{A_N}\subset [-\fK, \fK]$ and $ \mu_{B_N}(|x|^2)\leq \fK$, $\Sigma(\mu_B)\geq -\fK$, then 
the spectral measure of $\{X_N(t), t\in [0,1]\}$ (as in \eqref{e:defX})  conditioned at time $1$ to be equal to $\mu_{B_N}$  converges weakly almost surely towards the minimizer $\{\rho_t^*\}_{0\leq t\leq 1}$ of \eqref{e:Iexp}.
%
%
\end{theorem}
Strictly speaking the two previous theorems imply this theorem under the additional hypothesis that $\Sigma(\mu_A)$ is finite. However we can define $\rho^*_t$ even when $\Sigma(\mu_A)$  is infinite by continuity, see Lemma \ref{c:densitye}, and consider the large deviation principle from time $\varepsilon$ conditioned to equal $\rho^*_{\varepsilon}$ at the final time. Since the latest has finite entropy, all the previous arguments apply. Theorem \ref{convrho}, or \cite[Theorem 2.6]{Gurig}, also shows that $\{\rho^*_t\}_{t\in [0,1]}$ has the distribution of a free Brownian bridge, defined 
thanks to the notion of non-commutative joint law from free probability. 

The next theorem involves in an essential way some notions from free probability, which we now introduce. An abstract noncommutative probability space is a pair
$(\mathcal A,\tau)$, where $\mathcal A$ is a unital algebra over the field of complex numbers $\mathbb C$ and $\tau\colon\mathcal A\to\mathbb C$ is a $\mathbb C$-linear
functional which maps the unit of $\mathcal A$ to the complex number one. In our paper, we assume in addition that $\mathcal A$ is a von Neumann algebra (see \cite{Takesaki01}) 
and $\tau$ is {\em normal, faithful, tracial, and positive}, that is, it belongs to the predual of $\mathcal A$ (normality), 
it satisfies $\tau(x^*x)\ge0$ for all $x\in\mathcal A$ (positivity), with equality if and only if $x=0$ (faithfulness), and $\tau(yx)=\tau(xy)$ for all $x,y\in\mathcal A$ (traciality).
This notion does generalize the classical notion of probability space: the algebra of essentially bounded measurable functions on a classical probability space is a 
commutative von Neumann algebra and the integration with respect to the probability measure is a unit-preserving positive linear functional. By analogy with the classical
terminology, elements of $\mathcal A$ are called random variables, or sometimes noncommutative random variables (for this, and much of the material related to free probability,
we refer to \cite{VDN92}). 

As in classical analysis, one defines $L^p$ spaces with respect to $\tau$: for $1\le p<\infty$, the Banach space $L^p(\mathcal A,\tau)$ is the completion of $\mathcal A$ with respect
to the norm $\|x\|_p:=[\tau((x^*x)^{p/2})]^{1/p}$. Then it makes sense to write $L^\infty(\mathcal A,\tau)=\mathcal A$, and to let $L^\infty(S,\tau)$ (or just $L^\infty(S)$ when 
there is no risk of confusion) be the von Neumann subalgebra of $\mathcal A$ generated by the set $S$ (note that $S$ may contain unbounded operators: for instance, the von 
Neumann algebra generated by an unbounded selfadjoint operator equals the von Neumann algebra generated by its - clearly bounded - spectral projections). Another notation that we 
will sometimes use is $S''$, the double commutant of the set $S$ in the space of bounded operators on $L^2(\mathcal A,\tau)$; this makes use of the equivalent definition of a von 
Neumann algebra \cite[Definition 3.2 of Chapter II]{Takesaki01}.

While the distribution of a single selfdjoint random variable with respect to $\tau$ is a classical object, the distribution of several random variables, even selfadjoint,
is usually not a classical probability measure. For bounded random variables, the distribution of a tuple $(x_1,\dots,x_k)\in\mathcal A^k$ is conveniently defined as the collection
of all moments $\tau(x_{\iota_1}x_{\iota_2}\cdots x_{\iota_n}),n\in\mathbb N,\iota_1,\iota_2,\dots,\iota_n\in\{1,\dots,k\}$ (in this article we are only concerned with 
selfadjoint or normal random variables). 

A noncommutative notion of independence was introduced by Voiculescu: subalgebras $\{\mathcal A_\iota\}_{\iota\in I}$ of $\mathcal A$ containing the unit of $\mathcal A$ are called 
{\em freely independent with respect to} $\tau$, or simply {\em free}, if $\tau(x_1x_2\cdots x_n)=0$ whenever $\tau(x_j)=0$, $x_j\in\mathcal A_{\iota_j}1\le j\le n$, $\iota_1\neq
\iota_2,\iota_2\neq\iota_3,\dots,\iota_{n-1}\neq\iota_n, n\in\mathbb N$. Subsets $S_\iota$, $\iota\in I$ (possibly singletons) are called free if the algebras $\mathcal A_\iota$
generated by $S_\iota$ together with the unit of $\mathcal A$ are free. Naturally, if $x=x^*,y=y^*$ are free in $\mathcal A$ and distributed according to $\mu_x$ and $\mu_y$
with respect to $\tau$, then the distribution of the sum $x+y$ with respect to $\tau$ depends only on $\mu_x$ and $\mu_y$ and is called the free additive convolution
of $\mu_x$ and $\mu_y$, denoted by $\mu_x\boxplus\mu_y$. 

While the above definitions apply to subsets consisting of {\em bounded} random variables (indeed, elements of $\mathcal A$ automatically have bounded spectrum), they can all be 
extended to unbounded random variables by considering operators affiliated 
to $\mathcal A$ and recalling that the existence of $\tau$ guarantees that such operators do form an algebra. 
This result is due to Bercovici and Voiculescu \cite{BV93}, to which we refer for details. A more general notion of joint distribution of unbounded selfadjoint random variables 
(that is, beyond the sum of two unbounded selfadjoints) can be defined in various ways: the distribution of selfadjoint random variables $(x_1,\dots,x_k)$ affiliated to 
$\mathcal A$ can be defined as the collection of the classical distributions of $P(x_1,\dots,x_k)$ with respect to $\tau$ as $P$ runs through the set of all selfadjoint polynomials in $k$ 
non-commuting indeterminates, or as the collection of the noncommutative distributions of $(f_1(x_1),\dots,f_k(x_k))$ for all bounded measurable real-valued $f_j$ etc. For our 
purposes, the version described in \eqref{e:defcF} will suffice.

\begin{theorem}[{\cite[Theorem 2.6]{Gurig}}]\label{t:FBB}
Assume $\mu_{A}, \mu_{B}$ are compactly supported. Then the space of free Brownian path distributions $FBB(\mu_{A},\mu_{B})$ given by
$$(1-t)\mathsf a+t\mathsf b+\sqrt{t(1-t)} \mathsf s, \quad t\in [0,1],$$
is closed, 
where $(\mathsf a,\mathsf b)$ are free from  the semi-circle law $\mathsf s$, with joint distribution such that the distribution of $\mathsf a$ and $ \sf b$ are $\mu_{A}$ and $\mu_B$  respectively. Then, $\{\rho^*_t\}_{t\in [0,1]}$ belongs to $FBB(\mu_{A},\mu_{B})$.
\end{theorem}
 Here, we assumed $\mu_A,\mu_B$ compactly supported to rely on the standard definition of non-commutative laws based on evaluation at polynomial test functions. This however can be generalized as we will see in the next section. 

As in classical probability, one can define a notion of conditional expectation with respect to $\tau$ onto a subalgebra of $\mathcal A$. If $\mathcal B$ is a von Neumann subalgebra 
of $\mathcal A$, then the restriction of the projection from $L^2(\mathcal A,\tau)$ onto $L^2(\mathcal B,\tau|_{\mathcal B})$ to the von Neumann algebra $\mathcal A$,
which we denote by $\tau(\cdot|\mathcal B)$, has the remarkable properties that (i) it takes values in $\mathcal B$, (ii) it equals the identity when restricted to $\mathcal B$, and
(iii) it is a $\mathcal B$-$\mathcal B$ bimodule map, that is, $\tau(bxb'|\mathcal B)=b\tau(x|\mathcal B)b'$ for all $x\in\mathcal A,b,b'\in\mathcal B$. We will often write
$\tau(\cdot|S)$ for the conditional expectation onto the von Neumann subalgebra of $\mathcal A$ generated by the set $S$ (including when $S$ is a singleton). For details,
the reader is referred to \cite[Proposition 2.36, Chapter V]{Takesaki01}.

Theorem \ref{t:FBB} gives several a priori properties of the minimizers of $S_{\mu_{A}}$, for instance they are absolutely continuous and with bounded density for any time $t\in (0,1)$.
Putting together the characterizations of Theorems  \ref{theoCMP} and \ref{t:FBB}, we have the  following:

\begin{proposition}\label{propbur}
We assume that  the probability measures $\mu_A, \mu_B$ are compactly supported. The pair $(\rho^*, u^*)$ is the unique solution of the variational problem \eqref{e:Iexp}, and $f(t,x)=u_t^*(x)+\pi \ri \rho_t^*(x)$ satisfies the complex Burgers equation
\begin{equation}\label{burg}
\partial_t f(t,x)+f(t,x)\partial_x f(t,x)=0.
\end{equation} 
Moreover, 
\begin{enumerate}
\item $(t,x)\ra f(t,x)$ is real analytic in each component in the interior of $\Omega$ as defined in \eqref{Omega}.
\item $ \Im[f(t,x)]/\pi$ converges weakly
as $t$ goes to $0$ or $1$ to $\mu_A$ and $\mu_B$
respectively.
\item For any $g\in C^1$,
\begin{align}\label{e:lim}\lim_{t\ra 0+}
\int g(x){\Re}[f(t,x)] \rho_t^*(x) \rd x
=\int g(x) {\Re}[f(0,x)] \rd \mu_A,
\end{align}
and the same as time $t\rightarrow 1-$.
\item There exists a finite constant $\fK>0$ (depending on the support of $\mu_A, \mu_B$) such that
for all $(x,t)$ in the interior of $\Omega$,
$$|f(t,x)|\leq \frac{\fK}{\sqrt{t(1-t)}}.$$
\item For all $t\in (0,1)$,
all $x\in\Omega$ such that $(t,x_0)\in \partial\Omega$,
$$\rho_t^*(x)\leq \left({3\over 4\pi^3 t^2(1-t)^2}\right)^{1\over3}
(x-x_0)^{1\over 3}.$$
\item There exist two operators $\mathsf a, \sf b$
in a non-commutative probability
space $(\cal A,\tau)$ 
with marginal distribution $(\mu_A, \mu_B)$ 
so that $t\mathsf a+(1-t)\mathsf b +\sqrt{t(1-t)} \mathsf s$ has the law of $\rho^{*}_t(x){\rm d}x$
where $\mathsf s$ is a semicircular variable free from $(\mathsf a,\mathsf b)$.
\item  For all $t\in [0,1]$, let $\{\sf s_t\}_{0\leq t\leq 1}$ be a non-commutative Brownian motion independent of $\sf a, \sf b$, and 
\begin{align*}
\rd {\sf x}_t=\rd {\sf s}_t+\frac{{\sf b}-{\sf x}_t}{1-t}\rd t, \quad {\sf x}_0=\sf a, \quad {\sf x}_1=\sf b,
\end{align*}
then $\sf x_t$ has the law of $(1-t)\sf a+t\sf b+\sqrt{t(1-t)}\sf s$ given by $\rho_t^*$, and 
\begin{equation}\label{e:condut} u_t^*=\frac{1}{t-1}\tau(\mathsf x_t-\mathsf b|\mathsf x_t) +H\rho^*_t,\qquad \rho_{t}^{*}(x) \rd x \quad a.s. \end{equation}
where $\tau(\cdot|\cdot)$ is the free conditional expectation, and $H\rho(x)=\mbox{p.v.}\int (x-y)^{-1}\rho(y)\rd y$ is the Hilbert transform of $\rho$.
\end{enumerate}
\end{proposition}
\begin{proof} Most of the proof is already contained 
in \cite[Corollary 2.8]{Gurig} and lies in the representation
of the solution in terms of a free Brownian
bridge stated in the last two points above, i.e. Item 6,7, see Theorem \ref{t:FBB}; namely, it is shown that there exists two non-commutative
variables $\sf a, \sf b$ with marginals distributions
$\mu_A,\mu_B$ so that $\rho_t^*(x)\rd x$ is the law of
$(1-t)\mathsf a+t\mathsf b+\sqrt{t(1-t)}
\mathsf s$, with $\mathsf s$ a semi-circular law free with $\mathsf a,\mathsf b$.
Items 1,4,5 are then direct consequences of 
\cite{Biane}. By the definition of the variational problem \eqref{e:Iexp}, we have that $\rho_t^*(x)dx$ 
converges weakly towards $\mu_A$ as $t$ goes to $0$.
Finally, by (2.19) in \cite{Gurig}, $u^{*}$ is given by \eqref{e:condut}.
But
$$\int  \frac{g(x)}{t-1}\tau({\sf x}_t-{\sf b}|{\sf x}_t) \rho^*_t(x)\rd x= \tau( g({\sf x}_t) ({\sf x}_t-{\sf b}))$$
converges as $t$ goes to $0$ by continuity of
$g$ and ${\sf x}_t$
whereas since $H\rho_t^*\in L^2(\rho_t^*)$ (as $\rho^*_t\in L^3(\rd x)$)
$$\int g(x)H\rho^*_t(x)\rho_t^*(x)\rd x=
\frac{1}{2}\int\int \frac{g(x)-g(y)}{x-y} \rho_t^*(x)\rho_t^*(y)\rd x\rd y,$$
with $(x,y)\ra (g(x)-g(y))/(x-y)$ continuous when $g$ is $C^1$,
converges as $t$ goes to $0$ or $1$ by weak
continuity of $\rho^*_t(x)dx$. The claim \eqref{e:lim} follows from the above discussion.
\end{proof}

We recall that given a probability measure $\mu$, we denote by $T_\mu\colon(0,1)\mapsto (-\infty, \infty)$   the right continuous increasing function, such that $\mu$ is the push-forward of the uniform distribution on $(0,1)$ by $T_\mu$.
Instead of viewing $I(\mu_A, \mu_B)$  a function of $\mu_A, \mu_B$, we can think about  it as a function of $T_{\mu_A}, T_{\mu_B}$, i.e. $I(T_{\mu_A}, T_{\mu_B})\deq I(\mu_A, \mu_B)$. Since the spherical integral 
\begin{align*}
\lim_{N\rightarrow \infty}\frac{1}{\beta N^2}\log\int e^{\frac{\beta N}{2}\Tr A_NUB_NU^*}\rd U= I(T_{\mu_A}, T_{\mu_B})
\end{align*}
is convex in both $A_N$ and $B_N$, the limit is also convex in $T_{\mu_A}$ and $T_{\mu_B}$.

\subsection{The spherical integral for $L^{1}$ distributions and convergence of the non-commutative law}\label{s:L1spherical}

Later in the article, we will need to extend the  limit of the spherical integral \eqref{e:defI} to the setting where one measure has bounded support and the other measure has bounded first moment.
The following proposition states that the limit function $I(\cdot, \cdot)$ and other quantities appearing in our main theorems are continuous with respect to the Wasserstein distance $\rd_W(\cdot, \cdot)$ as defined in \eqref{e:wd}.

\begin{proposition}\label{p:continuity}
We assume that the probability measures $\mu,\nu$ and $\mu', \nu'$ satisfy that $\supp \mu, \supp\mu'\subset [-\fK,\fK]$, and $\nu(|x|), \nu'(|x|)\leq \fK$ for some constant $\fK>0$. Then there exists a finite constant $C_{\fK}$ so that for any  small $\delta>0$ such that $\rd_W(\mu, \mu')\leq\delta$ and $\rd_W(\nu, \nu')\leq\delta$,
\begin{align*}
\left|I( \nu,\mu)-I( \nu',\mu')\right|= C_\fK{\oo_\delta(1)},
\end{align*}
and 
\begin{align*}\left|\int T_{\mu}T_{\nu} \rd x-\int T_{\mu'}T_{\nu'} \rd x\right|
=C_\fK{\oo_\delta(1)},
\end{align*}
where we can take $C_\fK = (\fK +1)$ and $\oo_\delta(1)$ to be $\nu(|x|\bm1_{|x|\geq \delta^{-1/2}})+\nu'(|x|\bm1_{|x|\geq \delta^{-1/2}})+\delta^{1/2}$.

\end{proposition}

The set of probability measures with compact support are dense in the set of probability measures with bounded $L^1$ norm with respect to the Wasserstein distance. 
Using Proposition \ref{p:continuity}, we can extend Theorems \ref{t:exist} and  \ref{theoCMP} to measures $\nu$ and $\mu$ such that  $\nu(|x|)\leq \fK$ and $\supp \mu\subset [-\fK,\fK]$.
In this case, Proposition \ref{p:continuity} also implies that $I(\nu,\mu)$ is well defined and continuous with respect to the Wasserstein distance. The proof of Proposition \ref{p:continuity} involves some straightforward estimates of the spherical integral, we postpone it to the Appendix \ref{a:proofL1}. It is then straightforward to deduce that

\begin{corollary}\label{t:convl1}
Let $A_{N},B_{N}$ be two sequences of deterministic self-adjoint matrices, such that their spectral measures $\mu_{A_N}$ and $\mu_{B_N}$ converge in Wasserstein distance \eqref{e:wd} towards $\mu_{A}$ and $\mu_{B}$ respectively. We assume that  there exists a constant $\fK>0$, such that $ \mu_{A_N}(|x|)\leq \fK$ and $\supp \mu_{B_N}\subset [-\fK, \fK]$. 
 Then, the following limit exists :
\begin{align}\label{e:L1limit}
\lim_{N\rightarrow \infty}\frac{1}{\beta N^{2}}\log\int e^{\frac{\beta N}{2}\Tr(A_NUB_NU^{*})}\rd U=I(\mu_{A},\mu_{B}).
\end{align}
\end{corollary}

Theorem \ref{theoCMP}  gives a quite complicated formula for $I$. However, we can obtain asymptotic limits which are much easier to handle based on the following propositions. The estimates will be used to study the large deviation rate functions. The proofs of Propositions \ref{p:spbound} and \ref{p:spbound2} involve some straightforward estimates of the spherical integral, we postpone them to the Appendix \ref{a:proofL1}.

\begin{proposition}\label{p:spbound}
We assume that the probability measures $\nu, \mu$ satisfies that $\nu(|x|)<\infty$ and $\supp \mu \subset [-\fK,\fK]$  for some constant $\fK>0$. Then for any small $\varepsilon>0$, there exists a constant $C(\varepsilon)>0$ such that
\begin{align}\label{e:spbound}
\frac{1}{2}\int T_{\nu}T_{\mu}\rd x-\OO(\varepsilon)\nu(|x|)-C(\varepsilon)
\leq I( \nu,\mu)\leq\frac{1}{2}\int T_{\nu}T_{\mu}\rd x.
\end{align}
Here, one can take $\OO(\varepsilon)=5\varepsilon\fK $ and $C(\varepsilon)$ depending only on $\varepsilon$.
\end{proposition}

As a consequence, we deduce that if $L\#\nu$ is the pushforward of $\nu$ by the homothety of factor $L$: $\int f(Lx)\rd\nu(x)=\int f(x)\rd L\#\nu(x)$, then
$$\lim_{L\ra\infty}\frac{1}{L} I( L\#\nu,\mu)=\frac{1}{2}\int T_{\nu}T_{\mu}\rd x\,.$$

\begin{proposition}\label{p:spbound2}
We assume that the probability measures $\nu, \mu$ satisfies that $\nu(|x|)\leq \fK$ and $\supp \mu \subset [-\fK,\fK]$  for some constant $\fK>0$. Then for any small $\varepsilon>0$, it holds
\begin{align}\label{e:spbound2}
I( \nu,\mu)=I(\nu^\varepsilon, \mu)+\frac{1}{2}\int_{|T_\nu|> 1/\varepsilon} T_{\nu}T_{\mu}\rd x
+C_\fK\oo_\varepsilon(1),
\end{align}
where $\nu^\varepsilon$ is the restriction of $\nu$ on the interval $|x|\leq 1/\varepsilon$, i.e. $\nu^\varepsilon=\nu\bm1(|x|\leq 1/\varepsilon)+\delta_0\int_{|x|> 1/\varepsilon}\rd \nu$, and the implicit error term depends only on $\varepsilon$ and $\fK$.
\end{proposition}

 Later on, we shall need to differentiate the map $\mu_A\mapsto  I(\mu_{A},\mu_{B})$.  One of our problems is that our formulas, see e.g. \eqref{e:Iexp}, are ill defined when the matrices have unbounded variance or non-commutative entropy. These terms have to cancel with $\inf S(u,\rho)$ (as in \eqref{e:funcS}) which needs to be infinite as well since $I$ stays bounded. By a monotonicity property of the nonintersecting Brownian bridges, for $\beta=2$, we can extend Theorem  \ref{convrho} to this setting. Since the rate function $I(\mu_A,\mu_B)$ is independent of $\beta$, some results also hold for $\beta=1$.

Theorem \ref{convrho} can be used  to prove the statement \eqref{e:tauF0}, namely the weak convergence of the non-commutative law of 
$(A_N,UB_NU^*)$ under $\mu_{A_N,B_N}$ as in \eqref{e:lawU}. We recall that such non-commutative probability distribution $\tau$ can be described by 
$
\tau( F(\mathsf a,\mathsf b)),
$
where $F(\mathsf a,\mathsf b)$ belongs to  the complex vector space of test functions $\mathcal F$ generated by noncommutative polynomials in the form 
\begin{align}\label{e:defcF}
\sfb^{n_0}\frac{1}{z_1-\mathsf a}\sfb^{n_1} \frac{1}{z_2-\mathsf a}\sfb^{n_2} \frac{1}{z_k-\mathsf a} \sfb^{n_k},
\end{align}
where $k$ is any positive integer number and $\{z_j\}_{1\leq j\leq k}$ belong to $\mathbb C\backslash \mathbb R$ whereas $(n_i)_{0\leq i\leq k}$ are non-negative  integer numbers (here we use resolvents instead of polynomials because $\mathsf a$ has a priori only its first moment finite). In this way, we can view the non-commutative probability distribution of $(\mathsf a,\mathsf  b)$ as an element of the dual of  $\mathcal F$.

\begin{theorem}\label{t:convl1}
Take $\beta=2$. Let $A_{N},B_{N}$ be two sequences of deterministic self-adjoint matrices, such that their spectral measures $\mu_{A_N}$ and $\mu_{B_N}$ converge in Wasserstein distance \eqref{e:wd} towards $\mu_{A}$ and $\mu_{B}$ respectively. We assume that  there exists a constant $\fK>0$, such that $ \mu_{A_N}(|x|)\leq \fK$ and $\supp \mu_{B_N}\subset [-\fK, \fK]$. 

\begin{enumerate}
\item Let $H_N(t),t\in [0,1]$ be the Hermitian (resp. symmetric) Brownian motion as in \eqref{e:defBDM}. The law of eigenvalues of $X_N(t)=A_N+H_N(t)$, conditioned to have the same eigenvalues as $B_N$ at time $ t=1$, converges towards the measure valued process $(\rho^{*}_{t})_{0\leq t\leq 1}$ such that 
for all $t\in (0,1)$, $\partial_{t}\rho^{*}_{t}+\partial_{x}(\rho^{*}_{t}u^{*}_{t})=0$ and they satisfy the Euler equation \eqref{eulpr20}. 
\item  For any $F\in\mathcal F$ as defined in \eqref{e:defcF}, the following limit exists:
\begin{align}\label{e:tauF}
\lim_{N\rightarrow \infty}\int \frac{1}{N}\Tr(F(A_N,UB_NU^{*}))\rd\mu_{A_N,B_N}(U)=:\tau_{\mu_A,\mu_B} (F(\mathsf a, \mathsf b))\,,
\end{align}
where $\rd\mu_{A_{N},B_{N}}$ was defined in \eqref{e:lawU}.  \end{enumerate}
Moreover  if $(\mu_{A^{(p)}},\mu_{B^{(p)}})_{p\geq 0}$ are sequences of probability measures such that there exists a finite constant $\fK$ such that   $\supp\mu_{A^{(p)}}\subset[-\fK, \fK]$ and  $\mu_{B^{(p)}}(|x|)\leq \fK$, and they converge weakly towards $\mu_{A},\mu_{B}$ respectively, then 
\begin{enumerate}
\setcounter{enumi}{2}
\item $\{\rho^*_t\}_{0\leq t\leq 1}$ is uniquely described as the weak  limit of  $\{\rho^{(p)}_t\}_{0\leq t\leq 1}$ which minimizes the strictly convex function  $S(u,\rho)$ as in \eqref{e:funcS} for regularized boundary data $(\mu_{A^{(p)}},\mu_{B^{(p)}})$.
\item $\tau_{\mu_{A^{(p)}},\mu_{B^{(p)}}}$ converges towards $\tau_{\mu_A,\mu_B}$ in the sense that $\tau_{\mu_{A^{(p)}},\mu_{B^{(p)}}}(F(\mathsf a, \mathsf b))$ converges towards $\tau_{\mu_{A},\mu_{B}}(F(\mathsf a, \mathsf b))$ for all $F\in\mathcal F$ (as in \eqref{e:defcF}). 
\end{enumerate}
\end{theorem}

\begin{remark}\label{r:beta1}
For $\beta=2$ there is another interpretation of the eigenvalues of $X_N (t)$ as non-intersecting Brownian bridges, which satisfies  a monotonicity statement. As we will see in the proof, the monotonicity is crucial. 
Thus the statements in Theorem \ref{t:convl1} were only proven for $\beta=2$. 
Since the limiting spherical integral $I(\mu_A,\mu_B)$ does not depend on $\beta$,
the information from studying $\beta=2$ case is enough for us to analyze $I(\mu_A,\mu_B)$. For instance we can write 
$$I(\mu_A,\mu_B)=\int_0^1 \tau_{\mu_{uA},\mu_{B}}(\mathsf a\mathsf b) \rd u\,.$$
In particular we may use this representation to differentiate the spherical integral. 
\end{remark}

Before proving Theorem \ref{t:convl1}, we recall some results on nonintersecting Brownian motions.
We denote $A_N=\diag\{a_1\leq a_2\leq \cdots\leq  a_N\}$ and $B_N=\diag\{b_1\leq b_2\leq \cdots\leq  b_N\}$. Let $(H_{N}(t), t\in [0,1])$ be the matrix Hermitian Brownian motion and consider
the process $X_{N}(t)=A_{N}+H_{N}(t)$. From \eqref{e:lawX1}, one can see that the distribution of $X_{N}(1)=A_{N}+H_{N}(1)$ conditioning to have eigenvalues given by $B_N$ is the same as the law of $UB_NU^{*}$ where $U$ has distribution $\mu_{A_N,B_N}$ as in \eqref{e:lawU}. For $\beta=2$, there is another interpretation of the eigenvalues of $X_N(t)$ as non-intersecting Brownian bridges: the law of the eigenvalues of $X_N(t)$ is the law of nonintersecting Brownian bridges $w_{1}(t)\leq w_{2}(t)\leq \cdots \leq w_{N}(t)$, where $w_i(t)$ is from $a_i$ to $b_i$,  conditioning not to intersect each other.
In the following, we show that 
\begin{lemma}\label{c:densitye}  Assume $(1/N)\sum_{i=1}^N\delta_{a_i}$ and $(1/N)\sum_{i=1}^N\delta_{b_i}$ converges towards $\mu_A$ and $\mu_B$.
The empirical measure of nonintersecting Brownian motions 
\begin{align}
\label{e:empbridge}\mu^{N}_{t}=\frac{1}{N}\sum_{i=1}^N\delta_{w_{i}(t)},
\end{align}
converges weakly almost surely to a measure valued process $\{\rho_t^*\}_{0\leq t\leq 1}$. Moreover, the map between $(\mu_A,\mu_B)$ and the measure valued process $\{\rho_t^*\}_{0\leq t\leq 1}$ is continuous, with respect to the uniform weak topology.
\end{lemma}

\begin{proof}[Proof of Lemma \ref{c:densitye}]
There is a monotonicity statement for nonintersecting Brownian bridges \cite[Lemmas 2.6 and 2.7]{MR3152753}. Given two pairs of boundary data $(\hat a_1\leq \hat a_2\leq \cdots\leq \hat a_N)$, $(\hat b_1\leq \hat b_2\leq \cdots\leq \hat b_N)$,  $(\tilde a_1\leq \tilde a_2\leq \cdots\leq \tilde a_N)$ and $(\tilde b_1\leq \tilde b_2\leq \cdots\leq \tilde b_N)$. We consider nonintersecting Brownian bridges from $(\hat a_1,\hat a_2,\cdots, \hat a_N)$ and $(\tilde a_1, \tilde a_2, \cdots, \tilde a_N)$ to $(\hat b_1,\hat  b_2, \cdots,\hat b_N)$ and $(\tilde b_1, \tilde b_2, \cdots, \tilde b_N)$: $\hat w_1(t)\leq \hat w_2(t)\cdots\leq \hat w_N(t)$ and $\tilde w_1(t)\leq \tilde w_2(t)\cdots\leq \tilde w_N(t)$. If $a_i\geq \tilde a_i$ and $b_i\geq \tilde b_i$ for all $1\leq i\leq N$, \cite[Lemmas 2.6 and 2.7]{MR3152753} gives a coupling, such that at any time $0\leq t\leq 1$, $\hat w_i(t)\geq \tilde w_i(t)$ for all $1\leq i\leq N$. 

Especially if we denote the empirical particle density of the two nonintersecting Brownian bridges as $\hat \mu_t^N=(1/N)\sum_{i=1}^N \delta_{w_i(t)}$ and ${\tilde \mu_t}^N = (1/N)\sum_{i=1}^N \delta_{\tilde w_i(t)}$, then 
\begin{align}\label{e:monotone}
\hat h^N(x,t)\deq\hat \mu_t^N([-\infty, x])\leq  
\tilde\mu_t^N([-\infty,x])\eqd \tilde h^N(x,t),
\end{align}
for any $x\in \bR$, almost surely under the coupling. We note it is possible that some $\hat a_i, \tilde a_i, \hat b_i, \tilde b_i$ are at $\pm \infty$ and $\hat \mu_t^N, \tilde \mu_t^N$ may have delta mass at $\pm \infty$. The statements in \cite[Lemmas 2.6 and 2.7]{MR3152753} still hold if some particles are at $\pm \infty$.
Combining the discussion above with Theorem \ref{t:FBB}, if as $N\rightarrow \infty$, the empirical density of the boundary data converges to $\hat \mu_A,\hat  \mu_B, \tilde \mu_A, \tilde \mu_B$ respectively with compact support on $(-\infty, \infty)$, possibly some delta mass at $+\infty$, then $\hat \mu_t^N$, $\tilde \mu_t^N$ converge weakly to some measure-valued processes $\hat \mu_t, \tilde \mu_t$ respectively. Their cumulative densities satisfy
\begin{align}\label{e:densityb}
\hat h(x,t)\deq\hat \mu_t([-\infty, x])\leq  
\tilde\mu_t([-\infty,x])\eqd \tilde h(x,t).
\end{align}

If $\mu_{A_N}, \mu_{B_N}$ are both uniformly compactly supported, the convergence of the empirical particle density of the nonintersecting Brownian bridges follows from Theorem \ref{t:FBB}. If $\mu_{A_N}$ or $\mu_{B_N}$ are not compactly supported, we approximate them with compact ones and use the monotonicity property of the nonintersecting Brownian bridges to show the existence of limiting density. 

We denote $A^{\varepsilon+}_N, B^{\varepsilon+}_N$ the new boundary data by moving the first and last $\lfloor \varepsilon N\rfloor$ particles of $A_N, B_N$ to the location $-\infty$, i.e. they are $(-\infty\leq \cdots\leq -\infty\leq a_{\lfloor \varepsilon N\rfloor+1}\leq a_{\lfloor \varepsilon N\rfloor+2}\cdots\leq a_{N-\lfloor \varepsilon N\rfloor})$ and $(-\infty\leq \cdots\leq -\infty\leq b_{\lfloor \varepsilon N\rfloor+1}\leq b_{\lfloor \varepsilon N\rfloor+2}\cdots\leq b_{N-\lfloor \varepsilon N\rfloor})$. We denote the nonintersecting Brownian bridge between them as $w^+_1(t)\leq w^+_2(t)\leq \cdots \leq w^+_N(t)$. Then $A_N, B_N, A_N^{\varepsilon+}, B_N^{\varepsilon+}$ satisfy the monotone condition, \eqref{e:monotone} implies
\begin{align}\label{e:ubb1}
\mu_t^N([-\infty, x])\leq \frac{1}{N}\#\{j: w^+_j(t)\leq x\}\eqd h_+^N(x,t),
\end{align}
almost surely. The empirical particle densities of $A_N^{\varepsilon+}, B_N^{\varepsilon+}$ converge to  measures in the form of a delta mass at $-\infty$ plus a compactly supported measure.  From the discussion above and Theorem \ref{t:FBB}, the limits exist
\begin{align}\label{e:ubb2}
\lim_{N\rightarrow \infty} h_+^N(x,t)=h^\varepsilon_+(x,t), 
\end{align}
and $\lim_{x\rightarrow-\infty}h_+^{\varepsilon}(x,t)=2\varepsilon$, $\lim_{x\rightarrow+\infty}h_+^{\varepsilon}(x,t)=1$.
Similarly, we denote $A^{\varepsilon-}_N, B^{\varepsilon-}_N$ the new boundary data by moving the first and last $\lfloor \varepsilon N\rfloor$ particles of $A_N, B_N$ to the location $+\infty$, and define $h_-^\varepsilon$ analogously. Then we have 
\begin{align}\label{e:lbb}
\liminf_{N\rightarrow\infty}\mu_t^N([-\infty, x])\geq h_-^\varepsilon(x,t),
\end{align}
almost surely.
Combining those estimates \eqref{e:ubb1}, \eqref{e:ubb2} and \eqref{e:lbb} together, we get
\begin{align*}
h_-^\varepsilon(x,t)\leq \liminf_{N\rightarrow\infty}\mu_t^N([-\infty, x])\leq \limsup_{N\rightarrow\infty}\mu_t^N([-\infty, x])\leq h_+^\varepsilon(x,t),
\end{align*}
almost surely. From our construction, $h_-^\varepsilon(x,t)$ is simply a shift of 
$h_+^\varepsilon(x,t)$, i.e. $h_-^\varepsilon(x,t)=h_+^\varepsilon(x,t)-2\varepsilon$. Moreover, thanks to the monotonicity property of nonintersecting Brownian bridges, $h_-^\varepsilon(x,t)$  is nondecreasing in $\varepsilon$ and $h_+^\varepsilon(x,t)$ is nonincreasing in $\varepsilon$. Thus the limits exist
\begin{align*}
h(x,t)\deq \lim_{\varepsilon\rightarrow 0}h_-^\varepsilon(x,t)=\lim_{\varepsilon\rightarrow 0}h_+^\varepsilon(x,t),
\end{align*}
and $h(x,t)$ gives the limiting cumulative density of $\mu_t^N$. This finishes the proof of \ref{e:empbridge} and the first point of Lemma  \ref{c:densitye}.
For the second point, if we have a sequence of probability measures $(\mu_{A^{(p)}}, \mu_{B^{(p)}})_{p\geq 1}$ converging weakly to $(\mu_A, \mu_B)$. We denote the limiting height functions for corresponding Brownian bridges as $h^{(p)}(x,t)$ and $h^*(x,t)$. Since weak topology is compatible with the L{\'e}vy metric, for any $\varepsilon>0$, we have for $p$ large enough, 
\begin{align}\label{e:t01}
h^*(x-\varepsilon,t)-\varepsilon \leq h^{(p)}(x,t)\leq h^*(x+\varepsilon,t)+\varepsilon, \quad x\in \bR,\quad t\in\{0,1\}.
\end{align}
By the monotonicity argument, we have that  \eqref{e:t01} holds for $0\leq t\leq 1$, and the second point follows.  
\end{proof}

\begin{proof}[Proof of Theorem \ref{t:convl1}]
It follows from the second point of Lemma \ref{c:densitye} and the fact that the limit of solutions of Euler equations \eqref{eulpr10} and \eqref{eulpr20} is still a solution (even if we only assume $\mu_A(|x|)<\infty$ rather than $\mu_A(x^2)<\infty$) that
\begin{enumerate}
\item The law of eigenvalues of $X_N(t)=A_N+H_N(t)$, conditioned to have the same eigenvalues as $B_N$ at time $ t=1$, converges towards the distribution $(\rho^{*}_{t})_{0\leq t\leq 1}$ such that 
for all $t\in (0,1)$, $\partial_{t}\rho^{*}_{t}+\partial_{x}(\rho^{*}_{t}u^{*}_{t})=0$ and they satisfy the Euler equation \eqref{eulpr20}. 
\item $\{\rho^*_t\}_{0\leq t\leq 1}$ is uniquely described as the weak  limit of  $\{\rho^{(p)}_t\}_{0\leq t\leq 1}$ which minimizes the strictly convex function  $S(u,\rho)$ as in \eqref{e:funcS} for regularized boundary data $(\mu_{A^{(p)}},\mu_{B^{(p)}})$.
\end{enumerate}

In the following we prove \eqref{e:tauF}. 
We need to  show  that $\{\rho_t^*\}_{0\leq t\leq 1}$ characterize uniquely the limit points of  \begin{align}\label{e:deftau}
\tau_{N}(F)\deq \int \frac{1}{N}\Tr(F(A_N,UB_NU^{*}))\rd\mu_{A_N,B_N}(U)\,,
\end{align}
for all $F\in\mathcal F$.
The non-commutative law $\tau_N$ is sequentially tight since it belongs to a compact space. We can therefore consider a  limit point $\tau$ of $\tau_N$ and need to show it is unique. 

We have already proven that for all $t\in [0,1]$, all $z\in\mathbb C\backslash \mathbb R$,
$$\tau\left(\frac{1}{z-(1-t)\mathsf a-t\mathsf b-\sqrt{t(1-t)}\mathsf s}\right)=\int \frac{1}{z-x}\rho^{*}_{t}(x) \rd x\,.$$
This is enough to deduce the distribution $\nu_{t}$ of $(1-t)\mathsf a+t\mathsf b$, since the  $R$-transform formula yields for $z$  small enough
$$R_{\nu_{t}}(z)=R_{\rho^{*}_{t}}(z)-t(1-t) z, $$
where $R_{\nu_t}$ and $R_{\rho^*_t}$ are the $R$-transforms of $\nu_t$ and $\rho^{*}_t$ respectively. This defines uniquely the Stieltjes transform of $\nu_t$.
We then deduce by a change of variable $u=t/(1-t)$ that
\begin{align}\label{e:tauab}
\tau\left(\frac{1}{z-\mathsf a-u\mathsf b}\right)=\int \frac{1-t}{(1-t)z-x}\rd\nu_{t}(x)\,.
\end{align}
Since $\|\mathsf b\|<\infty$, we can Taylor expand the above expression
\begin{align}\label{e:tylor}
\tau\left(\frac{1}{z-\mathsf a-u\mathsf b}\right)=\sum_{k\geq 0}\tau\left(\frac{u^k}{z-\mathsf a}\left(\sfb \frac{1}{z-\mathsf a}\right)^k\right),
\end{align}
for $s>0$ small enough. Term by term, we can use \eqref{e:tylor} to extract the value of 
\begin{align}\label{e:tauk}
\tau\left(\frac{1}{z-\mathsf a}\left(\sfb \frac{1}{z-\mathsf a}\right)^k\right),
\end{align}
for any $k\geq 0$.
In the following we show that  those are enough to  retrieve the complete joint distribution $\tau$ thanks to loop equations. 
We want to show that  the observables
\begin{equation}\label{e:ob}\tau\left( \sfb^{n_0}\frac{1}{z_1-\mathsf a}\sfb^{n_1}\frac{1}{z_2-\mathsf a}\cdots \mathsf b^{n_{k-1}}\frac{1}{z_k-\mathsf a} \sfb^{n_k}\right)\end{equation}
are uniquely determined
for any  choice of $k\geq 0$, $z_i\in \bC\setminus\bR$ and $n_i\geq 0$ with $0\leq i\leq n$. Note that we may and shall assume that $\tau$ is tracial and hence take $n_0=0$.  Also, any $n_i$ can vanish and hence the above contains all polynomials in $(z_i-\sfa)^{-1}$ and $\sfb$. We prove the uniqueness  by induction on the degree of $\sfb$, i.e. $\deg_{\sfb}(F)=\sum n_i$.
We assume we know \eqref{e:ob} for any $F\in \cF$ with $\deg_{\sfb}(F)\leq m$, and prove the statement for $\deg_{\sfb}(F)=m+1$. For the base case $m=0$, $F$ is a polynomial of 
$(
\frac{1}{z_{i}-\mathsf a})_{1\leq i\leq k} $which is characterized by the marginal distribution $\mu_A$ of $\sfa$.

To proceed by induction we use the loop equations, see e.g. \cite{GN}, which implies that for any $G\in \cF$,
\begin{align}\label{e:loop}
\tau\otimes \tau (\partial G)=\tau(G (\mathsf a\mathsf b-\mathsf b\mathsf a)),
\end{align}
where 
$$\partial G=\sum_{G=G_{1}\mathsf b G_{2}} \left(G_{1}\mathsf b \otimes G_{2}-G_{1}\otimes \mathsf b G_{2}\right)\,.$$
In the following, we will use the loop equation \eqref{e:loop} to commute $1/(z-\sfa)$ and $\sfb$. For any $F\in \cF$, with $\deg_\sfb(F)=m+1$, since $\tau$ is tracial, we can rewrite it as
\begin{align}\begin{split}\label{e:gF}
\tau(F)=\tau\left(\tilde F \frac{1}{z-\sfa} \sfb\right)
&=\tau\left(\tilde F \sfb\frac{1}{z-\sfa} \right)+\tau\left(\tilde F \frac{1}{z-\sfa} (\sfa\sfb-\sfb\sfa)\frac{1}{z-\sfa}\right)\\
&=\tau\left(\tilde F \sfb\frac{1}{z-\sfa} \right)+\tau\left(\frac{1}{z-\sfa}\tilde F \frac{1}{z-\sfa} (\sfa\sfb-\sfb\sfa)\right),
\end{split}\end{align}
for some $z\in \{z_1, z_2,\cdots, z_k\}$, and $\deg_{\sfb}(\tilde F )=m$.  For the last term in \eqref{e:gF}, we can use the loop equation \eqref{e:loop} with 
$G=(z-\sfa)^{-1}\tilde F(z-\sfa)^{-1}$, 
\begin{align}
\tau\left(\frac{1}{z-\sfa}\tilde F \frac{1}{z-\sfa} (\sfa\sfb-\sfb\sfa)\right)
=\tau\otimes \tau \left(\partial \left(\frac{1}{z-\sfa}\tilde F \frac{1}{z-\sfa} \right)\right)
\end{align}
The righthand side decomposes into polynomials of $(z_{1}-\mathsf a)^{-1}, (z_{2}-\mathsf a)^{-1},\cdots,  (z_{k}-\mathsf a)^{-1},\sfb$, and the degree of $\sfb$ is at most $m$. Therefore, we can compute $\tau((z-\sfa)^{-1}\tilde F(z-\sfa)^{-1}(\sfa\sfb-\sfb\sfa))$ uniquely by our induction hypothesis. Using \eqref{e:gF}, we can commute $(1/(z-\sfa))$ and $\sfb$, the extra terms can be uniquely computed by our induction hypothesis. By repeated using \eqref{e:gF}, and commuting $(1/(z-\sfa))$ and $\sfb$, computing $\tau(F)$ is reduced to compute terms in the form
\begin{align}\label{e:res1}
\tau\left(P\left(\frac{1}{z_1-\sfa}, \frac{1}{z_2-\sfa}, \cdots, \frac{1}{z_k-\sfa}\right)\sfb^{m+1}\right)
\end{align}
where $P$ is a polynomial. By using partial fraction, we can rewrite the polynomial $P$ in \eqref{e:res1}   as a linear combination of terms in the form $1/(z_i-\sfa)^\ell$. 
Thus to compute \eqref{e:res1}, we only need to compute 
\begin{align}\label{e:zl}
\tau\left(\frac{1}{(z-\sfa)^\ell}\sfb^{m+1}\right),
\end{align}
for any $z\in \{z_1,z_2,\cdots, z_k\}$ and $\ell\geq 0$.
Thanks to  \eqref{e:tauk}, we can compute 
\begin{align}
\tau\left(\frac{1}{(z-\sfa)}\left(\sfb\frac{1}{z-\sfa}\right)^{m+1}\right).
\end{align}
If we take 
\begin{align}
F=\frac{1}{(z-\sfa)}\left(\sfb\frac{1}{z-\sfa}\right)^{m+1},
\end{align}
the above procedure allows us to commute $1/(z-\sfa)$ and $\sfb$, and compute 
\begin{align}\label{e:zm}
\tau\left(\frac{1}{(z-\sfa)^{m+2}}\sfb^{m+1}\right),
\end{align}
for any $z\in \bC\setminus \bR$.
We can use \eqref{e:zm} to compute \eqref{e:zl} by taking derivatives and antiderivatives with respect to $z$.
This yields our induction hypothesis for $\deg_{\sfb}(F)=m+1$ and  finishes the proof of the convergence of the non-commutative law.

Finally, for the last point, we notice that Lemma  \ref{c:densitye} shows the continuity of $\{\rho_t^*\}_{0\leq t\leq 1}$ in its boundary values, i.e. $\mu_A,\mu_B$, with respect to the weak topology. 
More precisely, if $(\mu_{A^{(p)}},\mu_{B^{(p)}})_{p\geq 0}$ is a sequence of probability measures such that there exists a finite constant $\fK$ such that   $\supp\mu_{A^{(p)}}\subset[-\fK, \fK]$ and  $\mu_{B^{(p)}}(|x|)\leq \fK$, and they converge weakly towards $\mu_{A},\mu_{B}$ respectively, then 
their limiting empirical measure $\{\rho_t^{(p)}\}_{0\leq t\leq 1}$ converges to $\{\rho_t^*\}_{0\leq t\leq 1}$ in weak topology uniform in $t$.  The same as in the proof of \eqref{e:tauF}, we can construct the tracial state $\tau_{\mu_{A^{(p)}},\mu_{B^{(p)}}}$ for the measure valued process $\{\rho_t^{(p)}\}_{0\leq t\leq 1}$. In the proof of \eqref{e:tauF}, we constructed $\tau$ using \eqref{e:tauk} and the recursion relation \eqref{e:gF}. Thus if 
$\{\rho_t^{(p)}\}_{0\leq t\leq 1}$ converges to $\{\rho_t^*\}_{0\leq t\leq 1}$ in weak topology uniform in $t$, we have $\tau_{\mu_{A^{(p)}},\mu_{B^{(p)}}}(F)$ converges towards $\tau_{\mu_{A},\mu_{B}}(F)$ for all $F\in\mathcal F$. 
Hence we can conclude that $\tau_{\mu_A,\mu_B}$ depends on $\mu_A,\mu_B$ continuously (in the weak topology). 
\end{proof}

\subsection{Derivatives of the Spherical Integral}\label{s:derSI}
In this section we compute the derivatives of the spherical integral. This will be crucial to analyze the critical points of our large deviations rate functions which are expressed as supremum of  functions depending on spherical integral.  Since the limiting functional $I(\mu_A, \mu_B)$ defined in \eqref{e:defI} is independent of $\beta$, in this section, we take 
$\beta=2$. We identify any $N\times N$ diagonal matrix $A_N=\diag\{a_1,a_2,\cdots,a_N\}$ (that is, $(A_N)_{i,j}=\delta_{i=j}a_i,1\le i,j\le N$) with the multiplicative operator 
$\tilde T_{A_N}:[0,1)\mapsto\bR,$ 
\begin{align*}
\tilde T_{A_N}(x)=\sum_{i=1}^N a_i {\bm 1}_{[\frac{i-1}{N}, \frac{i}{N})}(x).
\end{align*}
From the definition, the empirical eigenvalue distribution $\mu_{A_N}=(1/N)\sum_i \delta_{a_i}$ of $A_N$ is the push forward measure of the uniform measure on $[0,1]$ by $\tilde T_{A_N}$. We rearrange $a_1, a_2, \cdots, a_N$ in increasing order: $a_{1^*}\leq a_{2^*}\leq \cdots \leq a_{N^*}$ and define the multiplicative operator
\begin{align*}
T_{A_N}(x)=\sum_{i=1}^N a_{i^*} {\bm 1}_{[\frac{i-1}{N}, \frac{i}{N})}(x).
\end{align*}
Then $T_{A_N}$ is a right continuous nondecreasing function. Moreover, if $F_{A_N}$ is the cumulative density of the empirical eigenvalue distribution $\mu_{A_N}$, then $T_{A_N}$ is the functional inverse of $F_{A_N}$.
More generally for any measurable function $\tilde T_A: [0,1]\mapsto \bR$, we denote the measure $\mu_A=(\tilde T_A)_\# \unif[0,1]$ the pushforward of the uniform measure on $[0,1]$ by $\tilde T_A$, $F_A$ the cumulative density of $\mu_A$ and $T_{A}$ the functional inverse of $F_{A}$, which is right continuous and non-decreasing.

A sequence of measures $\mu_{A_N}$ converges weakly to $\mu_A$ if and only if  $T_{A_N}$ converges to $T_A$ at all continuous point of $T_A$. And $\mu_{A_N}$ converges  in Wasserstein distance \eqref{e:wd} to $\mu_A$ if and only if $T_{A_N}$ converges to $T_A$ in $L^1$ norm.

Let $A_{N},B_{N}$ be two sequences of deterministic self-adjoint matrices, such that their spectral measures $\mu_{A_N}$ and $\mu_{B_N}$ converge  in Wasserstein distance \eqref{e:wd} towards $\mu_{A}$ and $\mu_{B}$ respectively. We assume that  there exists a constant $\fK>0$, such that $ \mu_{A_N}(|x|)\leq \fK$ and $\supp \mu_{B_N}\subset [-\fK, \fK]$.   
 As a consequence of Theorem \ref{t:convl1}, for any bounded Lipschitz function $f: \bR\mapsto \bR$, 
\begin{align}\label{e:der}
\lim_{N\rightarrow \infty} \tau_N(f(A_N)UB_NU^*)
= \tau(f(\mathsf a) \mathsf b)
=\tau(\tau(f(\mathsf a)\mathsf b|\mathsf a))
=\tau(f(\mathsf a)\tau(\mathsf b|\mathsf a)),
\end{align}
where $U$ follows the law \eqref{e:lawU}.
The goal of this section is to characterize the derivative of the spherical integral using the non-commutative distribution $\tau$. Indeed, we have 
\begin{align}\begin{split}\label{e:perturbc}
&\phantom{{}={}}\left.\del_{\varepsilon}\frac{1}{\beta N^2}\ln \int e^{\frac{\beta N}{2}\Tr((A_N+\varepsilon C_N)UB_NU^*)}dU\right|_{\varepsilon=0}=\frac{1}{2N}
\frac{\int \Tr(C_NUB_NU^*)e^{\frac{\beta N}{2}\Tr(A_N UB_NU^*)}dU}{\int e^{\frac{\beta N}{2}\Tr(A_NUB_NU^*)}dU}.
\end{split}\end{align}
\begin{proposition}\label{p:derI}
Given two probability measures $\mu_A, \mu_B$, such that $\mu_A(|x|)<\infty$ and $\mu_B$ is compactly supported, for any compactly supported and Lipschitz  real-valued function $f$, it holds
\begin{align}\label{e:nondelta}
\left.\del_\varepsilon I(T_A+\varepsilon f(T_ A), T_B)\right|_{\varepsilon=0}=\frac{1}{2}\int f(x)\tau(\mathsf b|\mathsf a)(x)\rd \mu_A (x).
\end{align}
If the measure $\mu_A$ has a delta mass at $a$, for any bounded measurable function $\tilde T_C$ supported on $\{x: T_A(x)=a\}$, it holds
\begin{align}\label{e:delta}
\left.\del_\varepsilon I(T_A+\varepsilon \tilde T_C, T_B)\right|_{\varepsilon=0}=\frac{1}{2}\tau(\mathsf b|\mathsf a)(a)\int \tilde T_C(x)\rd x.
\end{align}
\end{proposition}
We remark that thanks to  \eqref{e:condut}, we can express the conditional expectation in terms of the solution $(\rho_t^*, u_t^*)$ of the variational problem \eqref{e:funcS} by $\tau(\mathsf b|\mathsf a)(x)=u_0^*(x)-H\mu_A(x)+x$.

\begin{proof}
Since the spherical integral $I_N(A_N, B_N)$ is convex in both $A_N$ and $B_N$, its limit $I(T_A, T_B)$ is also convex in $T_A$ and $T_B$. Especially for any sequence $C_{N}$ of deterministic self-adjoint matrices, such that their spectral measures $\mu_{C_N}$converge weakly towards $\mu_{C}$, both $I_N(A_N+\varepsilon C_N, B_N)$ and $I(T_A+\varepsilon \tilde T_C, T_B)$ are convex in $\varepsilon$. Then for sufficiently small $\varepsilon>0$,
\begin{align}\label{e:uplow}
\del_\varepsilon I_N(A_N+\varepsilon C_N, B_N)\geq \frac{I_N(A_N+\varepsilon C_N, B_N)-I_N(A_N, B_N)}{\varepsilon}\geq \left.\del_\varepsilon I_N(A_N+\varepsilon C_N, B_N)\right|_{\varepsilon=0}\,.
\end{align}
Thanks to \eqref{e:der} and \eqref{e:perturbc}, we have that if $C_N=f(A_N)$, then 
\begin{align}\label{e:oneside}
\lim_{N\rightarrow \infty}\left.\del_\varepsilon I_N(A_N+\varepsilon f(A_N), B_N)\right|_{\varepsilon=0}
=\frac{1}{2}\tau(f(\mathsf a)\mathsf b).
\end{align}
For the lefthand side of \eqref{e:uplow}, we have
\begin{align*}\begin{split}
&\phantom{{}={}}\del_\varepsilon I_N(A_N+\varepsilon f(A_N), B_N)
=\frac{1}{2N}
\frac{\int \Tr(f(A_N)UB_NU^*)e^{\frac{\beta N}{2}\Tr((A_N+\varepsilon f(A_N))UB_NU^*)}\rd U}{\int e^{\frac{\beta N}{2}\Tr((A_N+\varepsilon f(A_N))UB_NU^*)}\rd U}\\
&=\frac{1}{2N}
\frac{\int \Tr(f(A_N+\varepsilon f(A_N))UB_NU^*)e^{\frac{\beta N}{2}\Tr((A_N+\varepsilon f(A_N))UB_NU^*)}\rd U}{\int e^{\frac{\beta N}{2}\Tr((A_N+\varepsilon f(A_N))UB_NU^*)}\rd U}+\OO(\varepsilon),
\end{split}\end{align*}
provided $f$ is compactly supported and Lipschitz. By taking the limit $N\rightarrow \infty$, we get
\begin{align*}
\lim_{N\rightarrow \infty}\del_\varepsilon I_N(A_N+\varepsilon f(A_N), B_N)
=\frac{1}{2}\tau^{\varepsilon}(f(\mathsf a)\mathsf b)+\OO(\varepsilon),
\end{align*}
where $\tau^{\varepsilon}$
is the limit of the joint law of $(A_N+\varepsilon f(A_N), UB_NU^*)$ with the normalized trace $\Tr(\cdot)/N$ under the deformed measure $\rd\mu_{A_{N}+\varepsilon f(A_N),B_{N}}(U)$.
Thanks to the last point in Theorem \ref{t:convl1}, since the limiting empirical measure of $A_N+\varepsilon f(A_N)$ converges weakly to $\mu_A$ as $\varepsilon$ goes to zero,  $\tau^{\varepsilon}$ is continuous in $\varepsilon$. As a consequence, by taking the limit $N\rightarrow \infty$ in \eqref{e:uplow} and combining with \eqref{e:oneside}, we get
\begin{align}\label{e:derIbound}
\frac{1}{2}\tau(f(\mathsf a)\mathsf b)+\oo_\varepsilon(1)\geq \frac{I(T_A+\varepsilon f(T_A), T_B)-I(T_A, T_B)}{\varepsilon}\geq \frac{1}{2}\tau(f(\mathsf a)\mathsf b)\,.
\end{align}
By a similar argument, we also have the estimate \eqref{e:derIbound} for $\varepsilon\leq 0$. The first claim \eqref{e:nondelta} follows by taking $\varepsilon$ to zero. 

In the following we deal with the second case, that the measure $\mu_A$ has a delta mass at $a$ with $\mu_A(a)=m$, and $T_C$ is supported on that $\{x: T_A(x)=a\}$.
We take a sequence of diagonal matrices $A_N=\diag\{a_1, a_2, \cdots, a_N\}$ with non-decreasing diagonal entries, with empirical eigenvalue distribution $\mu_{A_N}$ converging to $\mu_A$. Moreover, $a_i=a$ for $i\in \qq{\alpha N+1, (\alpha+m)N}$. We also take the sequence of diagonal matrices $C_N=\diag\{c_1, c_2, \cdots, c_N\}$ (not necessarily non-decreasing), with empirical eigenvalue distribution $\mu_{C_N}$ converging to $\mu_C$. Moreover, $c_i=0$ for $i\not\in \qq{\alpha N+1, (\alpha+m)N}$. We can write $C_N$ in the block form $\bm 0_{\al N}\oplus \tilde C_N\oplus \bm 0_{(1-\al-m) N}$ where $\bm 0_{\al N}$ and $\bm 0_{(1-\al-m) N}$ are zero matrices of sizes $\alpha N$ and $(1-\alpha-m)N$ respectively, and $\tilde C_N$ is an $m N \times m N$ diagonal matrix.
We take $P_N$ to be the projection operator onto $i\in \qq{\alpha N+1, (\alpha+m)N}$ entries, and $\tilde P_N$ the $mN\times N$ rectangular matrix consisting of the $\qq{\alpha N+1, (\alpha+m)N}$ rows of $P_N$. Explicitly, it is diagonal and also a function of $A_N$:  $P_N=\bm1_{x=a}(A_N)$, $(P_N)_{ii}=\bm1((A_N)_{ii}=a)$ for $1\leq i\leq N$.
Since $U$ follows the Haar measure on orthogonal/unitary group, if we multiply $U$ by the block diagonal matrix $I_{\alpha N}\oplus V\oplus I_{(1-\alpha-m)N}$ where $I_{\alpha N}$ and $I_{(1-\alpha-m)N}$ are identity matrices of sizes $\alpha N$ and $(1-\alpha-m)N$ respectively, and $V$ is an $m N \times m N$ orthogonal/unitary matrix, the law of $U$ does not change. Then we have
\begin{align*}
&\phantom{{}={}}\int \Tr(C_NUB_NU^*)e^{\frac{\beta N}{2}\Tr((A_N+\varepsilon C_N)UB_NU^*)}\rd U
\\
&=\int e^{\frac{\beta N}{2}\Tr(A_NUB_NU^*)}\int  \Tr(V\tilde C_N V^* \tilde P_NUB_N U^* \tilde P_N^*)e^{\frac{\varepsilon\beta N}{2}\Tr(V\tilde C_N V^* \tilde P_NUB_N U^* \tilde P_N^*)}\rd V\rd U,
\\
&\phantom{{}={}}\int e^{\frac{\beta N}{2}\Tr((A_N+\varepsilon C_N)UB_NU^*)}\rd U
=\int e^{\frac{\beta N}{2}\Tr(A_NUB_NU^*)}\int e^{\frac{\varepsilon\beta N}{2}\Tr(V\tilde C_N V^* \tilde P_NUB_N U^* \tilde P_N^*)}\rd V\rd U.
\end{align*}
We can rewrite \eqref{e:perturbc} as
\begin{align}\begin{split}\label{e:derIbb}
&\del_{\varepsilon}I_N(A_N+\varepsilon C_N, B_N)
=\frac{1}{2N}
\frac{\int \Tr(C_NUB_NU^*)e^{\frac{\beta N}{2}\Tr((A_N+\varepsilon C_N)UB_NU^*)}\rd U}{\int e^{\frac{\beta N}{2}\Tr((A_N+\varepsilon C_N)UB_NU^*)}\rd U}\\
&=\frac{1}{2N}\frac{\int e^{\frac{\beta N}{2}\Tr(A_NUB_NU^*)}\int  \Tr(V\tilde C_N V^* \tilde P_NUB_N U^* \tilde P_N^*)e^{\frac{\varepsilon\beta N}{2}\Tr(V\tilde C_N V^* \tilde P_NUB_N U^* \tilde P_N^*)}\rd V\rd U}{\int e^{\frac{\beta N}{2}(\Tr(A_NUB_NU^*))}\int e^{\frac{\varepsilon\beta N}{2}\Tr(V\tilde C_N V^* \tilde P_NUB_N U^* \tilde P_N^*)}\rd V\rd U}.
\end{split}\end{align}
For the integral over $V$ conditionally to $U$, we use the results \cite[Theorem 0.1]{CGM} to find that for $\varepsilon$ small enough, $N$ large enough,
\begin{align}\label{e:integralV}
\frac{1}{2N}\frac{\int  \Tr(V\tilde C_N V^* \tilde P_NUB_N U^* \tilde P_N^*)e^{\frac{\varepsilon\beta N}{2}\Tr(V\tilde C_N V^* \tilde P_NUB_N U^* \tilde P_N^*)}\rd V}{\int e^{\frac{\varepsilon\beta N}{2}\Tr(V\tilde C_N V^* \tilde P_NUB_N U^* \tilde P_N^*)}\rd V}=\frac{\Tr(C_N)}{2mN} \frac{\Tr(P_NUB_NU^{*}P_N)}{N}+\oo_{\varepsilon, N}(1).
\end{align}
We recall that $P_N$ is  a function of $A_N$: $P_N={\bm1}_{x=a}(A_N)$. By plugging \eqref{e:integralV} into \eqref{e:derIbb}, we get that
\begin{align*}
\del_{\varepsilon}I_N(A_N+\varepsilon C_N, B_N)
=
\frac{\Tr(C_N)}{2m N^2}\frac{\int \Tr({\bm1}_{x=a}(A_N)UB_NU^*)e^{\frac{\beta N}{2}\Tr((A_N+\varepsilon C_N)UB_NU^*)}\rd U}{\int e^{\frac{\beta N}{2}\Tr((A_N+C_N)UB_NU^*)}\rd U}+\oo_{\varepsilon, N}(1).
\end{align*}
By taking the limit $N\rightarrow \infty$, we deduce
\begin{align}\label{e:derIbb2}
\lim_{N\rightarrow \infty}\del_\varepsilon I_N(A_N+\varepsilon f(A_N), B_N)
=\frac{\int \tilde T_C\rd x}{2\mu_A(a)}\tau^{\varepsilon}(\bm 1_{x=a}(\mathsf a)\mathsf b)+\oo_{\varepsilon}(1)\,.
\end{align}
By taking the limit $N\rightarrow \infty$ in \eqref{e:uplow}, and using that $\tau^\varepsilon$ is continuous in $\varepsilon$ at zero, \eqref{e:derIbb2} implies that
\begin{align}\label{e:derlbound2}
\frac{1}{2}\frac{\int \tilde T_C \rd x}{\mu_A(a)}\tau(\bm 1_{x=a}(\mathsf a)\mathsf b)+\oo_\varepsilon(1)\geq \frac{I(T_A+\varepsilon \tilde T_C, T_B)-I(T_A, T_B)}{\varepsilon}\geq \frac{1}{2}\frac{\int \tilde T_C \rd x}{\mu_A(a)}\tau(\bm 1_{x=a}(\mathsf a)\mathsf b)\,.
\end{align}
By a similar argument, we also have the estimate \eqref{e:derlbound2} for $\varepsilon\leq 0$. The second claim \eqref{e:delta} follows by taking $\varepsilon$ to zero. This finishes the proof of Proposition \ref{p:derI}.
\end{proof}

\subsection{Continuity at the boundary}\label{s:continuity}

In this section, we obtain a more precise description of the solutions $(\rho^*,u^*)$ of the variational problem
\eqref{e:Iexp} by transforming the complex Burgers equation \eqref{burg} into a Beltrami equation. This will be a key point to prove Corollary \ref{c:uniquemu}, which is central to the proof of  the lower bound in Theorem \ref{t:LR}. Due to some technical reason, in this section we assume that the boundary data $\nu,\mu$ are compactly supported.

Recall that $ f=u^*+i\pi\rho^*$ satisfies \eqref{burg}. Observe that $(\partial_x-\ri\partial_t)f=\partial_x f+\ri f\partial_xf=(1+\ri f)\partial_xf=\ri(f-\ri)\partial_xf,$ and 
$(\partial_x+\ri\partial_t)f=\partial_xf-\ri f\partial_xf=(1-\ri f)\partial_xf=-\ri(f+\ri)\partial_xf.$ Thus, 
$(\partial_x-\ri\partial_t)f=((\ri-f)/(\ri+f))(\partial_x+\ri\partial_t)f$. Recall that $\Im [f]>0$ for all 
$(t,x)\in\Omega$ as defined in \eqref{Omega},{ so that $|(\ri-f)/(\ri+f)|<1$ for all $(t,x)\in\Omega$. 
Let 
\begin{equation}\label{change}
z=x-\ri t,\quad\bar{z}=x+\ri t,\quad (t,x)\in\Omega.
\end{equation}
Then the above shows that \eqref{burg} is equivalent to the Beltrami equation (see, for instance, \cite{Letho,IM})
\begin{equation}\label{e:Bel}
\partial_{\bar{z}}f=\frac{\ri-f}{\ri+f}\partial_{z}f,\quad (\Im[ z],\Re [z])\in\Omega.
\end{equation}
In general, the Beltrami equation $\partial_{\bar{z}}f={\boldsymbol\mu}(z)\partial_zf$ is 
defined on a domain $\Omega$ on which the measurable function 
$\Omega\ni z\mapsto{\boldsymbol\mu}(z)\in\mathbb C$ satisfies $\|{\boldsymbol\mu}\|_\infty\leq k<1$ 
for a constant $k$. It is a nontrivial classical result that, given such a $\boldsymbol\mu$, the Beltrami equation has continuous
solutions (see \cite[Theorem 4.4]{MR867407}). The solutions have a remarkable geometric property, called quasiregularity:
they are absolutely continuous on lines and $|\partial_{\bar{z}}f|+|\partial_zf|\leq\frac{k+1}{k-1}(|\partial_zf|-|\partial_{\bar{z}}f|)$ a.e.
Even more remarkably, among these solutions 
there exists one which is a homeomorphism (called {\em quasiconformal}) and any other 
solution is obtained by composing it with an analytic map. Since 
$|(\ri-f)/(\ri+f)|\not\leq k$ on $\Omega$, for any $k<1$, it would appear that 
the general theory for obtaining solutions of Beltrami equations does not apply in our case. 
However, the regularity of ${\boldsymbol\mu}=(\ri-f)/(\ri+f)$ and the fact that 
$\Omega=\bigcup_{\varepsilon>0}\{\Im [f]>\varepsilon\}$
allows us easily to conclude that the solutions of \eqref{e:Bel} have 
essentially all the useful properties of quasiregular maps. 


\begin{theorem}\label{main}
Let $f$  be as in \eqref{burg}, with boundary condition $\nu, \mu$, such that they are compactly supported. Then for $\nu^{\rm ac}$-almost all $x\in\mathbb R$, we have
$\lim_{t\to0} f(t,x)=f(0,x)$, where $\nu^{\rm ac}$ is the absolutely continuous part of $\nu$.
\end{theorem}

\begin{corollary}\label{c:uniquemu}
We assume that the  probability measures $\mu,\mu', \nu$ are compactly supported and neither of them is 
concentrated in one point.
We denote the solutions of the variational problems $I(\nu, \mu)$ and $I(\nu,\mu')$ from 
\eqref{e:Iexp} by $f^{\nu\rightarrow \mu}(t,x)$ and $f^{\nu\rightarrow \mu'}(t,x)$ 
$($as in \eqref{burg}$)$ respectively. If all components of $\supp\nu$ are infinite sets and
$f^{\nu\rightarrow \mu}(0,x)=f^{\nu\rightarrow \mu'}(0,x)$ in the sense of distributions, 
then we have $\mu=\mu'$.
\end{corollary}
We shall prove this result  in two steps: first we show that Corollary \ref{c:uniquemu} is true under the assumption that in each connected component of $\nu$, 
there is a set of non-zero $\nu^{\rm{ac}}$-measure such that $f^{\nu\rightarrow \mu}(0,x)=f^{\nu\rightarrow \mu'}(0,x)$,
and then we use a strong version of the reflection principle to prove Corollary \ref{c:uniquemu} in the general case.

We postpone the proof of Theorem \ref{main} and Corollary \ref{c:uniquemu} to Section \ref{s:f}. 
As by-products of the proof of Theorem \ref{main}, we also prove a maximum principle of $f(t,x)$ 
and a precise description on the topology of the set $\Omega$, which might be of independent interest.

\section{Large Deviation Principle for  $UBU^*$}
We recall from Theorem \ref{t:UBU}, $B_N=\diag\{b_1, b_2,\cdots,b_N\}$ is a sequence of deterministic self-adjoint matrices such that the spectral measures $\mu_{B_N}=(1/N)\sum_{i=1}^N \delta_{b_i}$ of $B_N$ converge weakly towards $\mu_B$ as $N\to\infty$. 
In this section, we use the spherical integral to study the large deviation principle of the law $\bP_N$ of the empirical measure
\begin{align*}
\mu_N=\frac{1}{N}\sum_{i=1}^N \delta_{(UB_NU^*)_{ii}},
\end{align*}
and prove Theorem \ref{t:UBU}. 
\subsection{Study of the rate function}

The classical Schur-Horn theorem \cite{MR0063336} states that
the diagonal entries of $UB_NU^*$ are in the permutation polytope generated by $(b_1,b_2,\cdots, b_N)$, or equivalently
\begin{align}\label{e:admissible}
\int_0^1 (T_{\mu_N}-T_{\mu_{B_N}})\rd x= 0,\quad \int_y^1 (T_{\mu_N}-T_{\mu_{B_N}})\rd x\leq 0,
\end{align}
for any $0\leq y\leq 1$.
We recall the functions $H^{\mathsf D}_\mu(\cdot)$ and $\cal I^{\mathsf D}(\cdot)$ from Theorem \ref{t:UBU}, $\mathcal I^{\mathsf D}(\mu)=\sup_{\nu\in \cM}H^{\mathsf D}_\mu(\nu)$. 
 In the following proposition we study these functions, and show the rate function $\mathcal I^{\mathsf D}(\mu)$ equals $+\infty$ outside the admissible set $\mathcal A_{\mu_B}$ of probability measures $\mu$ described by the limiting Schur-Horn theorem \eqref{e:admissible}:
 \begin{align}
\label{e:limitadm}
 \int_0^1 (T_{\mu}-T_{\mu_{B}})(x)\rd x= 0,\quad \int_y^1 (T_{\mu}-T_{\mu_{B}})(x)\rd x\leq 0 \quad \forall y\in [0,1].
 \end{align}

\begin{proposition}\label{p:domain1}
Under the assumptions of Theorem \ref{t:UBU}, the function $H_\mu^{\mathsf D}(\cdot)$ and rate function $\mathcal I^{\mathsf D}(\cdot)$ as defined in Theorem \ref{t:UBU} satisfy:
\begin{enumerate}
\item For $\mu$ satisfying \eqref{e:limitadm}, $H_\mu^{\mathsf D}(\cdot)$ is { upper semi-continuous in the weak topology on $\{\nu\in \cM: \nu(|x|)\leq \fR\}$ for any $\fR>0$}. If we view $H_\mu^{\mathsf D}(\nu)$ as a function of $T_\nu$, i.e. $H_\mu^{\mathsf D}(T_\nu)\deq H_\mu^{\mathsf D}(\nu)$, then it is concave.

\item If $\int_0^1 (T_\mu(x)-T_{\mu_B}(x))\rd x\neq 0$, or there exists some $0<y<1$ such that 
\begin{align}\label{e:domain1}
\int_y^1(T_\mu(x)-T_{\mu_B}(x))\rd x>0, 
\end{align}
then $\mathcal I^{\mathsf D}(\mu)=+\infty$.

\item If there exists some small constant $\fc>0$
\begin{align}\begin{split}\label{e:domain2}
&\int_y^1 \left(T_{\mu}(x) -T_{\mu_{B}}(x)\right)\rd x\leq 
\left\{\begin{array}{ll}
 -\fc y,      &\text{ for } 0\leq y\leq \fc,\\
 -\fc,         &\text{ for } \fc\leq y\leq 1-\fc,\\
 -\fc(1-y),  &\text{ for } 1-\fc\leq y\leq 1,
 \end{array}\right.\\
\end{split}\end{align}
 then
$\mathcal I^{\mathsf D}(\mu)=H^{\mathsf D}_\mu(\nu^*)<\infty$ for some probability measure $\nu^{*}$ such that  $\nu^*(|x|)<\infty$. 

\item $\mathcal I^{\mathsf D}(\cdot)$ is nonnegative and  lower semicontinuous on $\cM([-\fK, \fK])$ (hence it is a good rate function). It vanishes only at the Dirac mass at  
$\int x\rd \mu_B$.

\item For any measure $\mu$ in the admissible set $\cA_{\mu_B}$ as defined in \eqref{e:limitadm}, there exists a sequence of measures $\mu^{\varepsilon}$ inside the region as given in \eqref{e:domain2}, converging to $\mu$ in the weak topology and 
$\lim_{\varepsilon\rightarrow 0}\mathcal I^{\mathsf D}(\mu^{\varepsilon})=\mathcal I^{\mathsf D}(\mu)$.
\end{enumerate}
\end{proposition}

\begin{proof}
 For Item 1,
unfortunately, $H_\mu^{\mathsf D}(\cdot)$ is not continuous in the weak topology, it is only continuous in the Wasserstein metric. In the following we show that $H_\mu^{\mathsf D}(\cdot)$ is upper semi-continuous in the weak topology on $\{\nu\in \cM: \nu(|x|)\leq \fR\}$ for any $\fR>0$.
Given a probability measure $\nu$ we denote  the truncated measure 
$\nu^\delta=\nu\bm1(|x|\leq \delta^{-1})+\delta_0\int_{|x|>\delta^{-1}}\rd \nu $. 

\begin{claim}\label{c:HDbound}
If $\mu$ is a probability measure  supported on $[-\fK,\fK]$ which satisfies \eqref{e:limitadm}, for any probability measure $\nu$ with $\nu(|x|)\leq \fK$,  it holds
\begin{align*}
H_\mu^{\mathsf D}(\nu)
\leq H_\mu^{\mathsf D}(\nu^\delta)+C_{\fK}\oo_\delta(1),
\end{align*}
where the implicit error $\oo_\delta(1)$ is independent of the measure $\nu$.
\end{claim}
\begin{proof}
We recall the definition of $H_\mu^{\mathsf D}(\nu)$ from \eqref{e:defHD}
\begin{align*}\begin{split}
H_\mu^{\mathsf D}(\nu)
&=\frac{1}{2}\int T_\nu(x)T_\mu(x)\rd x-I(\nu,\mu_B)\\
&=\frac{1}{2}\int T_\nu(x)T_\mu(x)\rd x-\left(\frac{1}{2}\int_{|T_\nu|>1/\delta} T_\nu(x) T_{\mu_B}(x)\rd x +I(\nu^\delta,\mu_B)+C_\fK\oo_\delta(1)\right)\\
&=\frac{1}{2}\int_{|T_\nu|\leq1/\delta} T_\nu(x)T_\mu(x)\rd x-I(\nu^\delta,\mu_B)+\frac{1}{2}\int_{|T_\nu|>1/\delta} T_\nu(x) (T_\mu(x)-T_{\mu_B}(x))\rd x+C_\fK\oo_\delta(1)\\
&\leq \frac{1}{2}\int_{|T_\nu|\leq1/\delta} T_\nu(x)T_\mu(x)\rd x-I(\nu^\delta,\mu_B)+C_{\fK}\oo_\delta(1)\\
&=\frac{1}{2}\int T_{\nu^\delta}(x)T_\mu(x)\rd x-I(\nu^\delta,\mu_B)+C_\fK\oo_\delta(1)=H_\mu^{\mathsf D}(\nu^\delta)+C_\fK\oo_\delta(1),
\end{split}\end{align*}
where we used Proposition \ref{p:spbound2} in the second line. In the fourth line, we used  \eqref{e:limitadm} and the fact that $x\rightarrow T_\nu(x) 1_{|T_\nu(x)|>1/\delta } $ is increasing to show that the last term in the third line is non-positive.
Finally, we used that  $|T_\mu(x)|\leq \fK$ in the last line.
\end{proof}

Let $\{\nu_n\}_{n\geq 1}$ be 
 a sequence of probability measures with $\nu_n(|x|)\leq \fR$ converging weakly to $\nu$. Take $\delta>0$ sufficiently small,  such that $\nu(\{\delta^{-1}, -\delta^{-1}\})=0$. It is easy to see that $\nu^{\delta}$ converges to $\nu$ in Wasserstein metric as $\delta\rightarrow0$. As a consequence, we get
\begin{align}\label{e:truncate}
H_\mu^{\mathsf D}(\nu)
=H_\mu^{\mathsf D}(\nu^\delta)+\oo_{\delta,\nu}(1),
\end{align}
where for any fixed measure $\nu$ with $\nu(|x|)\leq \fR$, $\oo_{\delta,\nu}(1)$ goes to zero as $\delta$ goes to zero. 
Moreover, we have that $\nu_n^\delta$ converges to $\nu^\delta$ in Wasserstein distance. Thus it gives
\begin{align}\label{e:limitdelta}
\limsup_{n\rightarrow\infty}H_\mu^{\mathsf D}(\nu_n^\delta)= H_\mu^{\mathsf D}(\nu^\delta).
\end{align}
It follows from combining \eqref{e:truncate}, Claim \ref{c:HDbound} and \eqref{e:limitdelta},
\begin{align*}\begin{split}
\limsup_{n\rightarrow\infty}H_\mu^{\mathsf D}(\nu_n)
&\leq \limsup_{n\rightarrow\infty}H_\mu^{\mathsf D}(\nu^\delta_n)+C_{\fK\vee\fR}\oo_\delta(1)\\
&=H_\mu^{\mathsf D}(\nu^\delta)+C_{\fK\vee\fR}\oo_\delta(1)
=H_\mu^{\mathsf D}(\nu)+C_{\fK\vee\fR}\oo_\delta(1)+\oo_{\delta,\nu}(1),
\end{split}\end{align*}
By sending $\delta$ to $0$ in the above estimate, we get that
\begin{align*}
\limsup_{n\rightarrow\infty}H_\mu^{\mathsf D}(\nu_n)
\leq H_\mu^{\mathsf D}(\nu),
\end{align*}
and the upper semicontinuity of $H_\mu^{\mathsf D}$ follows.

Both $\int T_\nu(x)T_\mu(x)\rd x$ and $-I(T_\nu, T_{\mu_B})$ are concave, so is $H_\mu^{\mathsf D}(T_\nu)$.

For Item 2, given any measure $\mu_Y$, we denote its dilation by a factor $L$ as $\mu_{LY}=L_\#\mu_Y$, then $T_{\mu_{LY}}=LT_{\mu_Y}$. Thanks to Proposition \ref{p:spbound},
for any $\varepsilon>0$, there exists a constant $C(\varepsilon)$,
\begin{align}\label{e:ratebound1}
\frac{L}{2}\int (T_\mu-T_{\mu_B})T_{ \mu_{Y}}\rd x\leq  H^{\mathsf D}_{\mu}(\mu_{LY})\leq  \frac{L}{2}\int (T_\mu-T_{\mu_{B}})T_{ \mu_{Y}}\rd x +L\OO(\varepsilon)\mu_{Y}(|x|)+C(\varepsilon).
\end{align}
If $\int_0^1 (T_\mu-T_{\mu_B})\rd x\neq 0$, we can take $\mu_Y=\delta_1$, then $T_{\mu_Y}={\bf1}_{[0,1]}$, and
\begin{align*}
\mathcal I^{\mathsf D}(\mu)\geq \lim_{L\rightarrow \infty}\max\{H_\mu^D(\mu_{ LY}),H_\mu^D(\mu_{-LY})\}
= \lim_{L\rightarrow \infty} \frac{L}{2}\left|\int_0^1 (T_\mu-T_{\mu_B})\rd x\right|=+\infty.
\end{align*}
If \eqref{e:domain1} holds for some $0<y<1$, we can take $\mu_Y=y\delta_0+(1-y)\delta_{1/(1-y)}$, then $T_{\mu_Y}={\bf 1}_{[y,1]}/(1-y)$, and 
\begin{align*}
\mathcal I^{\mathsf D}(\mu)\geq \lim_{L\rightarrow \infty}H_\mu^D(\mu_{LY})
= \lim_{L\rightarrow \infty} \frac{L}{2(1-y)}\int_y^1 (T_\mu-T_{\mu_B})\rd x=+\infty.
\end{align*}

For Item 3, we remark that the function $\nu\mapsto H^{\mathsf D}_\mu(\nu)$ is translation invariant when $\int_0^1 (T_\mu-T_{\mu_B})\rd x=0$. If $\nu$ has finite first moment, we can always translate $\nu$ to make $\int x\rd\nu=0$. In the rest of the proof, we will restrict ourselves to the set of measures in $\cM$ with mean zero.  We will first show that  \eqref{e:domain2}
implies that there exists a small $\delta>0$ and a large $L^*>0$ such that for any $\mu_Y\in \cM$ such that $\int |x|\rd \mu_Y= 1$, $\int x\rd \mu_Y=0$ and any $L\geq L^*$, then $ H^{\mathsf D}_{\mu}(\mu_{LY})\leq -\delta L$. Moreover, note that the set of measures $\mu_Y$ such that $T_{\mu_Y}$ is differentiable is dense in $\cM$. Hence,  by continuity of $ H^{\mathsf D}_\mu$, we may assume that $\mu_Y$ 
is such that $T_{\mu_Y}$ is differentiable. Given  such a  $\mu_Y$, integration by parts yields

$$\int_0^1 (T_\mu-T_{\mu_{B}})T_{ \mu_Y}\rd x=\int_0^1 T_{\mu_Y}'(y) \int_{y}^{1} (T_\mu(x)-T_{\mu_{B}}(x)) \rd x \rd y\,.$$
Since $T_{\mu_Y}$ is non-decreasing, $T_{\mu_Y}'$ is non-negative, and we deduce from \eqref{e:domain2} that 
\begin{equation}\label{jh}\int_0^1 (T_\mu-T_{\mu_{B}})T_{ \mu_Y}\rd x
\leq-\fc\int_0^1 ( y\bm 1_{[0,\fc]}(y) + \bm 1_{[\fc,1-\fc]}(y)+(1-y)\bm 1_{[1-\fc,1]}(y)) T_{\mu_Y}'(y) \rd y.
\end{equation}
On the other hand, because $\int_0^1|T_{\mu_Y}|(x)\rd x=1$ and $\int_0^1 T_{\mu_Y}(x)\rd x=0$, and thanks to  the smoothness and the monotonicity  of $T_{\mu_Y}$, we know that there exists $y_0\in [0,1]$ such that
$T_{\mu_Y}(y_0)=0$ and  $T_{\mu_Y}(y)\leq 0$ for $0\leq y\leq y_0$, 
$T_{\mu_Y}(y)\geq 0$ for $y_0\leq y\leq 1$.
Then we have $\int_{0}^{y_0}T_{\mu_{Y}}(y)\rd y=-1/2$ and $\int_{y_0}^{1}|T_{\mu_{Y}}(y)|\rd y=1/2$.
By an integration by part, we conclude
\begin{align}\label{e:cc0}
-\int_{0}^{y_0}T_{\mu_{Y}}(y)\rd y= \int_0^{y_0} y T_{\mu_Y}'(y)\rd y=\frac{1}{2},\quad 
\int_{y_0}^{1}T_{\mu_{Y}}(y)\rd y=\int_{y_0}^1 (1-y) T_{\mu_Y}'(y)\rd   y =\frac{1}{2}.
\end{align}
For $y_0\geq \fc$, we have for $y\in [0,1]$
\begin{align}\label{e:cc1}
y\bm 1_{[0,\fc]}(y) + \bm 1_{[\fc,1-\fc]}(y)+(1-y)\bm 1_{[1-\fc,1]}(y)\geq   (1-y) 1_{[y_0,1]}(y).
\end{align}
For $y_0\leq 1-\fc$, we have for $y\in [0,1]$
\begin{align}\label{e:cc2}
y\bm 1_{[0,\fc]}(y) + \bm 1_{[\fc,1-\fc]}(y)+(1-y)\bm 1_{[1-\fc,1]}(y)\geq   y 1_{[0,y_0]}(y).
\end{align}
We get from plugging \eqref{e:cc0}, \eqref{e:cc1} and \eqref{e:cc2} into \eqref{jh}, the following upper bound
\begin{equation}\label{e:yconstraint}
\int_0^1 (T_\mu-T_{\mu_{B}})(x)T_{ \mu_Y}(x)\rd x
\leq -\frac{\fc}{2}.
\end{equation}
We use \eqref{e:ratebound1} and \eqref{e:yconstraint} to estimate $H_\mu^{\mathsf D}(\mu_{LY})$. As a consequence, if we take $\varepsilon$ much smaller than $\fc$, \eqref{e:ratebound1} implies that there exists a small $\delta>0$ and a large $L^*>0$ (depending only on $\fc$) such that for any $L\geq L^*$, it holds $ H^{\mathsf D}_{\mu}(\mu_{LY})\leq -\delta L$.
%
We  conclude that 
\begin{align*}
\sup_{\nu\in \cM} H^{\mathsf D}_\mu(\nu)=\sup_{\mu_Y: \int |x|\rd \mu_Y\leq L^*}H^{\mathsf D}_\mu(\mu_Y)<\infty,
\end{align*}
and the supremum is achieved at some $\nu^*$ with $\int |x|\rd \nu^*\leq L^*$, since $\{\mu_Y: \int |x|\rd \mu_Y\leq L^*\}$ is compact and $H^D_\mu$ is upper semicontinuous.

For Item 4, since $(\mu, \nu)\mapsto H_\mu^{\mathsf D}(\nu)$ is continuous in $\mu$, $\cI^{\mathsf D}(\mu)=\sup_{\nu\in \cM}H_\mu^{\mathsf D}(\nu)$ is lower semicontinuous. Moreover $\cI^{\mathsf D}(\mu)\geq H_\mu^{\mathsf D}(\delta_0)=0$, so $\cI^{\mathsf D}(\cdot)$ is nonnegative.

If $\mathcal I^{\mathsf D}(\mu)=0$, then $\int x\rd \mu=\int x\rd \mu_B$ and $H_\mu^D(\nu)\leq 0$ for all probability measures $\nu\in \cM$,
\begin{equation}\label{po}
\frac{1}{2}\int T_{\mu}T_{\nu}\rd x\leq I(\nu,\mu_{B}).
\end{equation}
We denote $\nu_\varepsilon=\varepsilon_\#\nu$  the pushforward of $\nu$ by the homothety of factor $\varepsilon$, and then $T_{\nu_\varepsilon}=\varepsilon T_\nu$. \cite[Theorem 0.1]{CGM} implies that for $\varepsilon>0$ small enough
$$I(\nu_{\varepsilon},\mu_{B})= \frac{\varepsilon}{2}\int x\rd \mu_{B}\int x\rd\nu +O(\varepsilon^{2})\,.$$
Hence, we deduce from \eqref{po} by replacing $\nu$ by $\nu_{\varepsilon}$  and sending $\varepsilon$  to zero that 
$$\int T_{\mu}(x)T_{\nu}(x)\rd x\leq  \int x\rd \mu_{B}\int x\rd\nu = \int x\rd \mu_{B}\int T_{\nu}\rd x,$$
or equivalently for any probability measure $\nu\in \cM$
$$\int \left(T_{\mu}(x)-\int x\rd \mu_{B}\right)T_{\nu}(x)\rd x\leq 0.$$
Taking $T_{\nu}={\bf1}_{\{x: T_{\mu}(x)\geq \int x\rd \mu_{B}\}}$, we deduce that
$T_{\mu}(x)\leq \int x\rd \mu_{B}$ almost surely.
On the other hand, $\int T_{\mu}\rd x=\int x\rd \mu_{B}$, and therefore $T_{\mu}=\int x\rd \mu_{B}$ almost surely. We conclude that if $\cal I^D(\mu)=0$ then $T_\mu=\int x\rd \mu_{B}$ almost surely and $\mu$ is the delta mass at $\int x\rd \mu_{B}$.

{ Finally for the second point of Item 5, we pick $\mu\in \mathcal A_{\mu_B}$ and construct $\mu^\varepsilon$ satisfying \eqref{e:domain2}  
converging to $\mu$ when $\varepsilon$ goes to zero.
If $\mu$ 
 is not a delta mass, we have for small  enough $\varepsilon>0$, $T_\mu(1-\varepsilon)>T_\mu(\varepsilon)+2\varepsilon. $
We  take $$T_{\tilde \mu^\varepsilon}(y)=\left\{
\begin{array}{ll}
T_\mu(y)+\varepsilon &\mbox{ for }y\in[0,\varepsilon], \cr
T_\mu(\varepsilon)+\varepsilon &\mbox{ for } y\in   [\varepsilon,\varepsilon_{1}],\cr
T_{\mu}(y)  &\mbox{ for }y\in [\varepsilon_{1},\varepsilon_{2}],\cr
T(1-\varepsilon)-\varepsilon &\mbox{ for } y\in [\varepsilon_{2},1-\varepsilon],\cr
T_{\mu}(y)-\varepsilon &\mbox{ for }y\in[1-\varepsilon,1],\cr
\end{array}
\right.$$
where $$\varepsilon_{1}=\sup\{ x>\varepsilon : T_\mu(\varepsilon)+\varepsilon\geq T_\mu(x)\},\quad \varepsilon_{2}=\sup\{ x: T_{\mu}(x)\leq  T_{\mu}(1-\varepsilon)-\varepsilon\}\,.$$
Then we get $T_{\mu^{\varepsilon}}$ by shifting $T_{\tilde \mu^{\varepsilon}}$ such that its first moment is the same as $T_\mu$:
$$T_{\mu^{\varepsilon}}(y):=T_{\tilde \mu^{\varepsilon}}(y)-\Gamma,\quad \Gamma:=\int_{0}^{1}(T_{\tilde \mu^{\varepsilon}}(y)-T_{\mu}(y))\rd y \,.$$
We can check $\Gamma=\int_{\varepsilon}^{\varepsilon_{1}}(T_{\mu}(\varepsilon)+\varepsilon-T_{\mu}(y))\rd y+\int_{\varepsilon_{2}}^{1-\varepsilon}(T_{\mu}(1-\varepsilon)-\varepsilon-T_{\mu}(y))\rd y$.  On the interval $[\varepsilon, \varepsilon_1]$, $T_\mu$ is non decreasing, it holds that on this interval $0\leq T_\mu(\varepsilon)+\varepsilon-T_\mu(y)\leq \varepsilon$. Therefore $0\leq \int_{\varepsilon}^{\varepsilon_{1}}(T_{\mu}(\varepsilon)+\varepsilon-T_{\mu}(y))\rd y\leq \varepsilon(\varepsilon_1-\varepsilon)$. Similarly we have 
$-\varepsilon(1-\varepsilon-\varepsilon_2)\leq \int_{\varepsilon_{2}}^{1-\varepsilon}(T_{\mu}(1-\varepsilon)-\varepsilon-T_{\mu}(y))\rd y\leq 0$. As a consequence, we have $|\Gamma|\leq \varepsilon-\varepsilon^2$, and also  that $\mu^{\varepsilon}$ goes to $\mu$ as $\varepsilon$ goes to zero. 

We claim that $\mu^{\varepsilon}$ satisfies \eqref{e:domain2}. We denote
 for all $y\in [0,1]$,
\begin{equation}\label{sd}
\varphi(y):=\int_y^1 (T_\mu(x)-T_{\mu^\varepsilon}(x))\rd x\leq \int_y^1 (T_{\mu_B}(x)-T_{\mu^\varepsilon}(x))\rd x\,.
\end{equation}
From the construction $\varphi(0)=\varphi(1)=0$, and $\varphi(y)$ first decreases then increases on $[0,1]$. Moreover, for $y\in[0,\varepsilon]$,
$$\varphi(y)=-\int_{0}^{y}(T_\mu(x)-T_{\mu^\varepsilon}(x))\rd x
=-\int_{0}^{y}[(T_\mu(x)-T_{\tilde\mu^\varepsilon}(x))+(T_{\tilde\mu^\varepsilon}(x)-T_{\mu^\varepsilon}(x))]\rd x=(\varepsilon-\Gamma)y.$$
For $y\in[1-\varepsilon, 1]$, we have
$$\varphi(y)=\int_y^1 \left((T_\mu(x)-T_{\tilde \mu^\varepsilon}(x))+(T_{\tilde\mu^\varepsilon}(x)-T_{\mu^\varepsilon}(x))\right)\rd x= (\varepsilon+\Gamma)(1-y).$$
And for $y\in[\varepsilon, 1-\varepsilon]$, 
\begin{align*}
\varphi(y)=\int_y^1 (T_\mu(x)-T_{\mu^\varepsilon}(x))\rd x\geq \varepsilon^2-|\Gamma|\varepsilon.
\end{align*}
Therefore, $\mu^{\varepsilon}$ satisfies \eqref{e:domain2} with $\fc=\varepsilon^2-|\Gamma|\varepsilon\geq \varepsilon ^3>0$.
We finally prove that $\mathcal I^{\mathsf D}(\mu^{\varepsilon})$ goes to $\mathcal I^{\mathsf D}(\mu)$ as $\varepsilon$ goes to zero.
By lower semi-continuity of $\mathcal I^{\mathsf D}$, we already know that \begin{align}
\label{e:lowb}\mathcal I^{\mathsf D}(\mu)\leq  \liminf_{\varepsilon\rightarrow 0}\mathcal I^{\mathsf D}(\mu^{\varepsilon})\,.
\end{align}
For the converse bound, note that for all $\nu\in\mathcal M$, integration by parts and \eqref{sd} imply that
$$\int T_\nu(T_\mu-T_{\mu^\varepsilon})(x)\rd x =\int T_\nu'(y)\int_y^1 (T_\mu(x)-T_{\mu^\varepsilon}(x))\rd x \rd y\geq 0,$$
which results in
$$H^D_{\mu}(\nu)=\int T_\nu(x)T_\mu(x)dx-I(\nu,\mu_B)\geq \int T_\nu(x)T_{\mu^\varepsilon}(x)dx-I(\nu,\mu_B)=H^D_{\mu^\varepsilon}(\nu)\,.$$
As a consequence, we have
$$\mathcal I^D(\mu)\geq \mathcal I^D(\mu^\varepsilon),$$
and therefore 
\begin{align}\label{e:upbodd}
\mathcal I^{\mathsf D}(\mu)\geq \limsup_{\varepsilon\rightarrow 0}\mathcal I^{\mathsf D}(\mu^{\varepsilon})\,.
\end{align}
The claim follows from combining \eqref{e:lowb} and \eqref{e:upbodd}
.}

\end{proof}


\subsection{Large deviation upper bound}
In this section we prove the large deviation upper bound in Theorem \ref{t:UBU}.
We first notice that if $\Delta$ is the simplex $\{\bmy_N\in \mathbb R^{N}:y_1\geq y_2\geq \cdots \geq y_N\}$, since the law of $U$ is permutation invariant as well as $\mu_{N}$,
$$\bP_N(\mu_{N}\in .)=N!\,\bP_{N}\left( \{\mu_{N}\in .\}\cap \{((UB_NU^{*})_{ii})_{1\leq i\leq N}\in \Delta\}\right)\,.$$
We estimate the probability of a small $\delta$-neighborhood $\bB_\delta(\mu)$ of $\mu$, by tilting the measure as follows:
\begin{align}\begin{split}\label{e:weight1}
\bP_N(\bB_\delta(\mu))
=N!\,\bE\left[\bm 1(\{\mu_{N}\in \bB_{\delta}(\mu)\}\cap \{((UB_NU^{*})_{ii})_{1\leq i\leq N}\in \Delta\})\frac{ \exp\{(\beta /2)N\Tr(Y_NUB_NU^*)\}}{ \exp\{(\beta /2)N\Tr(Y_NUB_NU^*)\}}\right],
\end{split}\end{align}
where $Y_N=\diag\{y_1,y_2,\cdots, y_N\}$ is a sequence of diagonal matrices, with $\bmy_N\in \Delta$ and its spectral measure converging in Wasserstein distance \eqref{e:wd} towards $\mu_Y\in \cM$ (we can take $y_1, y_2, \cdots, y_N$ the $N$-quantiles of the measure $\mu_Y$).
We notice that when $\mu_{N}$ is   in the neighborhood $\bB_{\delta}(\mu)$, and the diagonal entries of $(UB_{N}U^{*})$ 
and $Y_N$ are both in $\Delta$, the integrand of the spherical integral is approximately
\begin{align}\label{e:integrand1}
\exp\{(\beta /2)N\Tr (Y_NUB_NU^*)\}= \exp\left\{\frac{\beta N^2}{2}\left( \int  T_{\mu_{Y_N}}T_{\mu}\rd x+\oo_{\delta}(1)\right)\right\}.
\end{align}
The estimates \eqref{e:weight1} and \eqref{e:integrand1} give a large deviation upper bound for the random measure $\mu_N$ as follows
\begin{align}\begin{split}\label{e:upUBU1}
\bP_N(\bB_\delta(\mu))
&=N!\,\bE\left[\bm 1(\{\mu_{N}\in \bB_{\delta}(\mu)\}\cap \{((UB_N U^{*})_{ii})_{1\leq i\leq N}\in \Delta\}))\frac{ \exp\{(\beta /2)N\Tr(Y_NUB_NU^*)\}}{ \exp\{(\beta /2)N\Tr(Y_NUB_NU^*)\}}\right]\\
&= N! \exp\left\{-\beta N^2\left(\frac{1}{2}\int   T_{ \mu_{Y_N}}T_{\mu}\rd x+\oo_{\delta}(1)\right)\right\}\bE\left[\bm 1(\mu_N\in \bB_\delta(\mu))\exp\{(\beta /2)N\Tr(Y_NUB_NU^*)\}
\right]\\
&\leq 
N! \exp\left\{-\beta N^2\left(\frac{1}{2}\int  T_{\mu_Y}T_{\mu} \rd x+\oo_{\delta}(1)+\oo_N(1)\right)\right\}\bE\left[\exp\{(\beta/2)N\Tr(Y_NUB_NU^*)\}\right]\\
 &= N!\exp\left\{-\beta N^2\left(\frac{1}{2}\int   T_{\mu_Y}T_{\mu}\rd x- I (\mu_Y,\mu_B)+\oo_{\delta}(1)+\oo_N(1)\right)\right\}.
\end{split}
\end{align}
It follows by taking the large $N$ limit, then $\delta$ going to zero and taking the infimum on the right hand side of \eqref{e:upUBU1}, we get the following large deviation upper bound
\begin{align}\label{e:LDPu1}
\limsup_{\delta\rightarrow 0}\limsup_{N\rightarrow 0}\frac{1}{\beta N^2}\log \bP_N(\bB_\delta(\mu))\leq -\sup_{\mu_Y\in \cM} H^{\mathsf D}_\mu(\mu_Y)=-\mathcal I^{\mathsf D}(\mu)\,.
\end{align}

\subsection{Large deviation lower bound}
In this section we derive the large deviation lower bound for the empirical measure of the diagonal entries of $UB_NU^*$, which matches the upper bound \eqref{e:LDPu1}. The large deviation lower bound follows from combining the following Propositions \ref{p:unique1} and \ref{p:lowerbound1}.

\begin{proposition}\label{p:unique1}

We assume the assumptions of Theorem \ref{t:UBU}.
For any probability measure $\mu_Y\in \cM$,  there exists a unique $\mu$ supported on $[-\fK, \fK]$ such that 
\begin{align}\label{e:choicemuYB}
\mu_Y\in \arg\sup_{\nu\in \cM} H^{\mathsf D}_\mu(\nu), \quad H^{\mathsf D}_\mu(\nu)=\frac{1}{2}\int  T_{\nu} T_\mu\rd x-I(\nu, \mu_B).
\end{align}
%
Here, $T_\mu$ is uniquely determined by $T_Y$ by
\begin{align*}
T_\mu=\tau(\mathsf b|\mathsf y)\circ T_Y.
\end{align*}
Here, $\tau(\mathsf b|\mathsf y)$ is the conditional expectation of $\mathsf b$ knowing $\mathsf y$ under the non-commutative distribution $\tau=\tau_{\mu_B, \mu_Y}$ uniquely associated to $(\mu_B,\mu_Y)$ as in \eqref{e:tauF} Theorem \ref{t:convl1}.
\end{proposition}
Observe that the above shows that $\tau(\sfb|\sfy)\circ T_Y$ is non-decreasing for any $\mu_Y$: this is coherent with the fact that we can always reorder the $(UB_N U^*)_{ii}$ up to neglectable factors $N!$ and that the density of the tilted measure is maximal when the $A_{ii}$ and $(UB_N U^*)_{ii}$ are both increasing.

\begin{proposition}\label{p:lowerbound1}
We assume the assumptions of Theorem \ref{t:UBU}. For any probability measure $\mu_Y\in \cM$,  let  $\mu$ be the unique measure supported on $[-\fK, \fK]$ so that 
\begin{align*}
\mu_Y\in \arg\sup_{\nu\in \cM} H^{\mathsf D}_\mu(\nu), \quad H^{\mathsf D}_\mu(\nu)=\frac{1}{2}\int  T_{\nu} T_\mu\rd x-I(\nu, \mu_B).
\end{align*}
Then we have
\begin{align}\label{e:LDPl}
\liminf_{\delta\rightarrow 0}\liminf_{N\rightarrow \infty}\frac{1}{\beta N^2}\log \bP_N(\bB_\delta(\mu))\geq -H_\mu^{\mathsf D}(\mu_Y)=-\mathcal I^{\mathsf D}(\mu).
\end{align}
\end{proposition}


\begin{proof}[Proof of Theorem \ref{t:UBU}]
Item 1 of Theorem \ref{t:UBU} follows from Proposition \ref{p:domain1}.
For Item 2, the large deviation upper bound follows from \eqref{e:LDPu1}. If $\mu$ does not satisfy $\int_0^1 (T_\mu(x)-T_{\mu_B}(x))\rd x\neq 0$ or  the limiting Schur-Horn inequalities \eqref{limSH}, then both sides of \eqref{e:UBU} are $-\infty$. There is nothing to prove. In the following we first prove \eqref{e:UBU} when $\mu$ satisfies $\int_0^1 (T_\mu(x)-T_{\mu_B}(x))\rd x= 0$ and the  strong limiting Schur-Horn inequalities \eqref{e:domain2} with some $\fc>0$.
In this case, thanks to Item 3 in Proposition \ref{p:domain1}, 
 there 
 exists a probability measure $\mu_Y$ such that
$\mathcal I^{\mathsf D}(\mu)=H^{\mathsf D}_\mu(\mu_Y)<\infty$ and $\mu_Y\in \cM$.  
Then Propositions \ref{p:unique1} and \ref{p:lowerbound1} imply that $\mu$ is uniquely 
determined by $\mu_Y$ and the large deviation lower bound holds. This gives the full large deviation principle when  the  strong limiting Schur-Horn inequalities \eqref{e:domain2} hold. 
Next we extend it to the boundary case by a continuity argument. Thanks to Item 5 in Proposition \ref{p:domain1},  for any measure $\mu$ inside the admissible set but not satisfying \eqref{e:domain2}, there exists a sequence of measures $\mu^{\varepsilon}$ inside the region as given in \eqref{e:domain2}, converging to $\mu$ in the weak topology and 
$\lim_{\varepsilon\rightarrow 0}\mathcal I^{\mathsf D}(\mu^{\varepsilon})=\mathcal I^{\mathsf D}(\mu)$.
Then for any $\delta>0$, there exists sufficiently small $\varepsilon>0$ 
\begin{align}\label{e:bbcase}
\liminf_{N\rightarrow \infty}\frac{1}{\beta N^2}\log \bP(\mu_N\in \bB_{\delta}(\mu))\geq 
\liminf_{N\rightarrow \infty}\frac{1}{\beta N^2}\log \bP(\mu_N\in \bB_{\delta/2}(\mu^{\varepsilon}))
=\mathcal I^{\mathsf D}(\mu^{\varepsilon})+\oo_\delta(1).
\end{align}
The large deviation lower bound follows by first sending $\delta$ to zero and then $\varepsilon$ to zero in the right hand side of  \eqref{e:bbcase}.
This finishes the proof of Theorem \ref{t:UBU}.

\end{proof}

The proofs of both Propositions \ref{p:unique1} and \ref{p:lowerbound1} rely on the following probability estimate.
\begin{proposition}\label{c:expbound1}
We assume the assumptions of Theorem \ref{t:UBU}. Fix probability measures $\mu_Y\in \cM$, and $\mu$ supported on $[-\fK, \fK]$.
Let $Y_N=\diag\{y_1,y_2,\cdots,y_N\}$ be a sequence of diagonal matrices whose spectral measures converge in Wasserstein distance \eqref{e:wd} towards $\mu_Y\in \cM$, such that
\begin{align}\label{e:asup1}
\sup_{\nu\in \cM}\left\{\frac{1}{2}\int T_{\nu}T_{\mu} \rd x-I(\nu, \mu_B)\right\}
>\frac{1}{2}\int  T_{\mu_Y}T_{\mu} \rd x-I( \mu_Y,\mu_B)\,,
\end{align}
then there exists a small $\delta>0$, and positive constant $c(\delta)>0$ such that
\begin{align}\label{e:expbound1}\begin{split}
\phantom{{}={}}\bE\left[\bm 1(\mu_N\in \bB_\delta(\mu))\exp\{(\beta/2)N\Tr(Y_NUB_NU^*)\}\right]
\leq e^{-c(\delta)N^2}\bE\left[ \exp\{(\beta/2)N\Tr(Y_NUB_NU^*)\}\right].
\end{split}\end{align}
\end{proposition}

\begin{proof}[Proof of Proposition \ref{c:expbound1}]
Under the assumption \eqref{e:asup1}, for sufficiently small $\varepsilon>0$, there exists a measure $\nu\in \cM$ such that
\begin{align}\label{e:epserror}
\frac{1}{2}\int  T_{\nu}T_\mu\rd x-I(\nu,\mu_B)
\geq\frac{1}{2}\int  T_{\mu_Y}T_{\mu}\rd x-I( \mu_Y,\mu_B)+\varepsilon.
\end{align}
The right hand side of \eqref{e:expbound1} is the spherical integral
\begin{align*}\begin{split}
\bE\left[\exp\{(\beta/2)N\Tr(Y_NUB_NU^*)\}\right]=\exp\{\beta N^2(I( \mu_Y,\mu_B)+\oo_N(1))\}.
\end{split}\end{align*}
We divide $\exp\{\beta N^2(I(\mu_Y, \mu_B)+\oo_N(1))\}$ on both sides of \eqref{e:expbound1}, and take a small $\delta>0$ which will be chosen later, 
\begin{align*}\begin{split}
&\phantom{{}={}}\exp\left\{-\beta N^2(I(\mu_Y, \mu_B)+\oo_N(1))\right\}\bE\left[\bm 1(\mu_N\in \bB_\delta(\mu))\exp\{(\beta/2)N\Tr(Y_NUB_NU^*)\}\right]
\\ 
&=
\exp\left\{-\beta N^2\left(I(\mu_Y, \mu_B)-\frac{1}{2}\int  T_{\mu_Y}T_\mu\rd x+\oo_{\delta}(1)+\oo_N(1)\right)\right\}\bE\left[\bm 1(\mu_N\in \bB_\delta(\mu))\right]\\
&\leq  
\exp\left\{-\beta N^2\left(I(\mu_Y, \mu_B)-\frac{1}{2}\int  T_{\mu_Y}T_\mu\rd x-I(\nu,\mu_B)+\frac{1}{2}\int T_\nu T_\mu \rd x+\oo_{\delta}(1)+\oo_N(1)\right)\right\}\\
&\leq \exp\left\{-\beta N^2(\varepsilon+\oo_{\delta}(1)+\oo_N(1))\right\},
\end{split}
\end{align*}
where in the first inequality we used the large deviation upper bound \eqref{e:LDPu1}, and \eqref{e:epserror} in the last inequality. The claim follows provided we take $\delta$ sufficiently small and $N$ large.
\end{proof}

\begin{proof}[Proof of Proposition \ref{p:unique1}]
We first prove the existence of such $\mu$ by contradiction. If there is no such $\mu$, i.e. for any measure $\mu$ supported on $[-\fK, \fK]$, we have 
\begin{align*}
\mu_Y\not\in\arg\sup_{\nu\in \cM}\left\{\frac{1}{2}\int  T_{\nu}T_\mu\rd x-I(\nu,\mu_B)\right\}.
\end{align*}
It follows that for any measure $\mu$ supported on $[-\fK, \fK]$, it holds
\begin{align*}
\sup_{\nu\in \cM}\left\{\frac{1}{2}\int  T_{\nu}T_\mu\rd x-I( \nu,\mu_B)\right\}
>\frac{1}{2}\int T_{\mu_Y}T_\mu \rd x-I( \mu_Y,\mu_B).
\end{align*}
Then  Proposition \ref{c:expbound1} implies that there exists a small $\delta>0$, and positive constant $c(\delta)>0$ such that
\begin{align*}\begin{split}
\phantom{{}={}}\bE\left[\bm 1(\mu_N\in \bB_\delta(\mu))\exp\{(\beta/2)N\Tr(Y_NUB_NU^*)\}\right]
\leq e^{-c(\delta)N^2}\bE\left[ \exp\{(\beta/2)N\Tr(Y_NUB_NU^*)\}\right].
\end{split}\end{align*}
Since the space of probability measures supported on $[-\fK, \fK]$ is compact, we get a finite open cover $ \cup \bB_{\delta_i}(\mu_i)$ of the set of probability measures supported on $[-\fK, \fK]$,  
\begin{align*}\begin{split}
&\phantom{{}={}}\bE\left[\exp\{(\beta/2)N\Tr(Y_NUB_NU^*)\}\right]=
\bE\left[\sum_i \bm 1(\mu_N\in \bB_{\delta_i}(\mu_i))\exp\{(\beta/2)N\Tr(Y_NUB_NU^*)\}\right]\\
&\leq \sum_ie^{-c(\delta_i)N^2}\bE\left[\exp\{(\beta/2)N\Tr(Y_NUB_NU^*)\}\right]<\bE\left[\int \exp\{(\beta/2)N\Tr(Y_NUB_NU^*)\}\right],
\end{split}\end{align*}
for sufficiently large $N$. This gives a contradiction.

In the following we prove the uniqueness of such measure $\mu$ satisfying \eqref{e:choicemuYB}.
Since $\mu_Y$ is one of the maximizer, then for any $\varepsilon>0$,
\begin{align*}\begin{split}
\frac{1}{2}\int  T_{Y} T_\mu\rd x-I(T_Y, T_B)
&\geq \frac{1}{2}\int  (T_Y+\varepsilon \tilde T_C)_\#(\unif[0,1]) T_\mu\rd x-I(T_Y+\varepsilon \tilde T_C, T_B)\\
&\geq \frac{1}{2}\int  (T_{Y}+\varepsilon \tilde T_C) T_\mu\rd x-I(T_Y+\varepsilon \tilde T_C, T_B),
\end{split}\end{align*}
By rearranging the above expression, and sending $\varepsilon$ to $0$, we have
\begin{align}\label{e:lowerbound}
\left.\del_{\varepsilon}I(T_Y+\varepsilon \tilde T_C, T_B)\right|_{\varepsilon=0}\geq {\frac{1}{2}} \int \tilde T_C T_\mu \rd x.
\end{align}
We will choose $\tilde T_C=f(T_Y)$ with some compactly supported and Lipschitz real-valued function $f$ as in the case \eqref{e:nondelta} or a bounded measurable function supported on $\{x: T_Y(x)=a\}$ as in \eqref{e:delta}. We notice that in both cases if we replace $\tilde T_C$ by $-\tilde T_C$, both sides of \eqref{e:lowerbound} change the sign. Therefore, we conclude that
\begin{align*}
\left.\del_{\varepsilon}I(T_Y+\varepsilon \tilde T_C, T_B)\right|_{\varepsilon=0}={\frac{1}{2}} \int \tilde T_C T_\mu \rd x.
\end{align*}
Now if we choose $\tilde T_C$ in \eqref{e:delta}, i.e. $\tilde T_C$ supported on that $\{x: T_Y(x)=a\}$, we have
\begin{align*}
\left.\del_{\varepsilon}I(T_Y+\varepsilon \tilde T_C, T_B)\right|_{\varepsilon=0}={\frac{1}{2}}\int \tilde T_C(x)\rd x\tau(\mathsf b|\mathsf y)(a)
={\frac{1}{2}}\int \tilde T_C T_\mu \rd x,
\end{align*}
We conclude that $T_\mu(x)=\tau(\mathsf b|\mathsf y)(a)=\tau(\mathsf b|\mathsf y)\circ T_Y(x)$ on $\{x: T_Y(x)=a\}$. Especially, on the intervals where $T_Y$ is a constant, we have $T_\mu(x)=\tau(\mathsf b|\mathsf y)\circ T_Y(x)$.
Next we take $\tilde T_C=f(T_Y)$ as in \eqref{e:nondelta},
\begin{align}\begin{split}\label{e:inc0}
\left.\del_{\varepsilon}I(T_Y+\varepsilon f(T_Y), T_B)\right|_{\varepsilon=0}
&={\frac{1}{2}}\int f(x)\tau(\mathsf b|\mathsf y)(x)\rd \mu_Y\\
&={\frac{1}{2}}\int f(T_Y)\tau(\mathsf b|\mathsf y)\circ T_Y(x)\rd x
={\frac{1}{2}}\int f( T_{Y}) T_\mu \rd x.
\end{split}\end{align}
On the intervals where $T_Y$ is increasing, we conclude from \eqref{e:inc0} that $T_\mu(x)=\tau(\mathsf b|\mathsf y)\circ T_Y(x)$.
Therefore, we  conclude that $T_\mu(x)=\tau(\mathsf b|\mathsf y)\circ T_Y(x)$ almost surely, and this finishes the proof.

\end{proof}

\begin{proof}[Proof of Proposition \ref{p:lowerbound1}]
Thanks to the uniqueness of $\mu$, we have that for any $\mu'\neq \mu$ supported on $[-\fK, \fK]$
\begin{align*}
\mu_Y\not\in \arg\sup_{\nu\in \cM}\left\{\frac{1}{2}\int  T_{\nu}T_{\mu'}\rd x-I(\nu,\mu_B)\right\}.
\end{align*}
As a consequence the assumption in Proposition \ref{c:expbound1} holds
\begin{align*}
\sup_{\nu\in \cM}\left\{\frac{1}{2}\int T_{\nu}T_{\mu'} \rd x-I(\nu,\mu_B)\right\}>\frac{1}{2}\int T_{\mu_Y}T_{\mu'} \rd x-I(\mu_Y,\mu_B).
\end{align*}
It follows from Proposition \ref{c:expbound1} that there exists a small $\delta>0$, and  a positive constant $c(\delta)>0$ such that
\begin{align*}\begin{split}
\bE\left[\bm 1(\mu_X\in \bB_\delta(\mu'))\exp\{(\beta/2)N\Tr(Y_NUB_NU^*)\}\right]\leq e^{-c(\delta)N^2}\bE\left[ \exp\{(\beta/2)N\Tr(Y_NUB_NU^*)\}\right].
\end{split}\end{align*}
Because the complement of the open ball $\bB_\delta(\mu)$ in the  space of probability measures supported  in $[-\fK, \fK]$ is compact, we get a finite open cover $ \cup \bB_{\delta_i}(\mu_i)$,  
\begin{align}\begin{split}\label{e:openball1}
&\phantom{{}={}}\bE\left[\bm 1(\mu_N\in \bB_{\delta}(\mu)) \exp\{(\beta/2)N\Tr(Y_NUB_NU^*)\}\right]\\
&\geq\bE\left[ \exp\{(\beta/2)N\Tr(Y_N UB_NU^*)\}\right]-\bE\left[\sum_i \bm 1(\mu_N\in \bB_{\delta_i}(\mu_i))\exp\{(\beta/2)N\Tr(Y_N UB_NU^*)\}\right]\\
&\geq \left(1-\sum_ie^{-c(\delta_i)N^2}\right)\bE\left[\exp\{(\beta/2)N\Tr(Y_N UB_NU^*)\}\right]\\
&= \left(1-\sum_ie^{-c(\delta_i)N^2}\right)\exp\{\beta N^2(I(\mu_Y, \mu_B)+\oo_N(1))\}.
\end{split}\end{align}
The large deviation lower bound at $\mu$ follows from \eqref{e:weight1} and the estimate \eqref{e:openball1}
\begin{align*}
\begin{split}
\bP_N(\bB_\delta(\mu))&= N! \exp\left\{-\beta N^2\left(\frac{1}{2}\int   T_{ \mu_{Y_N}}T_{\mu}\rd x+\oo_{\delta}(1)\right)\right\}\bE\left[\bm 1(\mu_N\in \bB_\delta(\mu))\exp\{(\beta /2)N\Tr(Y_NUB_NU^*)\}
\right]\\
&\geq \exp\left\{-\beta N^2\left(\frac{1}{2}\int T_{\mu_Y}T_{\mu}\rd x+\oo_{\delta}(1)+\oo_N(1)\right)\right\}\exp\{\beta N^2(I(\mu_Y, \mu_B)+\oo_N(1))\}\\
&= \exp\left\{-\beta N^2\left(\frac{1}{2}\int  T_{\mu_Y}T_\mu\rd x-I(\mu_Y, \mu_B)+\oo_{\delta}(1)+\oo_N(1)\right)\right\}.
\end{split}
\end{align*}
\end{proof}

\section{Large Deviation Estimates for $A+UBU^*$}
We recall from Theorem \ref{t:AUBU}, $A_N, B_N$ are two $N\times N$ diagonal matrices with eigenvalues 
$a_1\geq a_2\geq\cdots\geq a_N$ and $b_1\geq b_2,\cdots\geq b_N$. In this section, we use the spherical integral to study the large deviation of the law $\bP_N$ of the empirical eigenvalue distribution of $A_N+UB_NU^*$,
\begin{align*}
\mu_N=\frac{1}{N}\sum_{i=1}^N \lambda_i(A_N+UB_NU^*),
\end{align*}
and prove Theorem \ref{t:AUBU}.  Besides the relations $\Tr A_N +\Tr B_N =\Tr (A_N+UB_NU^*)$, and the Ky Fan inequalities,
\begin{align}\label{e:KyFan}
\sum_{i=1}^k a_i+\sum_{i=1}^k b_i\geq \sum_{i=1}^k\lambda_i(A_N+UB_NU^*),\quad 1\leq i\leq N,
\end{align}
Horn \cite{MR0140521} had conjectured the form of a set of necessary and sufficient inequalities to be satisfied for the eigenvalues of $A_N+UB_N U^*$.  
After contributions by several authors, see in particular
\cite{MR1654578}, these conjectures were proven by Knutson and Tao \cite{MR1671451,MR1811121}.

\subsection{Study of the rate function}
In the following proposition we study the rate function $\mathcal I^{\mathsf{A+B}}(\cdot)$ from Theorem \ref{t:AUBU}. Clearly, it is a good rate function again by Proposition \ref{p:continuity}. Unfortunately, it does not capture the admissible set for the possible eigenvalues given by Horn's problem.  However it contains the information about the constraints given by the Ky Fan inequalities, i.e. it equals $+\infty$ outside the region described by the limiting Ky Fan inequalities:
 \begin{align}
\label{e:limitadm2}
\int_0^1 \left(T_{\mu} -T_{\mu_{A}}-T_{\mu_{B}}\right)\rd x=0,\quad
\int_y^1 \left(T_{\mu} -T_{\mu_{A}}-T_{\mu_{B}}\right)\rd x\leq0, \quad \forall y\in [0,1].
 \end{align}

\begin{proposition}\label{p:rateprop}
Under the assumptions of Theorem \ref{t:AUBU}, the function $H_\mu^{\mathsf A+\mathsf B}(\cdot)$ and rate function $\mathcal I^{\mathsf A+\mathsf B}(\cdot)$ as defined in Theorem \ref{t:AUBU} satisfy:
\begin{enumerate}
\item For measure $\mu$ satisfies \eqref{e:limitadm2}, $H_\mu^{\mathsf A+\mathsf B}(\cdot)$ is  upper semi-continuous in the weak topology on $\{\nu\in \cM: \nu(|x|)\leq \fR\}$ for any $\fR>0$.

\item If $\int_0^1 \left(T_{\mu} -T_{\mu_{A}}-T_{\mu_{B}}\right)\rd x\neq 0$, or there exists some $0<y<1$ such that 
\begin{align}\label{e:domain0}
\int_y^1 \left(T_{\mu} -T_{\mu_{A}}-T_{\mu_{B}}\right)\rd x>0, 
\end{align}
then $\mathcal I^{\mathsf{A+B}}(\mu)=+\infty$.
\item 
If there exists some small constant $\fc>0$
\begin{align}\begin{split}\label{e:domain}
& \int_y^1 \left(T_{\mu} -T_{\mu_{A}}-T_{\mu_{B}}\right)\rd x\leq
\left\{\begin{array}{ll}
 -\fc y,      &\text{ for } 0\leq y\leq \fc,\\
 -\fc,         &\text{ for } \fc\leq y\leq 1-\fc,\\
 -\fc(1-y),  &\text{ for } 1-\fc\leq y\leq 1,
 \end{array}\right.
 \end{split}
\end{align}
then 
$\mathcal I^{\mathsf{A+B}}(\mu)=H^{\mathsf{A+B}}_\mu(\nu^*)<\infty$ for some probability measure $\nu^{*}$ such that  $\nu^*(|x|)<\infty$.

\item The rate function $\mathcal I^{\mathsf A+\mathsf B}(\cdot)$ is nonnegative and lower semicontinuous on $\cM([-2\fK, 2\fK])$. 

\end{enumerate}

\end{proposition}
%
%
\begin{proof}
For Item 1, we know that $H_\mu^{\mathsf D}(\cdot)$ is continuous in the Wasserstein metric. In the following we show that $H_\mu^{\mathsf D}(\cdot)$ is upper semi-continuous in the weak topology on $\{\nu\in \cM: \nu(|x|)\leq \fR\}$ for any $\fR>0$.

\begin{claim}\label{c:HABbound}
Under the assumption \eqref{e:limitadm2}, for any probability measure $\nu$ with $\nu(|x|)\leq \fK$ and probability measure $\mu$ supported on $[-\fK,\fK]$, then setting
$\nu^\delta=\nu\bm1(|x|\leq \delta^{-1})+\delta_0\int_{|x|>\delta^{-1}}\rd \nu $, we have
\begin{align}
H_\mu^{\mathsf A+\mathsf B}(\nu)
\leq H_\mu^{\mathsf A+\mathsf B}(\nu^\delta)+C_{\fK}\oo_\delta(1),
\end{align}
where the implicit error $\oo_\delta(1)$ is independent of the measure $\nu$.
\end{claim}
\begin{proof}
We recall the definition of $H_\mu^{\mathsf A+\mathsf B}(\nu)$ from \eqref{e:defHA+B} and \eqref{e:spbound2}:
\begin{align}\begin{split}\label{appH}
&\phantom{{}={}}H_\mu^{\mathsf A+\mathsf B}(\nu)
=I(\nu,\mu)-I(\nu,\mu_A)-I(\nu,\mu_B)\\
&=\frac{1}{2}\int_{|T_\nu|>1/\delta} T_\nu(T_{\mu}-T_{\mu_A}-T_{\mu_B})\rd x+(I(\nu^\delta,\mu)-I(\nu^\delta,\mu_A)-I(\nu^\delta,\mu_B))+C_\fK\oo_\delta(1)\\
&\leq I(\nu^\delta,\mu)-I(\nu^\delta,\mu_A)-I(\nu^\delta,\mu_B))+C_\fK\oo_\delta(1)
=H_\mu^{\mathsf A+\mathsf B}(\nu^\delta)+C_{\fK}\oo_\delta(1),
\end{split}\end{align}
where we used Proposition \ref{p:spbound2} in the second line, and Assumption \ref{e:limitadm2} in the last line(with the remark that $x\rightarrow 1_{|T_\nu(x)|>1/\delta} T_\nu(x)$ is non-decreasing).
\end{proof}

Let $\{\nu_n\}_{n\geq 1}$ be  a sequence of probability measures  with $\nu_n(|x|)\leq \fR$, converging weakly to $\nu$. Take $\delta>0$ sufficiently small,  such that $\nu(\{\delta^{-1}, -\delta^{-1}\})=0$. Then, $\nu^{\delta}$ converges to $\nu$ in Wasserstein metric as $\delta\rightarrow0$. Therefore for any fixed measure $\nu$ with $\nu(|x|)\leq \fK$, we have
\begin{equation}\label{e:truncate2}
H_\mu^{\mathsf A+\mathsf B}(\nu)
=H_\mu^{\mathsf A+\mathsf B}(\nu^\delta)+\oo_{\delta,\nu}(1).
\end{equation}
Moreover, $\nu_n^\delta$ converges to $\nu^\delta$ in Wasserstein distance, so that
\begin{align}\label{e:limitdelta2}
\limsup_{n\rightarrow\infty}H_\mu^{\mathsf A+\mathsf B}(\nu_n^\delta)= H_\mu^{\mathsf A+\mathsf B}(\nu^\delta).
\end{align}
It follows from combining \eqref{e:truncate2}, Claim \ref{c:HABbound} and \eqref{e:limitdelta2},
\begin{align*}\begin{split}
\limsup_{n\rightarrow\infty}H_\mu^{\mathsf A+\mathsf B}(\nu_n)
&\leq \limsup_{n\rightarrow\infty}H_\mu^{\mathsf A+\mathsf B}(\nu^\delta_n)+C_{\fK\vee\fR}\oo_\delta(1)\\
&=H_\mu^{\mathsf A+\mathsf B}(\nu^\delta)+C_{\fK\vee\fR}\oo_\delta(1)
=H_\mu^{\mathsf A+\mathsf B}(\nu)+C_{\fK\vee\fR}\oo_\delta(1)+\oo_{\delta,\nu}(1).
\end{split}\end{align*}
By sending $\delta$ to $0$ in the above estimate, the upper semicontinuity of $H_\mu^{\mathsf A+\mathsf B}$ follows as
\begin{align*}
\limsup_{n\rightarrow\infty}H_\mu^{\mathsf A+\mathsf B}(\nu_n)
\leq H_\mu^{\mathsf A+\mathsf B}(\nu).
\end{align*}

For Item 2, given any measure $\mu_Y$, we denote its dilation by a factor $L$ as $\mu_{LY}=L_{\#}\mu_Y$, then $T_{\mu_{LY}}=LT_{\mu_Y}$. Thanks to 
Proposition \ref{p:spbound}, for any $\varepsilon>0$, there exists a constant $C(\varepsilon)$ such that 
\begin{align}\label{e:ratebound}
H^{\mathsf{A+B}}_\mu( \mu_{LY})= \frac{L}{2}\int (T_\mu-T_{\mu_A}-T_{\mu_{B}})T_{ \mu_{Y}}\rd x +L\OO(\varepsilon)\mu_{Y}(|x|)+C(\varepsilon).
\end{align}
If $\int_0^1 \left(T_{\mu} -T_{\mu_{A}}-T_{\mu_{B}}\right)\rd x\neq 0$, we can take $\mu_Y=\delta_1$, then $T_{\mu_Y}={\bf 1}_{[0,1]}$ and
\begin{align*}
\cal I^{\mathsf{A+B}}(\mu)
&\geq \lim_{L\rightarrow \infty} 
\max\{H_\mu^{\mathsf{A+B}}(\mu_{LY}),H_\mu^{\mathsf{A+B}}(\mu_{-LY}) \}\\
&\geq \lim_{L\rightarrow \infty} \frac{L}{2}\left|\int (T_\mu-T_{\mu_A}-T_{\mu_{B}})T_{ \mu_{Y}}\rd x\right| +L\OO(\varepsilon)\mu_{Y}(|x|)+C(\varepsilon)=\infty.
\end{align*}
If \eqref{e:domain0} holds for some $0<y<1$, 
we can take $\mu_Y=y\delta_0+(1-y)\delta_{1/(1-y)}$, then $T_{\mu_Y}={\bf 1}_{[y,1]}/(1-y)$, and 
\begin{align*}
\cal I^{\mathsf{A+B}}(\mu)
&\geq \lim_{L\rightarrow \infty} H_\mu^{\mathsf{A+B}}(\mu_{LY})\\
&\geq \lim_{L\rightarrow \infty} \frac{L}{2(1-y)}\int_y^1 (T_\mu-T_{\mu_A}-T_{\mu_B})\rd x+L\OO(\varepsilon)\mu_Y(|x|)+C(\varepsilon)=\infty,
\end{align*}
provided we take $\varepsilon$ small enough.

For Item 3, 
to prove that \eqref{e:domain} implies that there exists $\nu^*\in \cM$ such that $H^{\mathsf{A+B}}_\mu(\nu^*)=\sup_{\nu\in \cM}H^{\mathsf{A+B}}_\mu(\nu)$,
we need to define another functional 
\begin{align*}
\tilde H^{\mathsf{A+B}}_\mu(\nu)=\frac{1}{2}\int T_{\nu}T_{\mu}\rd x-I(\nu, \mu_A)-I(\nu, \mu_B),
\end{align*}
which is an upper bound of $H^{\mathsf{A+B}}_\mu(\nu)$, i.e. $H^{\mathsf{A+B}}_\mu(\nu)\leq \tilde H^{\mathsf{A+B}}_\mu(\nu)$, thanks to Proposition \ref{p:spbound}. 
We remark that the functions $H^{\mathsf{A+B}}_\mu(\nu)$ and $\tilde H^{\mathsf{A+B}}_\mu(\nu)$ are both translation invariant. If $\nu$ has finite first moment and $\int_0^1 (T_\mu-T_{\mu_A}-T_{\mu_B})\rd x=0$, we can always translate $\nu$ to make $\int x\rd \nu=0$. In the rest of the proof, we will restrict ourselves to the set of measures in $\cM$ with mean zero.
Moreover, since $I(\mu_{LY}, \mu_A)$ and $I(\mu_{LY},\mu_B)$ are both convex in $L$, the function $\tilde H^{\mathsf{A+B}}_\mu(\mu_{LY})$ is concave in $L$ for any $\mu_Y$.  We remark that $H^{\mathsf{A+B}}_\mu(\mu_{LY})$ may not be concave in $L$, this is why we need to introduce $\tilde H^{\mathsf{A+B}}_\mu(\nu)$.
We first prove that under \eqref{e:domain}, there exists a small $\delta>0$ and a large $L^*>0$ such that for any $\mu_Y\in \cM$ with $\int |x|\rd \mu_Y= 1$, $\int x\rd \mu_Y=0$ and any $L\geq L^*$, then $\tilde H^{\mathsf{A+B}}_\mu(\mu_{LY})\leq -\delta L$. 
Given  such  a $\mu_Y$, there exists some $y_0\in (0,1)$ such that $T_{\mu_Y}(y_0)=0$, and the same as in \eqref{e:cc0}, we have
\begin{align}\label{e:yconstraint2}
\int_0^1 y\bm 1_{[0,y_0]}(y)T'_{\mu_Y}(y)\rd y=\frac{1}{2},\quad \int_0^1(1-y)\bm 1_{[y_0,1]}(y)T_{\mu_Y}'(y)\rd y=\frac{1}{2}.
\end{align}
Thanks to Proposition \ref{p:spbound}, we have the following estimates for $\tilde H^{\mathsf{A+B}}_\mu(\mu_{LY})$, 
\begin{align}\label{e:live1}
\tilde H^{\mathsf{A+B}}_\mu(\mu_{LY})= \frac{L}{2}\int \left(T_{\mu} -T_{\mu_{A}}-T_{\mu_{B}}\right)T_{\mu_Y}\rd x+L\OO(\varepsilon)\mu_{Y}(|x|)+C(\varepsilon).
\end{align}
Integration by parts yields
\begin{align}\begin{split}\label{e:live2}
&\phantom{{}={}}\int_0^1 (T_{\mu} -T_{\mu_{A}}-T_{\mu_{B}})T_{ \mu_Y}\rd x
=\int_0^1 T_{\mu_Y}'(y) \int_{y}^{1} (T_{\mu}(x) -T_{\mu_{A}}(x)-T_{\mu_{B}}(x)) \rd x \rd y\\
&\leq-\fc\int_0^1 ( y\bm 1_{[0,\fc]}(y) + \bm 1_{[\fc,1-\fc]}(y)+(1-y)\bm 1_{[1-\fc,1]}(y)) T_{\mu_Y}'(y) \rd y\\
&\leq-\fc \min\left\{\int_0^1 (y\bm 1_{[0,y_0]}(y)T_{\mu_Y}'(y)\rd y, \int_0^1(1-y)\bm 1_{[y_0,1]}(y))T_{\mu_Y}'(y)\rd y\right\}=-\fc/2,
\end{split}\end{align}
where we used Assumption \eqref{e:domain},  $T_{\mu_Y}$ is non-decreasing, \eqref{e:yconstraint2},\eqref{e:cc1} and \eqref{e:cc2}. Therefore, if we take $\varepsilon$ much smaller than $\fc$, \eqref{e:live1} and \eqref{e:live2} imply that there exists a small $\delta>0$ and a large $L^*>0$ (depending only on $\fc$) such that for any $L\geq L^*$, it holds $ \tilde H^{\mathsf {A+B}}_{\mu}(\mu_{LY})\leq -\delta L$.

%
%
%
Using Proposition \ref{p:spbound} again, for arbitrarily small $\varepsilon>0$, we have
\begin{align*}
0\leq \tilde H^{\mathsf{A+B}}_\mu(\mu_{LY})- H^{\mathsf{A+B}}_\mu(\mu_{LY})\leq L\OO(\varepsilon) \mu_Y(|x|) +C(\varepsilon),
\end{align*}
Therefore by taking $\varepsilon$ small and $L^*$ large enough, for $L\geq L^*$, $\int |x|\rd \mu_Y=1$ and $\int x\rd \mu_Y=0$ we have by \eqref{e:live1} and \eqref{e:live2}

\begin{align*}
H^{\mathsf{A+B}}_\mu(\mu_{LY})\leq -\delta L/2.
\end{align*}
We conclude that 
\begin{align*}
\sup_{\nu} H^{\mathsf{A+B}}_\mu(\nu)=\sup_{\nu: \int |x|\rd \nu\leq L^*}H^{\mathsf{A+B}}_\mu(\nu)<\infty,
\end{align*}
and the supremum is achieved at some $\nu^*$ with $\int |x|\rd \nu^*\leq L^*$, since $\{\mu_Y: \int |x|\rd \mu_Y\leq L^*\}$ is compact and $H^{\mathsf{A+B}}_\mu$ is upper semicontinuous, thanks to Item 1.

For Item 4, since $(\mu, \nu)\mapsto H_\mu^{\mathsf{A+B}}(\nu)$ is continuous in $\mu$, $\cI^{\mathsf{A+B}}(\mu)=\sup_{\nu\in \cM}H_\mu^{\mathsf{A+B}}(\nu)$ is lower semicontinuous. Moreover $\cI^{\mathsf{A+B}}(\mu)\geq H_\mu^{\mathsf{A+B}}(\delta_0)=0$, so $\cI^{\mathsf{A+B}}(\cdot)$ is nonnegative.
\end{proof}

\subsection{Large deviation upper bound}
In this section we prove the large deviation upper bound in Theorem \ref{t:AUBU}.
We take $Y_N=\diag\{y_1,y_2,\cdots, y_N\}$ a sequence of diagonal matrices, whose spectral measures converge  in Wasserstein distance \eqref{e:wd} towards $\mu_Y$ (we can take $y_1, y_2, \cdots, y_N$ the $N$-quantiles of the measure $\mu_Y$), and denote $\bB_\delta(\mu)$ a $\delta$ neighborhood of $\mu$. We have by Theorem \ref{t:convl1},
\begin{align}\begin{split}\label{e:exppAUBU}
&\phantom{{}={}}\bP_N(\bB_\delta(\mu))
=\bE\left[\bm 1(\mu_N\in \bB_\delta(\mu))\frac{\int \exp\{(\beta/2)N\Tr(Y_NV(A_N+UB_NU^*)V^*)\}\rd V}{\int \exp\{(\beta/2)N\Tr(Y_NV(A_N+UB_NU^*)V^*)\}\rd V}\right]\\
&= e^{-\beta N^2(I(\mu_Y,\mu)+\oo_\delta(1)+\oo_N(1))}\bE\left[\bm 1(\mu_{N}\in \bB_\delta(\mu))\int e^{\frac{\beta N}{2}\Tr(Y_NV(A_N+UB_NU^*)V^*)}\rd V\right],
\end{split}\end{align}
where we used Proposition \ref{p:continuity} that the functional $I(\cdot,\cdot)$ is continuous in Wasserstein distance.
Then we have
\begin{align}\label{e:upboundkk}
\begin{split}
\bP_N(\bB_\delta(\mu))
&=e^{-\beta N^2(I(\mu_Y,\mu)+\oo_\delta(1)+\oo_N(1))}
\bE\left[\bm 1(\mu_N\in \bB_\delta(\mu)) \int e^{\frac{\beta N}{2}\Tr(Y_NV(A_N+UB_NU^*)V^*)}\rd V\right]\\
&\leq e^{-\beta N^2(I(\mu_Y,\mu)+\oo_\delta(1)+\oo_N(1))}\bE\left[ \int e^{\frac{\beta N}{2}\Tr(Y_NV(A_N+UB_NU^*)V^*)}\rd V\right]\\
&= \exp\{-\beta N^2(I(\mu_Y,\mu)-I(\mu_Y,\mu_A)-I(\mu_Y,\mu_B)+\oo_\delta(1)+\oo_N(1))\}.
\end{split}
\end{align}
It follows by taking the large $N$ limit, then $\delta$ going to zero and taking the infimum on the right hand side of \eqref{e:upboundkk}, we get the following large deviation upper bound
\begin{align}\begin{split}\label{e:LDPu}
&\limsup_{\delta\rightarrow 0}\limsup_{N\rightarrow 0}\frac{1}{\beta N^2}\log \bP_N(\bB_\delta(\mu))\leq -\sup_{\mu_Y\in \cM} H^{\mathsf{A+B}}_\mu(\mu_Y)=-\mathcal I^{\mathsf{A+B}}(\mu),\\
&H^{\mathsf{A+B}}_\mu(\mu_Y)=I(\mu_Y,\mu)-I( \mu_Y,\mu_A)-I( \mu_Y,\mu_B),
\end{split}\end{align}
where $\cM$ is the set of probability measures with bounded first moment.


\subsection{Large deviation lower bound}\label{Subsection:4.3}
In this section we prove the large deviation lower bound in Theorem \ref{t:AUBU}.  We notice that without loss of generality we can take $\mu_A$ and $\mu_B$ to have smooth density. Especially, this implies that $\Sigma(\mu_A)$ and $\Sigma(\mu_B)$ are finite. To do it, we can change the 
matrices $A_N$ and $B_N$ into $A^\varepsilon_N,B^\varepsilon_N$  at distance smaller than $\varepsilon$ from $A_N,B_N$ for the operator norm, such that  their empirical eigenvalue distributions converge to $\mu_A^\varepsilon$ and $\mu_B^\varepsilon$ respectively. 
Such matrices $A_N^\varepsilon,B_N^\varepsilon$  can for instance be constructed by adding  deterministic matrices $\varepsilon S^A_N, \varepsilon S^B_N$ with semicircle distribution on the scale $\varepsilon$,  and so that the spectral measures of $A_N^\varepsilon=A_N+\varepsilon S^A_N$ and $B_N^\varepsilon=B_N+\varepsilon S^A_N$ converge towards the free convolution  $\mu_A\boxplus \sigma_\varepsilon$ and $\mu_B\boxplus\sigma_\varepsilon$, where $\sigma_\varepsilon$ is the semi-circle distribution on the scale $\varepsilon$. Under this perturbation, we have
\begin{align*}
\left\| (A_N+UB_NU^*)-(A^\varepsilon_N+UB^\varepsilon_NU^*)\right\|_2=\OO(\varepsilon),
\end{align*}
and
\begin{align*}
\left|(I(\mu_Y,\mu)-I( \mu_Y,\mu_A)-I( \mu_Y,\mu_B))-(I(\mu_Y,\mu)-I( \mu_Y,\mu^\varepsilon_A)-I( \mu_Y,\mu^\varepsilon_B))\right|=\OO(\varepsilon).
\end{align*}
Therefore, a large deviation lower bound for $(A^\varepsilon_N+UB^\varepsilon_NU^*)$ will imply the large deviation lower bound for $(A_N+UB_NU^*)$. In the rest of this section, we assume that $\Sigma(\mu_A)$ and $\Sigma(\mu_B)$ are finite.

Following the inequalities performed in \eqref{e:upboundkk}, to get a  lower bound on  $\bP_N(\bB_\delta(\mu))$, it is enough to show that for any $\delta>0$, there exists $Y_N=\diag\{y_1,y_2,\cdots, y_N\}$ a sequence of diagonal matrices,
 whose spectral measures converge  in Wasserstein distance \eqref{e:wd} towards $\mu_{Y}$ such that 
under the deformed measure
\begin{align}\label{e:lawABY}
\rd\mu_{A_N,B_N,Y_N}(U,V)=\frac{e^{\frac{\beta N}{2}\tr (A_{N}UY_{N} U^{*})+\frac{\beta N}{2}\tr (B_{N}VY_{N} V^{*})}\rd U \rd V}{I_{N}(A_N,Y_N) I_{N}(B_N,Y_N )},
\end{align}
 the spectral measure of  $(U^{*}A_{N}U+V^{*}B_{N}V)$ belongs to $\bB_\delta(\mu)$ with probability at least $1-\oo(1)$ for $N$ large enough. 
 We investigate below such possible limit laws and show that they are described by freeness by amalgamation, as in \eqref{e:defHAB}.

 \subsubsection{ Freeness with Amalgamation}
With the notions and notations introduced in Section \ref{s:prel} (following Theorems \ref{convrho} and \ref{t:FBB}), we
let $\tau_{\sfa, \sfy}$ and $\tau_{\sfb, \sfy}$ be the non-commutative joint distributions of $\{\sfa, \sfy\}$ and of $\{\sfb, \sfy\}$ respectively, $\sfa,\sfb$ being bounded and $\sfy$ in $L^1$. 
 We assume that $(\sfa,\sfb)$ are free with amalgamation over $\sfy$. That is, the joint law of $(\sfa,\sfb,\sfy)$ is the non-commutative distribution defined on the von Neumann algebra generated by functions in $\mathcal F$ of $(\sfa,\sfy)$ and $(\sfb,\sfy)$ as follows (see \cite{voiam}):
It is the law $\tau$ such that $\tau|_{L^\infty\{\sfa, \sfy\}}=\tau_{\sfa,\sfy}$, $\tau|_{L^\infty\{\sfb, \sfy\}}=\tau_{\sfb,\sfy}$, the marginal law of $\sfy$ is given by  $\mu_Y$ and 
for  any noncommutative functions $P_{i},Q_{i}\in\mathcal F$ for $1\leq i\leq k$
\begin{align}\begin{split}\label{freeam}
&\tau( ((P_{1}(\sfa,\sfy)-\tau_{\sfa,\sfy}(P_{1}(\sfa,\sfy)|\sfy))(Q_{1}(\sfb,\sfy)-\tau_{\sfb,\sfy}(Q_{1}|\sfy))(P_{2}(\sfa,\sfy)-\tau(P_{2}(\sfa,\sfy)|\sfy))\\
&\qquad\cdots (P_{k}(\sfa,\sfy)-\tau_{\sfa,\sfy}(P_{k}(\sfa,\sfy)|\sfy))(Q_{k}(\sfb,\sfy)-\tau_{\sfb,\sfy}(Q_{k}(\sfb,\sfy)|\sfy)|\sfy)=0
\end{split}\end{align}
 (recall from Section \ref{s:prel} that $\tau(\cdot|\sfy)$ denotes the conditional expectation onto $L^\infty(\sfy)$).
 
The following proposition provides a random matrix model for the joint law of $(\sfa,\sfb,\sfy)$ defined above, a model for a large class of distributions of elements which are free with 
amalgamation over a commutative algebra.
\begin{proposition} \label{discr} Let $\mu_A,\mu_B,\mu_Y$
be compactly supported measures such that $\Sigma(\mu_A)$ and $\Sigma(\mu_B)$ are finite,  and $\mu_Y$ be a discrete sum of atoms:
\begin{align*}
\mu_{Y}=\sum_{i=1}^p\alpha_{i}\delta_{y_{i}}.
\end{align*}
Let $A_N, B_N$ be a sequence of deterministic self-adjoint matrices, such that their spectral measures $\mu_{A_N}, \mu_{B_N}$ converge weakly towards $\mu_A, \mu_B$ respectively. We assume that  there exists a constant $\fK$, such that $\supp\mu_{A_N}, \supp\mu_{B_N}\subset [-\fK,\fK]$. 
Let $Y_{N}$ be a diagonal matrix with $N_{i}$ eigenvalues equal to $y_{i}$ for $1\leq i\leq p$ and $N=\sum_{i=1}^{p}N_{i}$. Moreover, the limit $N_{i}/N$ goes to $\alpha_{i}$ so that the spectral measure of $Y_{N}$ goes to $\mu_{Y}=\sum\alpha_{i}\delta_{y_{i}}$.
Then under the distribution 
\begin{align}\label{e:lawABY0}
\rd\mu_{A_N,B_N,Y_N}(U,V)=\frac{e^{\frac{\beta N}{2}\tr (A_{N}UY_{N} U^{*})+\frac{\beta N}{2}\tr (B_{N}VY_{N} V^{*})}\rd U \rd V}{I_{N}(A_N,Y_N) I_{N}(B_N,Y_N)},
\end{align} 
the joint non-commutative law of $(U^{*}A_{N}U,V^{*}B_{N}V,Y_{N})$ converges towards the law of $(\sfa,\sfb,\sfy)$, where $\sfa$ and $\sfb$ are free with amalgamation over $\sfy$ and 
the joint law of $(\sfa,\sfy)$ (resp.   $(\sfb,\sfy)$) are given by $\tau_{\mu_A, \mu_Y}$ (resp. $\tau_{\mu_B, \mu_Y}$) as in \eqref{e:tauF}.

\end{proposition}
\begin{proof}


Under \eqref{e:lawABY0} the  noncommutative joint distribution of $(U^{* }A_{N} U, Y_{N})$ converges towards $\tau_{\mu_A,\mu_Y}$ by Theorem \ref{t:convl1} and Remark \eqref{r:beta1}.
Moreover,  $\sfy$ is bounded and so that \eqref{e:tauF} extends to noncommutative polynomials $F$  in $\sfa,\sfy$.

As indicated in Section \ref{s:prel},  the conditional expectation with respect to $\sfy$ defined on $L^{2}(\tau_{\sfa,\sfy})$ is the orthogonal projection onto the $L^2$-completion of the von Neumann algebra generated by $\sfy$ (or, when $\sfy$ is bounded, simply onto $L^2(\tau_{\sfy})$). When viewed as a  linear operator on  the algebra generated by $(\sfa,\sfy)$ with value in the algebra generated by $\sfy$, it satisfies $\tau_{\sfa,\sfy}(1|\sfy)=1$ and $\tau_{\sfa,\sfy}(f_1(\sfy)g(\sfa, \sfy)f_2(\sfy)|\sfy)=f_1(\sfy)\tau(g(\sfa,\sfy)|\sfy)f_2(\sfy)$ for all bounded measurable functions $f_1,f_2$, and all rational functions $g$ in $L^\infty(\sfa)$ and $L^\infty(\sfy)$.

We notice that    $\tau_{\sfa,\sfy}(f(\sfa)|\sfy)$ is a function of $\sfy$. It is equal to
\begin{align*}
\tau_{\sfa,\sfy}(f(\sfa)|\sfy)=\sum_{i=1}^p \frac{\tau_{\sfa,\sfy}(f(\sfa)\bm1_{\sfy=y_{i}})}{\tau_{\sfa,\sfy}(\bm1_{\sfy=y_{i}})}\bm1_{\sfy=y_i}
\end{align*}

Let $\tau_{N}$ be the normalized trace restricted to the algebra generated by $U^{*}A_{N}U, V^*B_N V$ and $Y_N$:
\begin{align}\label{e:tauFN}
\tau_N(F(\sfa,\sfb, \sfy))\deq  \frac{1}{N}\Tr(F(U^*A_NU,V^*B_NV, Y_N)).
\end{align}
 \eqref{freeam} follows from showing that 
\begin{align}\begin{split}\label{freeam2}
&\tau_{N}\left( (P_{1}(\sfa,\sfy)-\tau_{N}(P_{1}|\sfy))(Q_{1}(\sfb,\sfy)-\tau_{N}(Q_{1}|\sfy))\right.\\
&\cdots \left.(P_{k}(\sfa,\sfy)-\tau_{N}(P_{k}|\sfy))(Q_k(\sfb,\sfy)-\tau_{N}(Q_{k}|\sfy)) T(\sfy)\right)
\end{split}\end{align}
goes to zero almost surely as $N$ goes to infinity for any polynomial $T$. 

We denote in short $\bar P_i (\sfa,\sfy)=P_i(\sfa,\sfy) -\tau_N(P_i|\sfy)$ as well as $\bar Q_i(\sfb,\sfy)=
Q_i(\sfb,\sfy)-\tau_N(Q_i|\sfy)$, and for all $1\leq i\leq p$, $R_i=\bm1_{\sfy=y_i}$. Observe that $\tau_{N}(\bar P_i (\sfa,\sfy))=0$ and $\tau_{N}(\bar Q_i(\sfb,\sfy))=0$ for all $i\in\{1,\ldots,k\}$. 
The main point  in the proof is to notice that under $\mu_{A_N,B_N,Y_N}$ the law of $U$ is invariant under  right multiplication  by $\hat U= \mbox{diag}(U_{1},\ldots, U_{p})$  where $U_{i}$ is an $N_{i}\times N_{i}$  unitary matrix for $1\leq i\leq p$. Indeed, this is true for  the Haar measure and the density is also invariant since
$$\tr(\hat U^{*}U^{*}A_{N}U \hat U Y_{N})=\tr(U^{*}A_{N}U\hat U Y_{N}\hat U^{*})=\tr(U^{*}A_{N}U Y_{N}),$$
because $Y_{N}$ commutes with $\hat U$. We take the $\hat U_i$ following the Haar measure on the unitary or orthogonal  group  $U(N_{i})$. Then the expression in \eqref{freeam2} has the same law as
\begin{align*}
\frac{1}{N} \Tr\left( \hat U^* \bar P_1(U^{*}A_{N}U, Y_{N})\hat U \bar Q_1(V^*B_NV,Y_N)\hat U^* \bar P_2(U^{*}A_{N}U, Y_{N})\hat U \cdots \bar Q_k(V^*B_NV,Y_N) T(Y_{N})\right)\,.
\end{align*}
But, it is well known that concentration of measures holds under the Haar measure  \cite[Theorem 4.4.27]{AGZ}  so that since the $N_{i}$ go to infinity, the above quantity is close to its expectation over $\hat U$. Taking $T$ to approximate the indicator function $\bm1_{\sfy=y_i}$ shows that
 \eqref{freeam2} converging to zero
 is equivalent to show  that for any $1\leq i\leq p$,
$$
L_N^{(k)}(i)\deq\bE_{\hat U}\left[\frac{1}{N}\Tr\left(R_i\hat U^* \bar P_1(U^{*}A_{N}U, Y_{N})\hat U \bar Q_1(V^*B_NV,Y_N)\hat U^* \bar P_2(U^{*}A_{N}U, Y_{N})\hat U \cdots \bar Q_k(V^*B_NV,Y_N)\right)\right]
$$
goes to zero, where the expectation is over $\hat U$.

For any $N\times N$ matrix $X$, with slight abuse of notations, let $X_{ij}$ be the $N_i\times N_j$ submatrix from the eigenspace of  $\{Y_N=y_i\}$ into $\{Y_N=y_j\}$ given by the non trivial part of 
$R_i XR_j$.  By the definition of  $L_N^{(k)}(i)$, it equals
\begin{align*}
\begin{split}
&\frac{1}{N}\sum_{1\leq i_1,\ldots i_{2k-1}\leq p} \bE_{\hat U}\left[\tr\left(\hat U_i^*
\bar P_1(U^{*}A_{N}U, Y_{N})_{ii_1}\hat U_{i_1} \bar Q_1(V^*B_NV,Y_N)_{i_1 i_2} \hat U_{i_2}^*\bar P_2(U^{*}A_{N}U, Y_{N})_{i_2 i_3} \hat U_{i_3}\right.\right.\\
&\cdots  \left.\left.\hat U_{i_{2k-2}}^*
\bar P_k(U^{*}A_{N}U, Y_{N})_{i_{2k-2}i_{2k-1}}\hat U_{i_{2k-1}}\bar Q_k(V^*B_NV,Y_N)_{i_{2k-1},i}\right)\right].
\end{split}
\end{align*}
We claim that this quantity goes to zero as $N$ goes to infinity and prove it by induction over $k$. The expectation is over $\hat U$, and $U,V$ are fixed. In the case $k=1$ we see that to have a non zero contribution we need that $i_1=i$ by invariance and independence of the law of the $\hat U_i$. Then, since $\bE[\hat U_i({ij})\hat U^*_i({kl})]$ vanishes unless $i=l, j=k$ where it equals $1/N_i$. We find that $L_N^{(1)}(i)$ equals
\begin{align*}\begin{split}
L_N^{(1)}(i)
&=\frac{1}{N}\sum_{1\leq i_1\leq p} \bE_{\hat U}\left[\tr\left(\hat U_i^*
\bar P_1(U^{*}A_{N}U, Y_{N})_{ii_1}\hat U_{i_1} \bar Q_1(V^*B_NV,Y_N)_{i_1 i} \right)\right]\\
&=\frac{1}{N N_i} \tr\left(\bar P_1(U^{*}A_{N}U, Y_{N})_{ii}\right)\tr(\bar Q_1(V^*B_NV,Y_N)_{i i})
\end{split}\end{align*}
which vanishes uniformly because by recentering $ \tr(\bar P_1(U^{*}A_{N}U, Y_{N})_{ii})$ vanishes. 
Assuming $L_N^{(k)}(i)$ goes to zero up to any choices of polynomials, $1\leq i\leq p$ and for $k\leq \ell$, we show it does as well for $k=\ell+1$. To this hand we use the loop equation which implies that for any noncommutative function $p$ in terms of these submatrices $\{\hat U_i, \hat U_i^*\}_{1\leq i\leq p}$ and $\{P_1(U^{*}A_{N}U, Y_{N})_{ij}, \bar Q_1(V^*B_NV,Y_N)_{i j})\}_{1\leq i,j\leq p}$,
\begin{align}\begin{split}\label{e:loopU}
&\phantom{{}={}}\bE_{\hat U}\left[\frac{1}{N_i} \tr (\hat U_i p)\right]=\bE_{\hat U}\left[\sum_{p=p_1 \hat U_i^* p_2} \frac{1}{N_i}\tr(p_1)\frac{1}{N_i}\tr(p_2)-\sum_{p=p_1\hat U_ip_2}  \frac{1}{N_i}\tr(p_1\hat U_i)\frac{1}{N_i}\tr(p_2 \hat U_i)\right]\\
&= \sum_{p=p_1 \hat U_i^* p_2}\bE_{\hat U}\left[ \frac{1}{N_i}\tr(p_1)\right]\bE_{\hat U}\left[\frac{1}{N_i}\tr(p_2)\right]-\sum_{p=p_1\hat U_ip_2} \bE_{\hat U}\left[ \frac{1}{N_i}\tr(p_1\hat U_i)\right]\bE_{\hat U}\left[\frac{1}{N_i}\tr(p_2 \hat U_i)\right]+\oo(1),
\end{split}\end{align}
where in the second line we finally used concentration of measures with respect to $\{\hat U_i\}_{1\leq i\leq p}$ (and the fact that the $N_i$ go to infinity), see \cite[Section 4.4.2]{AGZ}. Using this equation in $L_N^{(\ell +1)}$ with the choice
$$p=\bar P_1(U^{*}A_{N}U, Y_{N})_{ii_1}\hat U_{i_1} \bar Q_1(V^*B_NV,Y_N)_{i_1 i_2} \cdots \hat U_{i_{2k-1}}\bar Q_k(V^*B_NV,Y_N)_{i_{2k-1},i},$$
we see that we get terms of lower orders for which we can use our induction hypothesis, and when $i_1=i$, the trace of the first term, which vanishes by recentering. Hence, we conclude that $L_N^{(\ell+1)}$ goes to zero.
\end{proof}

We can have a more general version of Proposition \ref{discr} by approximation, i.e. $\mu_Y$ is any probability measure with bounded $L_1$ norm. For any $p\geq 1$, we take a discrete approximation of $\mu_Y$ by $\mu^{(p)}_Y$ which 
contains $p$ discrete atoms each with probability $1/p$
\begin{align*}
\mu^{(p)}_Y=\frac{1}{p}\sum_{i=1}^p\delta_{y_i},\quad y_1\leq y_2\leq \cdots \leq y_p.
\end{align*}
Let $Y^{(p)}_{N}$ be a diagonal matrix with $N_{i}$ eigenvalues equal to $y_{i}$ for $1\leq i\leq p$ and $N=\sum_{i=1}^{p}N_{i}$. Moreover, the limit $N_{i}/N$ goes to $1/p$ so that the spectral measure of $Y^{(p)}_{N}$ goes to $\mu^{(p)}_{Y}$. From the discussion above, under the measure $\rd\mu^{N}_{A_N,B_N,Y_N^{(p)}}$
 the joint non-commutative law of $(U^{*}A_{N}U,V^{*}B_{N}V,Y_{N}^{(p)})$ converges towards the law $\tau_{\mu_A,\mu_B,\mu_{Y^{(p)}}}$ of $(\sfa,\sfb,\sfy^{(p)})$, where 
the joint law of $(\sfa,\sfy^{(p)})$ is given by $\tau_{\mu_A, \mu_{Y^{(p)}}}$ as in \eqref{e:tauF}, and analogously the joint law of $(\sfb,\sfy^{(p)})$ is given by $\tau_{\mu_B, \mu_{Y^{(p)}}}$. Moreover, 
the joint law of $(\sfa,\sfb)$ is free with amalgamation over $\sfy^{(p)}$.

The proof of the large deviation lower bound   follows from a continuity argument given by the following statement that the joint law $\tau_{\mu_A,\mu_B,\mu_{Y^{(p)}}}$ of $(\sfa,\sfb,\sfy^{(p)})$ converges to $\tau_{\mu_A,\mu_B,\mu_{Y}}$
 as $p$ goes to infinity.
This claim is a consequence of the continuity of $\mu_Y\mapsto\tau_{\mu_A,\mu_Y}$ of Theorem \ref{t:convl1}  and the  following general statement 
\begin{lemma}\label{c:continuityfa}
Given a sequence of $\{\sfa, \sfb, \sfy^{(p)}\}_{p\geq 0}$ non-commutative variables, such that $(\sfa,\sfb)$ are free with amalgamation over $\sfy^{(p)}$.
If as $p$ goes to infinity, the joint law of $(\sfa,\sfy^{(p)})$ and  $(\sfb,\sfy^{(p)})$
converges to the joint laws of $(\sfa,\sfy)$ and $(\sfb,\sfy)$, then $\{\sfa, \sfb, \sfy^{(p)}\}$ converges towards the law of $(\sfa,\sfb,\sfy)$, where $(\sfa,\sfb)$ are free over amalgamation over $\sfy$. 
\end{lemma}
\begin{proof}
We can consider the ultraproduct $\{\sfa,\sfb,\sfy^{(p)}\}_{p\geq 0}$. Any limit point for the joint distributions of $(\sfa,\sfb,\sfy^{(p)})$ can be seen as a trace on this ultra-product equipped with a well chosen ultra-filter $\omega$. By \cite[Proposition 4]{Uedaamal}, the algebras
generated by  $(\sfa,\sfy)$ and $(\sfb,\sfy)$ are free by amalgation over $\sfy$, independently of the choice of the ultra-filter $\omega$ and hence of the limit point. This concludes the proof.

\end{proof}

\subsubsection{ Proof of the large deviation lower bound}
If $\mu\in \mathcal H^{\sf A+B}$ as defined in \eqref{e:defHAB}, we know by Lemma \ref{c:continuityfa}
 that we can approximate it 
by some $\mu_p$ which is the distribution of $\sfa+\sfb$ where
\begin{enumerate}
 \item $\sfa$ and $\sfb$ are free with amalgamation over $\sfy^{(p)}$;
 \item the law of $\sfa, \sfb, \sfy^{(p)}$ are given by $\mu_A, \mu_B,\mu_{Y^{(p)}}$, and $\mu_{Y^{(p)}}$ contains $p$ discrete atoms each with probability $1/p$
\begin{align*}
\mu^{(p)}_Y=\frac{1}{p}\sum_{i=1}^p\delta_{y_i},\quad y_1\leq y_2\leq \cdots \leq y_p.
\end{align*}
 \item the joint laws of $(\sfa, \sfy^{(p)})$ and $(\sfb, \sfy^{(p)})$ are given by $\tau_{\mu_A,\mu_{Y^{(p)}}}$ and $\tau_{\mu_B,\mu_{Y^{(p)}}}$ respectively.
 \end{enumerate}

Let $Y^{(p)}_{N}$ be a diagonal matrix with $N_{i}$ eigenvalues equal to $y_{i}$ for $1\leq i\leq p$ and $N=\sum_{i=1}^{p}N_{i}$. Moreover, the limit $N_{i}/N$ goes to $1/p$ so that the spectral measure of $Y^{(p)}_{N}$ goes to $\mu^{(p)}_{Y}$. Then for $p$ large enough, such that $\bB_{\delta/2}(\mu_p)\subset \bB_\delta(\mu)$ we have
\begin{align*}
\bP_N(\bB_\delta(\mu))&\geq \bP_N(\bB_{\delta/2}(\mu_p))\\
&=\bE_{U}\left[\bm 1(\mu_N\in \bB_{\delta/2}(\mu_p))\frac{\int \exp\{(\beta/2)N\Tr(Y_N^{(p)} V(A_N+UB_NU^*)V^*)\}\rd V}{\int \exp\{(\beta/2)N\Tr(Y_N^{(p)} V(A_N+UB_NU^{*}))V^*\}\rd V}
\right]\\
&\geq  e^{\beta N^{2}(-I( \mu_{Y^{(p)}},\mu_p)+I(\mu_A,\mu_{Y^{(p)}})+I(\mu_B,\mu_{Y^{(p)}})+\oo(N)}\mu_{A_N,B_N,Y_N^{(p)}}( \mu_N\in \bB_{\delta/2}(\mu_p))\\
&\geq  e^{-\beta N^{2} \{ I( \mu_{Y^{(p)}},\mu)-I(\mu_A,\mu_{Y^{(p)}})-I(\mu_B,\mu_{Y^{(p)}})\}+\oo(N)}\geq e^{\beta N^2( \mathcal I^{\sf A+B}(\mu)+o(1))},
\end{align*}
where in the last line we used that $\mu_{A_N,B_N,Y_N^{(p)}}( \mu_N\in \bB_{\delta/2}(\mu_p))$ goes to one by Proposition \ref{discr} and the continuity of spherical integrals of Proposition \ref{p:continuity}. This finishes the proof for the large deviation lower bound.

\subsubsection{Improved large deviations upper bound}

A drawback of our large deviation bound is that we do not know how to prove that the rate function $\mathcal I^{\mathsf{A+B}}$ has a unique minimizer at $\mu_{A}\boxplus\mu_{B}$ (and whether this is true). To circumvent this fact we can improve our large deviation upper bound as follows. From \cite{CGM}, there exists $\varepsilon>0$ such that for any sequence of Hermitian matrices $Y_N$ such that $\fK\|Y_N\|_{\infty}<\varepsilon$ and with spectral measure converging towards $\mu_{Y}$, we know that the following limit exists :
\begin{align*}
\lim_{N\rightarrow\infty}\frac{1}{\beta N^{2}}\log\mathbb E\left[\frac{1}{\int e^{\frac{\beta N}{2}\Tr(Y_NV(A_N+UB_NU^{*})V^{*})} \rd V}\right]=-I(\mu_Y,\mu_{A},\mu_{B}).
\end{align*}

\begin{theorem}
Let $A_N, B_N$ be a sequence of deterministic self-adjoint matrices, such that their spectral measures $\hat \mu_{A_N}, \hat \mu_{B_N}$ converge weakly towards $\mu_A, \mu_B$ respectively, and there exists a constant $\fK>0$, such that $\supp\hat \mu_{A_N}, \supp \hat\mu_{B_N}\subset [-\fK,\fK]$,
then the empirical eigenvalue distribution $\mu_N$ of $A_N+UB_NU^{*}$ satisfies a large deviation upper bound with rate function 
$$\tilde{\mathcal I}^{\mathsf{A+B}}(\mu)=\max \{\mathcal I^{\mathsf{A+B}}(\mu),\mathcal I_{-}^{\mathsf{A+B}} (\mu)\},$$ where 
 $$\mathcal I_{-}^{\mathsf{A+B}} (\mu)=\sup_{ \supp(\mu_{Y})\subset \atop[-\varepsilon/\fK, \varepsilon/\fK]}\{I(\mu_Y,\mu_{A},\mu_{B})-I(\mu_Y,\mu)\},$$
where the supremum is taken over probability measures $\mu_{Y}$ with support in $[-\varepsilon/\fK, \varepsilon/\fK]$. 
Moreover, $\tilde{\mathcal I}^{\mathsf{A+B}}(\cdot)$ is a good rate function and vanishes only at $\mu_{A}\boxplus\mu_{B}$.
\end{theorem}
\begin{proof}
The same reasoning as for the proof of Theorem \ref{t:AUBU} shows that
\begin{align*}
\bP_N(\bB_\delta(\mu))
&=\bE\left[\bm 1(\mu_N\in \bB_\delta(\mu))\frac{\int \exp\{(\beta/2)N\Tr(Y_N V(A_N+UB_NU^{*})V^*)\}\rd V}{\int \exp\{(\beta/2)N\Tr(Y_NV(A_N+UB_NU^{*})V^*)\}\rd V}\right]\\
&\leq  e^{\beta N^{2}(I( \mu_Y,\mu)+\oo_\delta(1)+\oo_N(1))}\bE\left[\frac{1}{\int \exp\{(\beta/2)N\Tr(Y_NV(A_N+UB_NU^{*})V^*)\}\rd V}\right]\\
&=e^{\beta N^{2}(I( \mu_Y,\mu)-I(\mu_Y, \mu_A,\mu_B)+\oo_\delta(1)+\oo_N(1))},
\end{align*}
which gives the large deviation upper bound. Hence, the only thing to show is that $\tilde{\mathcal I}^{\mathsf{A+B}}(\cdot)$ is non negative and vanishes only at $\mu_{A}\boxplus\mu_{B}$. Moreover, if  $\mu_Y$ has the distribution of $\varepsilon \tilde Y$ with $\tilde Y$ uniformly bounded with law $\tilde \mu_Y$ and $\varepsilon>0$ small enough, by \cite{CGM},  we see that $I(\mu_Y, \mu_{A},\mu_{B})
$ is an absolutely  converging series in $\varepsilon$ whose coefficients only depends on the moments of $\tilde \mu_Y, \mu_{A},\mu_{B}$.  As a consequence, it is a continuous function of these compactly supported measures. It clearly follows that $\tilde{\mathcal I}^{\mathsf{A+B}}(\cdot)$ is a good rate function. To show that it vanishes only at $\mu_{A}\boxplus\mu_{B}$, we use that from  \cite[p. 38]{CGM}, it is proven that if we take 
$\nu_\tau=(1-\tau)\delta_0+\tau\delta_\theta$, with $\tau, \theta$ small enough, then
%
\begin{align*}
I(\nu_{\tau},\mu)=\tau\int_{0}^{\theta} R_{\mu}(t)\rd t +\OO(\tau^{2}),
\end{align*}
where $R_\mu(\cdot)$ is the $R$-transform of the measure $\mu$.
Hence,
\begin{align*}H^{\mathsf{A+B}}_{\mu}(\nu_{\tau})= \tau\int_{0}^{\theta} (R_{\mu}(t)-R_{\mu_{A}}(t)-R_{\mu_{B}}(t)) \rd t +\OO(\tau^{2}),
\end{align*}
which implies that 
\begin{align*}\mathcal I^{\mathsf{A+B}}(\mu)\geq \tau\int_{0}^{\theta} (R_{\mu}(t)-R_{\mu_{A}}(t)-R_{\mu_{B}}(t)) \rd t +\OO(\tau^{2}).
\end{align*}
This implies that if $\mathcal I^{\mathsf{A+B}}(\mu)=0$ then for sufficiently small $\theta$
\begin{align}
\label{e:upb}\int_{0}^{\theta} (R_{\mu}(t)-R_{\mu_{A}}(t)-R_{\mu_{B}}(t)) \rd t \leq 0,
\end{align}
which is still verified by many probability measures. Next, we use the symmetrization to show that $\tilde{\cal I}^{\mathsf{A+B}}=0$ implies that the equality holds in \eqref{e:upb}. In fact for $\theta,\tau$ small enough
\begin{align*}
I(\nu_\tau,\mu_{A},\mu_{B})= \tau\int_{0}^{\theta} (R_{\mu_{A}}(t)+R_{\mu_{B}}(t)) \rd t+\OO(\tau^{2}).\end{align*}
Hence for sufficiently small $\tau, \theta$ we have
\begin{align*}
\tilde{\mathcal  I}^{\mathsf{A+B}}(\mu) \geq  \tau  \left|\int_{0}^{\theta} (R_{\mu}(t)-R_{\mu_{A}}(t)-R_{\mu_{B}}(t)) \rd t \right|+\OO(\tau^{2}).
\end{align*}
By sending $\tau$ to zero, we have that $\tilde{\cal I}^{\mathsf{A+B}}(\mu)=0$ implies that $R_\mu(\theta)=R_{\mu_A}(\theta)+R_{\mu_B}(\theta)$ for all $\theta$ small enough, which further implies that $\mu=\mu_A\boxplus \mu_B$.
We conclude that $\tilde{\mathcal I}^{\mathsf{A+B}}(\cdot)$  vanishes only at $\mu_{A}\boxplus\mu_{B}$.

\end{proof}

\section{Large Deviation Principle for Kostka numbers}

In this section, we use the spherical integral to derive the large deviation estimates of the Kostka numbers and prove Theorem \ref{t:Kostka}. From the definition \eqref{e:defKostka} of Kostka numbers,
\begin{align}
K_{\bmla_N\bmeta_N}\leq \frac{S_{\bmla_N}(e^{Y_N})}{m_{\bmeta_N}(e^{Y_N})},\quad e^{Y_N}=(e^{y_1}, e^{y_2}, \cdots, e^{y_N}),\label{kostkaub}
\end{align}
where $Y_N =\diag\{y_1,y_2,\cdots,y_N\}$ is a sequence of diagonal matrices, with $y_1\geq  y_2\geq \cdots\geq y_N$  and spectral measure converging in Wasserstein distance \eqref{e:wd} towards $\mu_Y$ (we can take $y_1, y_2, \cdots, y_N$ the $N$-quantiles of $\mu_Y$).
For the monomial symmetric function $$m_{\bm\eta_N}(\bmx_N)=\sum_{\bma_N\sim\bm\eta_N}x_1^{a_1}x_2^{a_2}\cdots x_N^{a_N},$$
where $\bma_N=(a_1,a_2,\cdots, a_N)\sim\bm \eta_N=(\eta_1\geq \eta_2\geq \cdots\geq \eta_N)$, if the parts of $\bma_N$ is a rearrangement of the parts of $\bm\eta_N$. We easily see that if
$\bmx_N=e^{Y_N}$, $Y_{N}=\diag\{y_1, y_2,\cdots,y_N\}$ with $y_1\geq y_2\geq\cdots \geq y_N$, then 
\begin{equation}\label{gf}
e^{\eta_1 y_1+\eta_2y_2+\cdots+\eta_N y_N}\leq m_{\bm\eta_N}(\bmx_N)=m_{\bm\eta_N}(e^{Y_N})\leq N!e^{\eta_1 y_1+\eta_2y_2+\cdots+\eta_N y_N}.
\end{equation}
We recall from \eqref{e:defm} that $m[\bmeta_N]=\frac{1}{N}\sum \delta(\frac{\eta_{i}+N-i}{N})$ where $i\rightarrow \eta_{i}+N-i$ is non-decreasing. Hence, if $m[\bmeta_N]$ goes to $(T_\mu)_\#(\unif[0,1])$, $\frac{1}{N}\sum\delta(\frac{\eta_i}{N})$ goes to $(T_\mu-x)_\# (\unif[0,1])$.
This implies, with \eqref{gf}, that
\begin{align}\label{e:msf}
\frac{1}{N^2}\log m_{\bmeta_N}(e^{Y_N})
=\int (T_\mu-x)T_{\mu_Y}\rd x+\oo_{\delta}(1)+\oo_N(1),
\end{align}
if $m[\bmeta_N]\in \mathbb B_{\delta}(\mu)$ for some probability measure $\mu$ such that $T_{\mu}(x)\geq x$. 
The large deviation upper bound follows from combining the asymptotics of Schur symmetric polynomials \eqref{e:logS}, \eqref{kostkaub},  and \eqref{e:msf},
\begin{align}\begin{split}\label{e:LDPu3}
&\limsup_{\delta\rightarrow 0}\limsup_{N\rightarrow \infty}\frac{1}{N^2}\log \sup_{m[\bmeta_N]\in \bB_\delta(\mu)}K_{\bmla_N\bmeta_N}\leq  -H^{\mathsf{K}}_{\mu}(\mu_Y),\\
&H^{\mathsf{K}}_\mu(\mu_Y)=\int (T_\mu-x)T_{\mu_Y}\rd x-J(\mu_Y,m_\bmla),
\end{split}\end{align}
where the functional $J(\cdot,\cdot)$ is defined in \eqref{e:logS}.
Taking the  infimum over $\mu_Y\in \cM$ on the right hand side of \eqref{e:LDPu3}  finishes the proof of the large deviation upper bound in Theorem \ref{t:Kostka}.
It is known that the Kostka number $K_{\bmla_N \bmmu_N}$ is positive if and only if $\bmla_N$ and $\bmeta_N$ are of the same size, and $\bmla_N$ is larger than $\bmeta_N$ in dominance order:
\begin{align}\label{e:Kregion}
\lambda_1+\la_2+\cdots+\la_i\geq \eta_1+\eta_2+\cdots+\eta_i, \quad 1\leq i\leq N.
\end{align}
We recall  from Theorem \ref{t:Kostka} that $\mathcal I^{\mathsf K}(\mu)=\sup_{\nu\in \cM}H^{\mathsf K}_\mu(\nu)$. 
It turns out that the rate function $\cal I^{\mathsf K}(\mu)$  equals $+\infty$ outside the admissible region $\cA_{m_\bmla}$ described by the limit of \eqref{e:Kregion}:
 \begin{align}\label{e:ami1}
 \int_0^1 (T_{\mu}-T_{m_\bmla})(x)\rd x= 0,\quad \int_y^1 (T_{\mu}-T_{m_\bmla})(x)\rd x\leq 0 \quad \forall y\in [0,1].
 \end{align}
In fact, from the expression \eqref{e:logS} of $J( \mu_Y, m_\bmla)$, we have
\begin{align}\label{e:JLY}
J( \mu_{LY}, m_\bmla)
=2I(\mu_{LY}, m_\bmla)-L\int xT_{\mu_{Y}}\rd x+\OO(\log L),
\end{align}
as $L\rightarrow \infty$.
Using the estimate \eqref{e:JLY} as input, Item 2 in Proposition \ref{p:rateIK} can be proven by a similar argument as in Item 2 in Proposition \ref{p:domain1}. For other parts of Proposition \ref{p:rateIK}, the proofs are essentially the same as those in Proposition \ref{p:domain1}, so we omit them.
\begin{proposition}\label{p:rateIK}
Under the assumptions  and notations of Theorem \ref{t:Kostka}, the function $H_\mu^{\mathsf K}(\cdot)$ and rate function $\mathcal I^{\mathsf K}(\cdot)$  satisfy:
\begin{enumerate}
\item
For $\mu$ satisfies \eqref{e:limitadm}, $H_\mu^{\mathsf K}(\cdot)$ is { upper semi-continuous in the weak topology on $\{\nu\in \cM: \nu(|x|)\leq \fR\}$ for any $\fR>0$}.
\item If $\int_0^1 (T_\mu(x)-T_{m_\bmla}(x))\rd x\neq 0$, or there exists some $0<y<1$ such that 
\begin{align*}
\int_y^1(T_\mu(x)-T_{m_\bmla}(x))\rd x>0, 
\end{align*}
then $\mathcal I^{\mathsf K}(\mu)=+\infty$.

\item If there exists some small constant $\fc>0$
\begin{align}\begin{split}\label{e:domainK}
&\int_y^1 \left(T_{\mu}(x) -T_{\mu_{\bmla}}(x)\right)\rd x\leq
\left\{\begin{array}{ll}
 -\fc y,      &\text{ for } 0\leq y\leq \fc,\\
 -\fc,         &\text{ for } \fc\leq y\leq 1-\fc,\\
 -\fc(1-y),  &\text{ for } 1-\fc\leq y\leq 1.
 \end{array}\right.\\
\end{split}\end{align}

 then
$\mathcal I^{\mathsf K}(\mu)=H^{\mathsf K}_\mu(\nu^*)<\infty$ for some probability measure $\nu^{*}$ such that  $\nu^*(|x|)<\infty$. 

\item The rate function $\mathcal I^{\mathsf K}(\cdot)$ is lower semicontinuous on $\cM^{\rm b}([0, \fK])$ and  achieves its minimum value only at the uniform measure $\unif[\int x\rd m_\bmla-1/2, \int x\rd m_\bmla+1/2]$.

\item For any measure $\mu$ in the admissible set  $\mathcal A_{m_\bmla}$ as defined in \eqref{e:ami1}, there exists a sequence of measures $\mu^{\varepsilon}$ inside the region as given in \eqref{e:domainK}, converging to $\mu$ in the weak topology and 
$\lim_{\varepsilon\rightarrow 0}\mathcal I^{\mathsf D}(\mu^{\varepsilon})=\mathcal I^{\mathsf D}(\mu)$.
\end{enumerate}
\end{proposition}

\subsection{Large deviation lower bound}

In this section we prove the large deviation lower bound in Theorem \ref{t:Kostka}. It follows from combining the following Propositions \ref{p:unique3} and \ref{p:lowerbound3}, and noticing that the number of partitions with at most $N$ rows in $\bB_\delta(\mu)$ is at most $\exp\{\OO(N\log N)\}$.

\begin{proposition}\label{p:unique3}
We assume the assumptions of Theorem \ref{t:Kostka}.
For any probability measure $\mu_Y\in \cM$,  there exists a unique $\mu\in \cM^{\rm b}([0,\fK])$ such that 
\begin{align}\label{e:choicemuY}
\mu_Y\in \arg\sup_{\nu\in \cM} H^{\mathsf K}_\mu(\nu), \quad H^{\mathsf K}_\mu(\nu)=\int  T_{\nu} (T_\mu-x)\rd x-J(\nu, m_\bmla).
\end{align}
and $T_\mu$ is uniquely determined by $T_Y$ by
\begin{align*}
T_\mu(x)=\tau({\mathsf m}_\bmla |\mathsf y)\circ T_Y(x)+x+\int \left(\frac{1}{T_Y(x)-T_Y(y)}-\frac{1}{1-e^{T_Y(y)-T_Y(x)}}\right)\rd y.
\end{align*}
Here, $\tau({\mathsf m}_\bmla |\mathsf y)$ is the conditional expectation of ${\mathsf m}_\bmla$ knowing $\mathsf y$ under the non-commutative distribution $\tau$ uniquely associated to $(m_\bmla,\mu_Y)$ as in Theorem \ref{t:convl1}.
\end{proposition}

\begin{proposition}\label{p:lowerbound3}
We assume the assumptions of Theorem \ref{t:Kostka}. For any probability measure $\mu_Y\in \cM$, and $\mu$ be the unique measure in $\cM^{\rm b}([0, \fK])$ so that 
\begin{align*}
\mu_Y\in \arg\sup_{\nu\in \cM} H^{\mathsf K}_\mu(\nu), \quad H^{\mathsf K}_\mu(\nu)=\frac{1}{2}\int  T_{\nu} (T_\mu-x)\rd x-J(\nu, m_\bmla).
\end{align*}
Then we have
\begin{align}\label{e:LDP3}
\frac{1}{N^2}\log\sup_{\bmeta_N: m[\bmeta_N]\in \bB_\delta(\mu)}K_{\bmla_N\bmeta_N}\geq -\left(H_\mu^{\mathsf K}(\mu_Y)+\oo_{\delta}(1)+\oo_N(1)\right).
\end{align}
\end{proposition}

\begin{proof}[Proof of Theorem \ref{t:Kostka}]
Item 1 of Theorem \ref{t:Kostka} follows from Proposition \ref{p:rateIK}.
For Item 2, the large deviation upper bound follows from \eqref{e:LDPu3}. If $\mu$ does not satisfy $\int_0^1 (T_\mu(x)-T_{\mu_B}(x))\rd x\neq 0$ or  the limiting Schur-Horn inequalities \eqref{limSHK}, then both sides of \eqref{e:Kostka} are $-\infty$. There is nothing to prove. In the following we first prove \eqref{e:Kostka} when $\mu$ satisfies $\int_0^1 (T_\mu(x)-T_{\mu_B}(x))\rd x= 0$ and the strong limiting Schur-Horn inequalities \eqref{e:domainK} with some $\fc>0$.
In this case, thanks to Item 3 in Proposition \ref{p:rateIK}, 
 there 
 exists a probability measure $\mu_Y$ such that
$\mathcal I^{\mathsf K}(\mu)=H^{\mathsf K}_\mu(\mu_Y)<\infty$ and $\mu_Y\in \cM$.  
Then Propositions \ref{p:unique3} and \ref{p:lowerbound3} imply that $\mu$ is uniquely 
determined by $\mu_Y$ and the large deviation lower bound holds. This gives the full large deviation principle when  the  strong limiting Schur-Horn inequalities \eqref{e:domainK} hold. 
Next we extend it to the boundary case by a continuity argument. Thanks to Item 5 in Proposition \ref{p:rateIK},  for any measure $\mu$ inside the admissible set \eqref{e:ami1} but not satisfying \eqref{e:domainK}, there exists a sequence of measures $\mu^{\varepsilon}$ inside the region as given in \eqref{e:domainK}, converging to $\mu$ in the weak topology and 
$\lim_{\varepsilon\rightarrow 0}\mathcal I^{\mathsf K}(\mu^{\varepsilon})=\mathcal I^{\mathsf K}(\mu)$.
Then for any $\delta>0$, there exists sufficiently small $\varepsilon>0$ 
\begin{align}\label{e:bbcaseK}
\liminf_{N\rightarrow\infty}\frac{1}{N^2}\log \sup_{m[\bmeta_N]\in \bB_\delta(\mu)}K_{\bmla_N\bmeta_N}
\geq \liminf_{N\rightarrow\infty}\frac{1}{N^2}\log \sup_{m[\bmeta_N]\in \bB_{\delta/2}(\mu^\varepsilon)}K_{\bmla_N\bmeta_N}
=\mathcal I^{\mathsf K}(\mu^{\varepsilon})+\oo_\delta(1).
\end{align}
The large deviation lower bound follows by first sending $\varepsilon$ and then $\delta$ to zero in \eqref{e:bbcaseK}.
This finishes the proof of Theorem \ref{t:Kostka}.

\end{proof}

The proofs of both Propositions \ref{p:unique3} and \ref{p:lowerbound3} rely on the following probability estimate.
\begin{proposition}\label{c:expbound3}
We assume the assumptions of Theorem \ref{t:Kostka}.
Let $Y_N=\diag\{y_1,y_2,\cdots, y_N\}$ be a sequence of diagonal matrices, whose spectral measures converge  in Wasserstein distance \eqref{e:wd} towards $\mu_Y$. For any $\mu$ with support on $[0, \fK]$, if 
\begin{align}\label{e:asup3}
\sup_{\nu\in \cM}\left\{\int (T_\mu-x) T_{\nu}\rd x-J(\nu,m_\bmla)\right\}
>\int (T_\mu-x) T_{\mu_Y}\rd x-J(\mu_Y,m_\bmla).
\end{align}
Then there exists a small $\delta>0$, and positive constant $c(\delta)>0$ such that
\begin{align}\label{e:expbound3}\begin{split}
\phantom{{}={}}\sum_{\bmeta_N: m[\bmeta_N]\in \bB_\delta(\mu)}K_{\bmla_N\bmeta_N}m_{\bmeta_N}(e^{Y_N})
\leq e^{-c(\delta)N^2}S_{\bmla_N}(e^{Y_N}).
\end{split}\end{align}
\end{proposition}

\begin{proof}[Proof of Proposition \ref{c:expbound3}]
Under Assumption \eqref{e:asup3}, for sufficiently small $\varepsilon>0$, there exists a measure $\nu\in \cM$ such that
\begin{align}\label{e:epserrorK}
\int (T_\mu-x) T_{\nu}\rd x-J(\nu,m_\bmla)
\geq\int (T_\mu-x) T_{\mu_Y}\rd x-J(\mu_Y,m_\bmla)+\varepsilon.
\end{align}
We divide by $S_{\bmla_N}(e^{Y_N})$ on both sides of \eqref{e:expbound3}, and use the estimates \eqref{e:msf} and \eqref{e:logS}, 
\begin{align*}\begin{split}
&\phantom{{}={}}(1/S_{\bmla_N}(e^{Y_N}))\sum_{\bmeta_N: m[\bmeta_N]\in \bB_\delta(\mu)}K_{\bmla_N\bmeta_N}m_{\bmeta_N}(e^{Y_N})\\
&=\exp\left\{-N^2\left(J(\mu_Y, m_\bmla)-\int  (T_\mu-x)T_{\mu_Y}\rd x+\oo_{\delta}(1)+\oo_N(1)\right)\right\}\sum_{\bmeta_N: m[\bmeta_N]\in \bB_\delta(\mu)}K_{\bmla_N\bmeta_N}\\
&\leq  
\exp\left\{- N^2\left(J(\mu_Y, m_\bmla)-\int  (T_\mu-x)T_{\mu_Y}\rd x-J(\nu, m_\bmla)+\int  (T_\mu-x)T_{\nu}\rd x+\oo_{\delta}(1)+\oo_N(1)\right)\right\}\\
&\leq \exp\left\{- N^2(\varepsilon+\oo_{\delta}(1)+\oo_N(1))\right\},
\end{split}
\end{align*}
where in the first inequality we used the large deviation upper bound \eqref{e:LDPu3}, and \eqref{e:epserrorK} in the last inequality. The claim follows provided we take $\delta$ sufficiently small and $N$ large.
\end{proof}

\begin{proof}[Proof of Proposition \ref{p:unique3}]
We first prove the existence of such $\mu$ by contradiction. If there is no such $\mu$, i.e. for any measure $\mu$ supported on $[0, \fK]$, we have 
\begin{align*}
\mu_Y\not\in\arg\sup_{\nu\in \cM}\left\{\int (T_\mu-x) T_{\nu}\rd x-J(\nu, m_\bmla)\right\}.
\end{align*}
Then it follows from Proposition \ref{c:expbound3} that  there exists a small $\delta>0$, and positive constant $c(\delta)>0$ such that
\begin{align*}\begin{split}
\phantom{{}={}}\sum_{\bmeta_N: m[\bmeta_N]\in \bB_\delta(\mu)}K_{\bmla_N\bmeta_N}m_{\bmeta_N}(e^{Y_N})
\leq e^{-c(\delta)N^2}S_{\bmla_N}(e^{Y_N}).
\end{split}\end{align*}
Since the space of probability measure supported on $[0, \fK]$ is compact, we get a finite open cover $ \cup \bB_{\delta_i}(\mu_i)$ of the set of probability measure supported on $[0, \fK]$,  we get a contradiction since for $N$ large enough
\begin{align*}\begin{split}
S_{\bmla_N}(e^{Y_N})
=\sum_i\sum_{\bmeta_N: m[\bmeta_N]\in \bB_{\delta_i}(\mu_i)}K_{\bmla_N\bmeta_N}m_{\bmeta_N}(e^{Y_N})
\leq\sum_i e^{-c(\delta_i)N^2}S_{\bmla_N}(e^{Y_N})<S_{\bmla_N}(e^{Y_N}).
\end{split}\end{align*}

In the following we prove the uniqueness of such measure $\mu$ satisfying \eqref{e:choicemuY}. We note that if $\mu\in \cM^{\rm b}([0, \fK])$, then $T_\mu(x)-x$ is monotonic increasing.
Since $\mu_Y$ is one of the maximizer, then for any $\varepsilon>0$,
\begin{align*}\begin{split}
\int  T_{Y} (T_\mu-x)\rd x-J(T_Y, T_{m_\bmla})
&\geq\int  (T_Y+\varepsilon \tilde T_C)_\#(\unif[0,1]) (T_\mu-x)\rd x-J(T_Y+\varepsilon \tilde T_C, T_{m_\bmla})\\
&\geq\int  (T_{Y}+\varepsilon \tilde T_C) (T_\mu-x)\rd x-J(T_Y+\varepsilon \tilde T_C, T_{m_\bmla}),
\end{split}\end{align*}
By rearranging the above expression, and sending $\varepsilon$ to $0$, we have
\begin{align}\label{e:lowerboundK}
\left.\del_{\varepsilon}J(T_Y+\varepsilon \tilde T_C, T_{m_\bmla})\right|_{\varepsilon=0}\geq \int \tilde T_C (T_\mu-x) \rd x.
\end{align}
We recall that $J$ as in \eqref{e:logS} is given by $I$ and some explicit integrals. We compute the derivative of $J$,
\begin{align*}\begin{split}
\left.\del_{\varepsilon}J(T_Y+\varepsilon \tilde T_C, T_{m_\bmla})\right|_{\varepsilon=0}
&=\left.2\del_\varepsilon I(T_Y+\varepsilon \tilde T_C, T_{m_\bmla})\right|_{\varepsilon=0}-\frac{1}{2}\int {\tilde T}_C\rd x,\\
&+\frac{1}{2}\int \left(\frac{1}{T_Y(x)-T_Y(y)}-\frac{1}{1-e^{T_Y(y)-T_Y(x)}}\right)(\tilde T_C(x)-\tilde T_C(y))\rd x\rd y.
\end{split}\end{align*}
We will choose $\tilde T_C$ in either the case \eqref{e:nondelta} or \eqref{e:delta}. We notice that in both cases if we replace $\tilde T_C$ by $-\tilde T_C$, both sides of \eqref{e:lowerboundK} change the sign. Therefore, we conclude that
\begin{align*}
\left.\del_{\varepsilon}J(T_Y+\varepsilon \tilde T_C, T_{m_\bmla})\right|_{\varepsilon=0}=\int \tilde T_C (T_\mu-x) \rd x.
\end{align*}
Now if we choose $\tilde T_C$ in \eqref{e:delta}, i.e. $\tilde T_C$ supported on that $\{x: T_Y(x)=a\}$, we have
\begin{align*}
\left.2\del_{\varepsilon}I(T_Y+\varepsilon \tilde T_C, T_{m_\bmla})\right|_{\varepsilon=0}-\frac{1}{2}\int T_C\rd x=\int T_C(x)\rd x\tau({\mathsf m}_\bmla |\mathsf y)(a)-\frac{1}{2}\int T_C\rd x
=\int \tilde T_C (T_\mu-x) \rd x,
\end{align*}
We conclude that $T_\mu(x)=\tau({\mathsf m}_\bmla |\mathsf y)(a)+x-1/2=\tau({\mathsf m}_\bmla |\mathsf y)\circ T_Y(x)+x-1/2$ on $\{x: T_Y(x)=a\}$. Especially, on the intervals  where $T_Y$ is a constant, we have 
\begin{align}\label{e:consteq}
T_\mu(x)=\tau({\mathsf m}_\bmla |\mathsf y)\circ T_Y(x)+x-1/2.
\end{align}
Next we take $\tilde T_C=f(T_Y)$ as in \eqref{e:nondelta},
\begin{align}\begin{split}\label{e:inc}
&\phantom{{}={}}\left.\del_{\varepsilon}J(T_Y+\varepsilon f(T_Y), T_{m_\bmla})\right|_{\varepsilon=0}\\
&=\int f(x)\tau({\mathsf m}_\bmla |\mathsf y)(x)\rd \mu_Y+\int \int \left(\frac{1}{T_Y(x)-T_Y(y)}-\frac{1}{1-e^{T_Y(y)-T_Y(x)}}\right)\rd y f(T_Y(x))\rd x\\
&=\int f(T_Y)\tau({\mathsf m}_\bmla |\mathsf y)\circ T_Y(x)\rd x+\int \int \left(\frac{1}{T_Y(x)-T_Y(y)}-\frac{1}{1-e^{T_Y(y)-T_Y(x)}}\right)\rd y f(T_Y(x))\rd x\\
&=\int f( T_{Y}) (T_\mu-x) \rd x.
\end{split}\end{align}
On the intervals where $T_Y$ is increasing,  \eqref{e:inc} implies that \begin{align}\label{e:inceq}
T_\mu(x)=\tau({\mathsf m}_\bmla |\mathsf y)\circ T_Y(x)+x+\int \left(\frac{1}{T_Y(x)-T_Y(y)}-\frac{1}{1-e^{T_Y(y)-T_Y(x)}}\right)\rd y.
\end{align}
Therefore, we conclude from \eqref{e:consteq} and \eqref{e:inceq} that \eqref{e:inceq} holds almost surely on $[0,1]$, which uniquely determines $\mu$. This finishes the proof of Proposition \ref{p:unique3}.

\end{proof}

\begin{proof}[Proof of Proposition \ref{p:lowerbound3}]
Thanks to the uniqueness of $\mu$, we have that for any $\mu'\neq \mu$ in $\cM^{\rm b}([0, \fK])$
\begin{align*}
\mu_Y\not\in \arg\sup_{\nu\in \cM}\left\{\int (T_{\mu'}-x) T_{\nu}\rd x-J(\nu, m_\bmla)\right\}.
\end{align*}
As a consequence the assumption in Proposition \ref{c:expbound3} holds, 
\begin{align*}
\sup_{\nu\in \cM}\left\{\int (T_{\mu'}-x) T_{\nu}\rd x-J(\nu,m_\bmla)\right\}
>\int (T_\mu-x) T_{\mu_Y}\rd x-J(\mu_Y,m_\bmla).
\end{align*}
and there exists a small $\delta>0$, and positive constant $c(\delta)>0$ such that
\begin{align*}\begin{split}
\sum_{\bmeta_N: m[\bmeta_N]\in \bB_\delta(\mu')}K_{\bmla_N\bmeta_N}m_{\bmeta_N}(e^{Y_N})
\leq e^{-c(\delta)N^2}S_{\bmla_N}(e^{Y_N}).
\end{split}\end{align*}
The space of probability measures $\cM^{\rm b}([0,\fK])$, removing the open $\bB_\delta(\mu)$ is compact, we get a finite open cover $ \cup \bB_{\delta_i}(\mu_i)$,  
\begin{align}\begin{split}\label{e:openball4}
&\phantom{{}={}}\sum_{\bmeta_N: m[\bmeta_N]\in \bB_\delta(\mu)}K_{\bmla_N\bmeta_N}m_{\bmeta_N}(e^{Y_N})\geq S_{\bmla_N}(e^{Y_N})-\sum_i\sum_{\bmeta_N: m[\bmeta_N]\in \bB_{\delta_i}(\mu_i)}K_{\bmla_N\bmeta_N}m_{\bmeta_N}(e^{Y_N})
\\
&\geq \left(1-\sum_ie^{-c(\delta_i)N^2}\right)S_{\bmla_N}(e^{Y_N})= \left(1-\sum_ie^{-c(\delta_i)N^2}\right)\exp\{ N^2(J(\mu_Y, \mu_\bmla)+\oo_N(1))\}.
\end{split}\end{align}
The large deviation lower bound at $\mu$ follows from the estimate \eqref{e:openball4} and \eqref{e:msf}
\begin{align*}
\begin{split}
&\phantom{{}={}}\sum_{\bmeta_N: m[\bmeta_N]\in \bB_\delta(\mu)}K_{\bmla_N\bmeta_N}\\
&=\exp\left\{-N^2\left(\int (T_\mu-x)T_{\mu_Y}\rd x+\oo_{\delta}(1)+\oo_N(1)\right)\right\}\sum_{\bmeta_N: m[\bmeta_N]\in \bB_\delta(\mu)}K_{\bmla_N\bmeta_N}m_{\bmeta_N}(e^{Y_N})\\
&\geq \exp\left\{-N^2\left(\int (T_\mu-x)T_{\mu_Y}\rd x+\oo_{\delta}(1)+\oo_N(1)\right)\right\}\exp\{ N^2(J(\mu_Y, \mu_\bmla)+\oo_N(1))\}\\
&= \exp\left\{- N^2\left(\int (T_\mu-x)T_{\mu_Y}\rd x-J(\mu_Y, m_\bmla)+\oo_{\delta}(1)+\oo_N(1)\right)\right\}.
\end{split}
\end{align*}
\end{proof}

\section{Large Deviation Estimates for Littlewood-Richardson Coefficients}

In this section, we use the spherical integral to study the large deviation of the Littlewood-Richardson coefficients and prove Theorem \ref{t:LR}. From the definition \eqref{e:defLR} of Littlewood-Richardson coefficients,
\begin{align}\label{e:polybound}
c_{\bmla_N\bmeta_N}^{\bmkappa_N}\leq \frac{S_{\bmla_N}(e^{Y_N})S_{\bmeta_N}(e^{Y_N})}{S_{\bmkappa_N}(e^{Y_N})}.
\end{align}
The large deviation upper bound follows from combining the upper bound \eqref{e:polybound} and the asymptotics of Schur symmetric polynomials \eqref{e:logS} ,
\begin{align}\begin{split}\label{e:upbound}
&\frac{1}{N^2}\log \sup_{\bmkappa_N: m[\bmkappa_N]\in \bB_\delta(\mu)}c_{\bmla_N\bmeta_N}^{\bmkappa_N}\leq  -H_{\mu}^{\mathsf{LR}}(\mu_Y)+\oo_{\delta}(1)+\oo_N(1),\\
&H_\mu^{\mathsf{LR}}(\nu)=J(\nu,\mu)-J(\nu, m_\bmla)-J(\nu, m_\bmeta),
\end{split}\end{align}

In the following proposition we collect some properties of the rate function $\mathcal I^{\mathsf{LR}}(\cdot)$ from Theorem \ref{t:LR}. Its proof is very similar to the argument as in Proposition \ref{p:rateprop}, so we omit it. Unfortunately, Proposition \ref{p:rateprop2} does not capture the admissible set for the possible eigenvalues given by Horn's problem.  However it contains the information about the constraints given by the Ky Fan type inequalities, i.e. it equals $+\infty$ outside the region described by the Ky Fan type inequalities:
 \begin{align}
\label{e:limitadm4}
\int_0^1 \left(T_{\mu} -T_{\mu_{\bmla}}-T_{\mu_{\bmeta}}\right)\rd x=0,\quad
\int_y^1 \left(T_{\mu} -T_{\mu_{\bmla}}-T_{\mu_{\bmeta}}\right)\rd x\leq0, \quad \forall y\in [0,1].
 \end{align}

\begin{proposition}
\label{p:rateprop2}
Under the assumptions of Theorem \ref{t:LR}, the function $H_\mu^{\mathsf{LR}}(\cdot)$ and rate function $\mathcal I^{\mathsf{LR}}(\cdot)$ as defined in Theorem \ref{t:LR} satisfy:
\begin{enumerate}
\item For measure $\mu$ satisfies \eqref{e:limitadm4}, $H_\mu^{\mathsf{LR}}(\cdot)$ is  upper semi-continuous in the weak topology on $\{\nu\in \cM: \nu(|x|)\leq \fR\}$ for any $\fR>0$. 
\item If $\int_0^1 \left(T_{\mu} -T_{\mu_{\bmla}}-T_{\mu_{\bmeta}}\right)\rd x\neq 0$, or there exists some $0<y<1$ such that 
\begin{align}\label{e:domainLR}
\int_y^1 \left(T_{\mu} -T_{m_\bmla}-T_{m_\bmeta}\right)\rd x>0, 
\end{align}
then $\mathcal I^{\mathsf{LR}}(\mu)=+\infty$.
\item 
If there exists some small constant $\fc>0$
\begin{align}\begin{split}\label{e:domainLRs}
& \int_y^1 \left(T_{\mu} -T_{m_\bmla}-T_{m_\bmeta}\right)\rd x\leq
\left\{\begin{array}{ll}
 -\fc y,      &\text{ for } 0\leq y\leq \fc,\\
 -\fc,         &\text{ for } \fc\leq y\leq 1-\fc,\\
 -\fc(1-y),  &\text{ for } 1-\fc\leq y\leq 1,
 \end{array}\right.
 \end{split}
\end{align}
then 
$\mathcal I^{\mathsf{LR}}(\mu)=H^{\mathsf{LR}}_\mu(\nu^*)<\infty$ for some probability measure $\nu^{*}$ such that  $\nu^*(|x|)<\infty$.

\item The rate function $\mathcal I^{\mathsf{LR}}(\cdot)$ is nonnegative and lower semicontinuous on $\cM^{\rm b}([0, 2\fK])$. 
\end{enumerate}

\end{proposition}
Finally,  we prove the large deviation lower bound in Theorem \ref{t:LR}. It follows from combining the following Propositions \ref{t:unique4} and \ref{p:lowerbound4}, and noticing the number of partitions with at most $N$ rows in $\bB_\delta(\mu)$ is at most $\exp\{\OO(N\log N)\}$.

\begin{proposition}\label{t:unique4}
We assume the assumptions of Theorem \ref{t:LR}.
 Let $\mu_Y$ be compactly supported and all components of $\supp\mu_Y$ are infinite sets. Then there exists a unique $\mu\in \cM^{\rm b}([0,2\fK])$ such that 
\begin{align*}
\mu_Y\in \arg\sup_{\nu\in \cM}H_\mu^{\mathsf{LR}}(\nu),\quad H_\mu^{\mathsf{LR}}(\nu)=J(\nu,\mu)-J(\nu, m_\bmla)-J(\nu, m_\bmeta).
\end{align*}
\end{proposition}

\begin{proposition}\label{p:lowerbound4}
We assume the assumptions of Theorem \ref{t:LR}. If $\mu\in \cM^{\rm b}([0,2\fK])$ is the unique measure with 
\begin{align*}
\mu_Y\in \arg\sup_{\nu\in \cM}H_\mu^{\mathsf{LR}}(\nu),\quad H_\mu^{\mathsf{LR}}(\nu)=J(\nu,\mu)-J(\nu, m_\bmla)-J(\nu, m_\bmeta).
\end{align*}
Then we have
\begin{align}\label{e:LDP4}
\liminf_{\delta\rightarrow 0}\liminf_{N\rightarrow\infty}\frac{1}{N^2}\log\sum_{\bmkappa_N: m[\bmkappa_N]\in \bB_\delta(\mu)}c^{\bmkappa_N}_{\bmla_N\bmeta_N}\geq -H_\mu^{\mathsf{LR}}(\mu_Y)=-\cI^{\mathsf{LR}}(\mu).
\end{align}
\end{proposition}

The proofs of both Propositions \ref{t:unique4} and \ref{p:lowerbound4} relies on the following probability estimate.
\begin{proposition}\label{c:expbound4}
We assume the assumptions of Theorem \ref{t:LR}.
Let $Y_N=\diag\{y_1,y_2,\cdots, y_N\}$ be a sequence of diagonal matrices, whose spectral measures converge in Wasserstein distance \eqref{e:wd} towards $\mu_Y$.  For any $\mu\in \cM^{\mathsf b}([0,2\fK])$, such that 
\begin{align}\label{e:asup4}
\sup_{\nu\in \cM}\left\{J(\nu,\mu)-J(\nu, m_\bmla)-J(\nu, m_\bmeta)\right\}
>J(\mu_Y,\mu)-J(\mu_Y, m_\bmla)-J(\mu_Y, m_\bmeta).
\end{align}
Then there exists a small $\delta>0$, and positive constant $c(\delta)>0$ such that
\begin{align}\label{e:expbound4}\begin{split}
\phantom{{}={}}\sum_{\bmkappa_N: m[\bmkappa_N]\in \bB_\delta(\mu)}c^{\bmkappa_N}_{\bmla_N\bmeta_N}S_{\bmkappa_N}(e^{Y_N})
\leq e^{-c(\delta)N^2}S_{\bmla_N}(e^{Y_N})S_{\bmeta_N}(e^{Y_N}).
\end{split}\end{align}
\end{proposition}

\begin{proof}[Proof of Proposition \ref{c:expbound4}]
Under Assumption \eqref{e:asup4}, for sufficiently small $\varepsilon>0$, there exists a measure $\nu\in \cM$ such that
\begin{align}\label{e:boundLR}
J(\nu,\mu)-J(\nu, m_\bmla)-J(\nu, m_\bmeta)
>J(\mu_Y,\mu)-J(\mu_Y, m_\bmla)-J(\mu_Y, m_\bmeta)+\varepsilon.
\end{align}
We divide $S_{\bmla_N}(e^{Y_N})S_{\bmeta_N}(e^{Y_N})$ on both sides of \eqref{e:expbound4}, and use the asymptotics of Schur symmetric polynomials \eqref{e:logS}, 
\begin{align*}\begin{split}
&\phantom{{}={}}(1/(S_{\bmla_N}(e^{Y_N})S_{\bmeta_N}(e^{Y_N})))\sum_{\bmkappa_N: m[\bmkappa_N]\in \bB_\delta(\mu)}c^{\bmkappa_N}_{\bmla_N\bmeta_N}S_{\bmkappa_N}(e^{Y_N})
\\ 
&=
e^{ N^2(J(\mu_Y,\mu)-J(\mu_Y, m_\bmla)-J(\mu_Y, m_\bmeta)+\oo_\delta(1)+\oo_N(1))}\sum_{\bmkappa_N: m[\bmkappa_N]\in \bB_\delta(\mu)}c^{\bmkappa_N}_{\bmla_N\bmeta_N}\\
&\leq  
e^{N^2(J(\mu_Y,\mu)-J(\mu_Y, m_\bmla)-J(\mu_Y, m_\bmeta)+\oo_\delta(1)+\oo_N(1))}
e^{- N^2 (J(\nu,\mu)-J(\nu, m_\bmla)-J(\nu, m_\bmeta)+\oo_\delta(1)+\oo_N(1))}\\
&\leq e^{- N^2(\varepsilon+\oo_\delta(1)+\oo_N(1))},
\end{split}
\end{align*}
where in the first equality we used the large deviation upper bound \eqref{e:upbound}, and \eqref{e:boundLR} in the last inequality. The claim follows provided we take $\delta$ sufficiently small.
\end{proof}

\begin{proof}[Proof of Proposition \ref{t:unique4}]
We prove the existence of such $\mu$ by contradiction. If there is no such $\mu$, i.e. for any measure $\mu\in \cM^{\rm b}([0,2\fK])$, we have 
\begin{align*}
\mu_Y\not\in\arg\sup_{\nu\in \cM}\{J(\nu,\mu)-J(\nu, m_\bmla)-J(\nu, m_\bmeta)\}.
\end{align*}
Then it follows from proposition \ref{c:expbound4}, there exists a small $\delta>0$, and positive constant $c(\delta)>0$ such that
\begin{align*}\begin{split}
\sum_{\bmkappa_N: m[\bmkappa_N]\in \bB_\delta(\mu)}c^{\bmkappa_N}_{\bmla_N\bmeta_N}S_{\bmkappa_N}(e^{Y_N})
\leq e^{-c(\delta)N^2}S_{\bmla_N}(e^{Y_N})S_{\bmeta_N}(e^{Y_N}).
\end{split}\end{align*}
Since the space of probability measure supported on $[0, 2\fK]$ with density bounded by $1$ is compact, we get a finite open cover $ \cup \bB_{\delta_i}(\mu_i)$ of $\cM^{\rm b}([0, 2\fK])$,  we get the contradiction
\begin{align*}\begin{split}
&
\phantom{{}={}}S_{\bmla_N}(e^{Y_N})S_{\bmeta_N}(e^{Y_N})
=
\sum_i\sum_{\bmkappa_N: m[\bmkappa_N]\in \bB_{\delta_i}(\mu_i)}c^{\bmkappa_N}_{\bmla_N\bmeta_N}S_{\bmkappa_N}(e^{Y_N})\\&
\leq \sum_i e^{-c(\delta_i)N^2}S_{\bmla_N}(e^{Y_N})S_{\bmeta_N}(e^{Y_N})<S_{\bmla_N}(e^{Y_N})S_{\bmeta_N}(e^{Y_N}).
\end{split}
\end{align*}

The proof of the uniqueness of Proposition \ref{t:unique4} is where we need the regularity of the measure $\mu_Y$. If $\mu_{Y}$ is a maximizer  of $H^{\mathsf{LR}}_{\mu}(\cdot)$,
 then 
 \begin{align}\begin{split}\label{e:derLR}
&0=\left.\del_{\varepsilon}H^{\mathsf{LR}}_{\mu}(T_Y+\varepsilon \tilde T_C)\right|_{\varepsilon=0}\\
&=\left.\del_{\varepsilon}J(T_Y+\varepsilon \tilde T_C, T_\mu)\right|_{\varepsilon=0}
-\left.\del_{\varepsilon}J(T_Y+\varepsilon \tilde T_C, T_{m_{\bmla}})\right|_{\varepsilon=0}
-\left.\del_{\varepsilon}J(T_Y+\varepsilon \tilde T_C, T_{m_\bmeta})\right|_{\varepsilon=0}.
\end{split}\end{align}
We take $\tilde T_C=f(T_Y)$ and recall from \eqref{e:inc}, we have
\begin{align}\begin{split}\label{e:incLR}
&\phantom{{}={}}\left.\del_{\varepsilon}J(T_Y+\varepsilon f(T_Y), T_\mu)\right|_{\varepsilon=0}\\
&=\int f(x)\tau(\mu |\mathsf y)(x)\rd \mu_Y+\int \int \left(\frac{1}{x-y}-\frac{1}{1-e^{y-x}}\right)\rd \mu_Y(y) f(x)\rd \mu_Y(x).
\end{split}
\end{align}
Using \eqref{e:incLR}, we can simplify \eqref{e:derLR} as
\begin{align}\begin{split}\label{e:iic}
0&=\left.\del_{\varepsilon}J(T_Y+\varepsilon \tilde T_C, T_\mu)\right|_{\varepsilon=0}
-\left.\del_{\varepsilon}J(T_Y+\varepsilon \tilde T_C, T_{m_{\bmla}})\right|_{\varepsilon=0}
-\left.\del_{\varepsilon}J(T_Y+\varepsilon \tilde T_C, T_{m_\bmeta})\right|_{\varepsilon=0}\\
&=\int f(x)(\tau(\mu|\mathsf y)(x)-\tau(\mathsf m_\bmla|\mathsf y)(x)-\tau(\mathsf m_\bmeta|\mathsf y)(x))\rd \mu_Y(x)\\
&-\int \int \left(\frac{1}{x-y}-\frac{1}{1-e^{y-x}}\right)\rd \mu_Y(y) f(x)\rd \mu_Y(x).
\end{split}\end{align}
Since we can take $f$ any bounded Lipschitz function, we conclude from \eqref{e:iic} that $\mu_Y$-almost surely for any $x\in \bR$,
\begin{align}
\label{e:dsumLR}
\tau(\mu|\mathsf y)(x)=\tau(\mathsf m_{\bmla}|\mathsf y)(x)+\tau(\mathsf m_{\bmeta}|\mathsf y)(x)+\int \left(\frac{1}{x-y}+\frac{1}{1-e^{y-x}}\right)\rd \mu_Y(y).
\end{align}
We denote the solutions of the variational problem $I(\mu_Y, \mu), I(\mu_Y,m_{\bmla}), I(\mu_Y,m_{\bmeta})$ as given in \eqref{e:funcS} by $f^{\mu_Y\rightarrow \mu}(t,x), f^{\mu_Y\rightarrow m_{\bmla}}(t,x), f^{\mu_Y\rightarrow m_{\bmeta}}(t,x)$ respectively.
Then Item 7 in Proposition \ref{propbur} gives that
\begin{align*}\begin{split}
&\tau(\mu|\mathsf y)(x)=\Re [f^{\mu_Y\rightarrow \mu}(0,x)]+H\mu_Y-x,\\
&\tau(m_{\bmla}|\mathsf y)(x)=\Re [f^{\mu_Y\rightarrow m_{\bmla}}(0,x)]+H\mu_Y-x,\\
&\tau(m_{\bmeta}|\mathsf y)(x)=\Re [f^{\mu_Y\rightarrow m_{\bmeta}}(0,x)]+H\mu_Y-x.
\end{split}\end{align*}
Therefore \eqref{e:dsumLR} implies that the real part of $f^{\mu_Y\rightarrow \mu}(0,x)$ is given by
\begin{align*}\begin{split}
\Re [f^{\mu_Y\rightarrow \mu}(0,x)]
&=\Re [f^{\mu_Y\rightarrow \mu_{\bmla}}(0,x)]+\Re [f^{\mu_Y\rightarrow \mu_{\bmeta}}(0,x)]\\
&+H\mu_Y-x+\int \left(\frac{1}{x-y}+\frac{1}{1-e^{y-x}}\right)\rd \mu_Y(y),
\end{split}\end{align*}
for $\mu_Y$-almost surely all $x\in \bR$. The imaginary part of $f^{\nu\rightarrow \mu}(0,.)$ is given by the measure $\mu_Y$. By our assumption that $\mu_Y$ is compactly supported and all components of $\supp\mu_Y$ are infinite sets, we conclude from Corollary \ref{c:uniquemu} that $\mu$ is uniquely determined by $\mu_Y$ and $f^{\mu_Y\rightarrow \mu}(0,x)$. This finishes the proof of Proposition \ref{t:unique4}.
\end{proof}

\begin{proof}[Proof of Proposition \ref{p:lowerbound4}]
Thanks to the uniqueness of $\mu$, we have that for any $\mu'\neq \mu$ in $\cM^{\rm b}([0, 2K])$
\begin{align*}
\mu_Y\not\in \arg\sup_{\nu\in \cM}\{J(\nu,\mu')-J(\nu, m_\bmla)-J(\nu, m_\bmeta)\}.
\end{align*}
As a consequence the assumption in Proposition \ref{c:expbound4} holds, 
\begin{align*}
\sup_{\nu\in \cM}\left\{J(\nu,\mu')-J(\nu, m_\bmla)-J(\nu, m_\bmeta)\right\}
>J(\mu_Y,\mu')-J(\nu, m_\bmla)-J(\nu, m_\bmeta),
\end{align*}
and there exists a small $\delta>0$, and positive constant $c(\delta)>0$ such that
\begin{align}\begin{split}\label{e:hhop}
\sum_{\bmkappa_N: m[\bmkappa_N]\in \bB_\delta(\mu')}c^{\bmkappa_N}_{\bmla_N\bmeta_N}S_{\bmkappa_N}(e^{Y_N})
\leq e^{-c(\delta)N^2}S_{\bmla_N}(e^{Y_N})S_{\bmeta_N}(e^{Y_N}).
\end{split}\end{align}
 Again the  space of probability measure supported on $[0, 2\fK]$ with density bounded by $1$ removing the open ball $\bB_\delta(\mu)$ is compact, we get a finite open cover $ \cup \bB_{\delta_i}(\mu_i)$,  
\begin{align}\begin{split}\label{e:openball}
&\sum_{\bmkappa_N: m[\bmkappa_N]\in \bB_\delta(\mu)}c^{\bmkappa_N}_{\bmla_N\bmeta_N}S_{\bmkappa_N}(e^{Y_N})\geq S_{\bmla_N}(e^{Y_N})S_{\bmeta_N}(e^{Y_N})
-\sum_i \sum_{\bmkappa_N: m[\bmkappa_N]\in \bB_{\delta_i}(\mu_i)}c^{\bmkappa_N}_{\bmla_N\bmeta_N}S_{\bmkappa_N}(e^{Y_N})\\
&\geq \left(1-\sum_ie^{-c(\delta_i)N^2}\right)S_{\bmla_N}(e^{Y_N})S_{\bmeta_N}(e^{Y_N}),
\end{split}\end{align}
where we used \eqref{e:hhop} in the last line.
The large deviation lower bound at $\mu$ follows from the estimates \eqref{e:openball} and  \eqref{e:logS},
\begin{align*}
\begin{split}
\sum_{\bmkappa_N: m[\bmkappa_N]\in \bB_\delta(\mu)}c^{\bmkappa_N}_{\bmla_N\bmeta_N}
&=e^{-N^2 (J(\mu_Y, \mu)+\oo_\delta(1)+\oo_N(1))}\sum_{\bmkappa_N: m[\bmkappa_N]\in \bB_\delta(\mu)}c^{\bmkappa_N}_{\bmla_N\bmeta_N}S_{\bmkappa_N}(e^{Y_N})\\
&\geq e^{-N^2 (J(\mu_Y, \mu)+\oo_\delta(1)+\oo_N(1))}e^{N^2 (J(\mu_Y, m_\bmla)+J(\mu_Y, m_\bmeta)+\oo_N(1))}\\
&= e^{- N^2(J(\mu_Y, \mu)-J(\mu_Y, m_\bmla)-J(\mu_Y,m_\bmeta)+\oo_\delta(1)+\oo_N(1))}.
\end{split}
\end{align*}
\end{proof}

\section{Properties of $f(t,x)$}\label{s:f}


We recall that given two {compactly supported} Borel probability measures $\mu, \nu$ on $\mathbb R$, $(\rho^*, u^*)$ is the unique solution of the variational problem 
\eqref{e:Iexp}, and $f(t,x)=u_t^*(x)+\pi \ri \rho_t^*(x)$ satisfies the Burgers equation \eqref{burg}. We recall the set $\Omega$ from \eqref{Omega},
\begin{align}\label{e:defO}
\Omega=\{(t,x)\in (0,1)\times\bR : \rho_t^*(x)>0\}.
\end{align}
In this section we study the continuity properties of $f(t,x)$ at boundaries, and prove Theorem \ref{main} and Corollary \ref{c:uniquemu} showing that $f(0,.)$ uniquely determines 
$f(1,.)$. 
{For the rest of this section, we assume that neither $\mu$ nor $\nu$ are Dirac masses. If supplementary hypotheses are required in order for certain results to hold, we
will state them when needed.}

We introduce the following notation: given an open set $\Omega$, a function $\omega\colon\Omega\to
\mathbb C$, and a subset ${\bf D}\subseteq\overline{\Omega}$, 
we denote by $\bC_\Omega(\omega,{\bf D})$ the set of all points $w\in\mathbb C\cup\{\infty\}$ for which there exists 
a sequence $\{w_k\}_{k\in\mathbb N}\subset\Omega$ such that $\lim_{n\to\infty}\sup_{k>n}\text{dist}(w_k,{\bf D})=0$ and $\lim_{k\to\infty}\omega(w_k)=w$.
When there is no risk of confusion, we suppress the domain $\Omega$ from the notation and denote $\bC_\Omega(\omega,{\bf D})$ simply by $\bC(\omega,{\bf D})$.
Also, if ${\bf D}=\{x\}$ is a singleton, we write $\bC(\omega,x)$ rather than $\bC(\omega,\{x\})$. From Proposition \ref{propbur}, we deduce

\begin{proposition}\label{propbur2}
Assume that  the probability measures $\mu,\nu$ are compactly supported. The pair $(\rho^*, u^*)$ is the unique solution of the variational problem \eqref{e:Iexp}, and 
$f(t,x)=u_t^*(x)+\pi \ri \rho_t^*(x)$ satisfies the Burgers equation \eqref{burg}.
 For any $t\in(0,1),$ the function $x\mapsto\Im [f(t,x)]$ extends  continuously to all points $(t,x)\in\partial\Omega,$ where it takes the value zero. Moreover, for any $m\in(0,1)$ 
and any connected subset ${{\bf D}\subseteq\partial\Omega\cap([0,m]\times\mathbb R)}$, the function $g(t,x)=x-tf(t,x)$ satisfies the condition $\bC(g,{\bf D})\subset\mathbb R,$ 
and $\bC(g,{\bf D})$ is bounded.
\end{proposition}

\subsection{Properties of the solution of a special case of the Beltrami equation}
Our main tool in this section is a change of variable, transforming \eqref{burg} into a Beltrami equation (see, for instance, \cite[Section 4 of Chapter I]{MR867407}).
We recall from \eqref{e:Bel} that $f(z)$ satisfies the following Beltrami equation
\begin{equation}\label{Bel}
\partial_{\bar{z}}f=\frac{\ri-f}{\ri+f}\partial_{z}f,\quad z=x-\ri t,\quad (-\Im[z],\Re[z])\in\Omega.
\end{equation}
We also define the functions $g(t,x)$ and $h(t,x)$ by
\begin{equation}\label{gh}
g(t,x)=x-tf(t,x),\quad h(t,x)=x+(1-t)f(t,x),\quad (t,x)\in\Omega.
\end{equation}
Then $(\partial_x-\ri\partial_t)g=1-t\partial_xf+\ri f+\ri t\partial_tf=1+\ri f-t\partial_xf-\ri tf\partial_xf=(1+\ri f)(1-t\partial_xf)$ and $(\partial_x+\ri\partial_t)g=
1-t\partial_xf-\ri f+\ri tf\partial_xf=(1-\ri f)(1-t\partial_xf).$ Thus, 
\begin{align*}
\partial_{\bar{z}}g=(\partial_x-\ri \partial_t)g=\frac{(1+\ri f)(1-t\partial_xf)}{(1-\ri f)(1-t\partial_xf)}(\partial_x+\ri\partial_t)g
=\frac{\ri-f}{\ri+f}\partial_zg.
\end{align*}
That is, $f$ and $g$ satisfy the same Beltrami equation \eqref{Bel}. Since $h=g+f$, $h$ also satisfies \eqref{Bel}.

The classical version of the Beltrami equation, for whose study we refer the reader to \cite{MR2472875,Letho,MR867407}, is of the form 
$$
\partial_{\bar{z}}f={\boldsymbol\mu}(z)\partial_{z}f,
$$
where the {\em complex dilatation} ${\boldsymbol\mu}(z)$ is a measurable function defined on some open subset of $\mathbb C$ on which its $L^\infty$ norm is {\em strictly} less than 
one: there exists $K\in[1,+\infty)$ such that $\|{\boldsymbol\mu}\|_\infty\leq\frac{K-1}{K+1}$. Conventionally such functions $\mu$ are extended to all of $\mathbb C$ by being 
assigned the value zero outside this open subset of $\mathbb C$. A solution of such an equation is $K$-quasiregular\footnote{We refer here exclusively to solutions in the
Sobolev class $W_{\rm loc}^{1,p},p\ge2$ - our $f$ does satisfy that condition.}. It is known that for any $K\in(1,+\infty)$, there exists a $K$-quasiconformal (i.e. $K$-quasiregular 
and injective) solution of the Beltrami equation \cite[Theorem 5.1.2]{MR2472875} which is defined on the whole complex plane and is analytic outside the support of the function 
${\boldsymbol\mu}$. If one imposes a certain normalization on the solution, it becomes unique. Moreover, any other solution of the Beltrami equation is obtained by composing this 
solution with an analytic map. The regularity of the solutions of the Beltrami equation is influenced by the regularity of the function ${\boldsymbol\mu}$ in the obvious way, roughly 
speaking. The reader is justified to object to our naming \eqref{Bel} a Beltrami equation, since it is known that in our case ${\boldsymbol\mu}(z)=\frac{\ri-f(z)}{\ri+f(z)}$ satisfies 
$\|{\boldsymbol\mu}\|_\infty=1$: indeed, as seen in Proposition \ref{propbur2}, our ${\boldsymbol\mu}(z)$ extends continuously to the boundary of $\Omega\cap((0,1)\times\mathbb 
R)$ with values in the unit circle of $\mathbb C$. By now there is a fairly vast literature regarding Beltrami equations for which $\|{\boldsymbol\mu}\|_\infty=1$. In that case, it is 
not always true that the Beltrami equation has a solution, and certainly not a quasiconformal solution that extends to all the complex plane. Fortunately we are affected only marginally 
by this problem, thanks to two crucial facts: first, we are already guaranteed the existence of a sufficiently regular solution, namely $f$ itself, and second, the set on which our $
{\boldsymbol\mu}(z)=\frac{\ri-f(z)}{\ri+f(z)}$ is defined, namely $\Omega$, is conveniently described by the property $\Im f(z)>0$, i.e. $|{\boldsymbol\mu}(z)|<1$. In particular, $
(t,x)\in\partial\Omega\cap((0,1)\times\mathbb R)\iff|{\boldsymbol\mu}(z)|=1,z=x-{\rm i}t$. Let us collect below some of the properties $f,g,h$ (and most other solutions of \eqref{Bel})
must satisfy as a consequence of these properties of ${\boldsymbol\mu}(z)$ and $\Omega$. We refer to \cite{Letho,MR867407} for details.
\begin{enumerate}
\item $f,g,h$ are open, in the sense that the image of any open subset of $\Omega$ via one of these functions is an open set. Indeed, choose an arbitrary point $z_0\in\Omega$.
Continuity of $\Im f$ on $\Omega$ (Proposition \ref{propbur}, Item 1) together with the definition of $\Omega$ (Equation \eqref{e:defO}) guarantee that there exists a connected 
neighborhood $V$ of $z_0$ such that $\overline{V}\subset\Omega$ (thus, $\varepsilon_0:=\min\{t,1-t\colon (t,x)\in\overline{V}\}>0$), and $\varepsilon_1:=
\min_{\overline{V}}\Im f>0.$ In particular, by Proposition \ref{propbur}, Item 4, $|{\boldsymbol\mu}(z)|^2=\left|\frac{\ri-f(z)}{\ri+f(z)}\right|^2\leq
\max\left\{\frac{\mathfrak K^2+\varepsilon_0^2(1-\varepsilon_1)^2}{\mathfrak K^2+\varepsilon_0^2(1+\varepsilon_1)^2},
\frac{(\varepsilon_0-\mathfrak K)^2+\mathfrak K^2}{(\varepsilon_0+\mathfrak K)^2+\mathfrak K^2}\right\}<1$ for all $z\in\overline{V}$, so that $f$ is a solution of the 
Beltrami equation \eqref{Bel} when restricted to $V$. As $f$ is continuously $\mathbb R$-differentiable on $\overline{V}$ (Proposition \ref{propbur}, Item 1), it is quasiregular 
\cite[Theorems 5.3.4 and 5.5.1]{MR2472875}, hence open or constant on $V$ (see \cite[Corollary 5.5.2]{MR2472875}). As Biane's result \cite[Corollary 4]{Biane} guarantees 
$x\mapsto\Im f(t,x)$ is not constant, $f$ must be open on $V$. This applies to a small enough neighborhood of any point in $\Omega$, so $f$ is open on $\Omega$. Since 
$\Im g=-t\Im f,\Im h=(1-t)\Im f$, the above argument applies verbatim to show that $g$ and $h$ (and most linear combinations of $f,g,h$, for that matter) are open on $\Omega$.
\label{Open}

\item $f,g,h$ are {\em light}, meaning that the preimage of any point via one of these functions is totally disconected in $\Omega$. The argument parallels the one above: assume 
towards contradiction that there exists a point $w_0\in\mathbb C$ such that $f^{-1}(\{w_0\})$ has a limit point, call it $z_0$, in $\Omega$. As argued above, there exists a  connected 
neighborhood $V$ of $z_0$ such that $\overline{V}\subset\Omega$ and $\max_{z\in\overline{V}}|{\boldsymbol\mu}(z)|=\max_{z\in\overline{V}}\left|\frac{{\rm i}-f(z)}{{\rm i}+f(z)}\right|<1$.
This makes $f$ a (smooth) non-constant solution of the classical Beltrami equation \eqref{Bel} on $V$, ensuring that $f^{-1}(\{w_0\})$ cannot have any accumulation point in $V$, and 
in particular that it must be totally disconnected in $V$. This provides the desired contradiction.
\label{Light}

\item None of $f,g,h$ have a local maximum or minimum in $\Omega$. The lack of a local maximum is known (and a simple exercise in elementary topology to prove - for a stronger 
result, see Corollary \ref{maxi} and its proof) for complex-valued open maps. The lack of a minimum is based on the fact that $\Im f>0$ on $\Omega$ by definition of $\Omega$, so, 
as seen in the previous point, $\Im h>0,\Im g<0$ on $\Omega$. In particular, zero cannot be a local minimum value for either of these functions, so the argument for the lack of local 
maximum applies to prove the lack of a local minimum. 

\item $\partial f({\bf D})\subseteq f(\partial{\bf D})$ for any open set $\varnothing \neq {\bf D} \subseteq\Omega$, and the same for $g,h$. If $f$ does not extend 
continuously to $\partial{\bf D}$, then $f(\partial{\bf D})$ should be understood as the set of limit points of $f$ along sequences in ${\bf D}$ converging to points in $\partial{\bf D}$.
This is a general property of open continuous mappings, and arguments as in part \ref{Light}. above apply (see also \cite{MR0082545}).

\item The branching points of $f$ are the critical points $\{(t,x)\in(0,1)\times\mathbb R\colon\partial_tf(t,x)=0=\partial_xf(t,x)\}$; at all other points, $f$ is a local homeomorphism.
This follows easily by viewing $f$ as a map between open subsets of $\mathbb R^2$, computing its Jacobian and applying either of \eqref{Bel} or \eqref{burg}. Similar
statements (with the obvious modifications) hold for $g,h$.
\end{enumerate}

We focus next on some simple but useful topological properties of $\Omega$.

\begin{lemma}\label{simply}
Any connected component of $\Omega$ as defined in \eqref{e:defO} is simply connected.
\end{lemma}

\begin{proof}
By definition, $\Omega=\{(t,x)\in(0,1)\times\mathbb R\colon\Im[ f_t(x)]>0\}$. 
Assume towards contradiction that a smooth simple closed curve $\gamma$ in $\Omega$
is not homotopy equivalent to a point in $\Omega$. Let $C$ be the intersection of $\Omega$ with the bounded component 
of $\mathbb C\setminus\gamma$. By continuity, $f(\gamma)\subset\mathbb C^+$ is a compact set 
(hence bounded and bounded away from $\mathbb R$) and $f(\partial C\setminus\gamma)\subset\mathbb R$ is a bounded
set according to Item 4 in Proposition \ref{propbur}.
Thus, $\partial f(C)\subset f(\gamma)\cup f(\partial C\setminus\gamma)$ 
is a bounded set in $\mathbb R\cup(\mathbb C^++\ri\varepsilon)$ for some $\varepsilon>0$.
This implies that $f(C)$ cannot be bounded,  contradicting Item $4$ in Proposition \ref{propbur}.
\end{proof}

As an easy consequence of the conservation of mass, we have
\begin{lemma}\label{l:mass}
If $\Omega_0$ is a connected component of $\Omega$, then the function \!$[0,1]\!\ni\!t\mapsto\int_{\bar\Omega_0\cap(\{t\}\!\times\mathbb R)}\Im [f(t,x)]\,{\rm d}x \in(0,1]$ is constant.
\end{lemma}

Thanks to the weak continuity of $t\mapsto\Im f(t,\cdot)$, an interesting phenomenon is illustrated by
the above lemmas, and particularly the above conservation of mass: in order for $\Omega$ to fail being 
connected (and hence simply connected), it is necessary (but not sufficient) that both $\supp\mu$ and 
$\supp\nu$ have more than one connected component, and there exist strict subsets $\{C_1^\mu,
\dots,C_\ell^\mu\}$ of the set of connected components of $\supp\mu$ and $\{C_1^\nu,\dots,
C_m^\nu\}$ of that of $\supp\nu$ such that $\mu(C_1^\mu\cup\cdots\cup C_\ell^\mu)=\nu
(C_1^\nu\cup\cdots\cup C_m^\nu),$ for some $1\leq\ell,m\in\mathbb N\cup\{\infty\}$.

Let us discuss next some of the properties of the boundary of $\Omega$. As $\overline{\Omega}\subset[0,1]\times\mathbb R$, we shall not hesitate in the following 
to refer to parts of $\overline{\Omega}$ at (or nearer to) $\{1\}\times\mathbb R$ as ``upper'' and to parts  of $\overline{\Omega}$ at (or nearer to) $\{0\}\times\mathbb R$ 
as ``lower'', and similarly for ``left'' and ``right''. However, we warn the reader that our change of variable $z=x-\ri t$ actually places $\overline{\Omega}\subset
\{z\in\mathbb C\colon\Im z\in[-1,0]\}$, so, when viewed as the set on which the solution(s) of \eqref{Bel} are defined, ``upper'' and ``lower'' become misnomers. 
Nevertheless, we shall stick with these expressions, hoping the reader will keep in mind our conventions and avoid confusion. First, of course, it follows immediately from 
\cite{Biane} and Proposition \ref{propbur}, Item 6, that $\partial\Omega\cap(\{0\}\times\mathbb R)=\supp\nu,\partial\Omega\cap(\{1\}\times\mathbb R)=\supp\mu.$
As shown in \cite[Proposition 3]{Biane}, if $t\in(0,1)$, then the support of $x\mapsto\Im f(t,x)=\rho^*_t(x)$ is the closure of its interior, and the boundary (in $\mathbb R)$ of this support is given 
in terms of the Cauchy-Stieltjes transform of the distribution of  $\sf x_t=(1-t)\sf a+t\sf b+\sqrt{t(1-t)}\sf s$, where $\sf{s,a,b}$ are as in Proposition \ref{propbur}.
As it bears some significance to our paper, we give here the outline of that description. Thus, given a Borel probability measure $\lambda$ on $\mathbb R$, we denote by 
$$
G_\lambda(z)=\int_\mathbb R\frac{{\rm d}\lambda(s)}{z-s},\quad z\in\mathbb C\setminus\supp\lambda,
$$
the Cauchy-Stieltjes transform of $\lambda$. It is analytic on its domain, it sends the upper half-plane $\mathbb C^+$ into the lower half-plane $\mathbb C^-$ and vice-versa, and is 
monotonically decreasing on each interval in $\mathbb R\setminus\supp\lambda$. Moreover, if $\supp\lambda$ is compact, then $G_\lambda(z)=\sum_{n=0}^\infty m_n(\lambda)z^{-
n-1}$ for $|z|>\max_{x\in\supp\lambda}|x|$, where $m_n(\lambda)$ is the $n^{\rm th}$ moment of $\lambda$. We denote by $\gamma_r$ the centered semicircular distribution of 
variance $r$ (i.e. the distribution of $\sqrt{r}\sf{s}$) and by $\lambda_t$ the distribution of $(1-t)\sf{a}+t\sf{b}$. Thus, the distribution of $\sf{x}_t=(1-t)\sf a+t\sf b+\sqrt{t(1-t)}\sf 
s$ is $\gamma_{t(1-t)}\boxplus\lambda_t$. It follows from the expression $G_{\lambda_t\boxplus\gamma_{t(1-t)}}(z)=\tau\left((z-(1-t)\sf a-t\sf b-\sqrt{t(1-t)}\sf s)^{-1}\right)$, 
provided by Proposition \ref{propbur}, Item 6, that the correspondence $(z,t)\mapsto G_{\lambda_t\boxplus\gamma_{t(1-t)}}(z)$ is analytic as a function of two complex variables; the 
domain of analyticity includes $\{(t,z)\colon t\in(0,1),z\in\mathbb C\setminus\supp(\lambda_t\boxplus\gamma_{t(1-t)})\}$ as a strict subset. It is known from \cite[Proposition 2]{Biane} 
that there exists a self-map $\omega_t$ of $\mathbb C^+$ which satisfies the equalities $\omega_t(z)=z-t(1-t)G_{\gamma_{t(1-t)}\boxplus\lambda_t}(z)=z-t(1-t)G_{\lambda_t}
(\omega_t(z)),$ $z\in\mathbb C\setminus\supp(\gamma_{t(1-t)}\boxplus\lambda_t)$. Direct computation shows that $\omega_t$ is the compositional inverse to the right of the map 
$H_t(w)=w+t(1-t)G_{\lambda_t}(w)$. For a point $x_0$ to be an endpoint of a component of the support of $\gamma_{t(1-t)}\boxplus\lambda_t$ it is necessary that either 
$\omega_t(x_0)$ is an endpoint of a connected component of $\supp\lambda_t$ and $H'(w)>0$ on one side of $\omega_t(x_0)$, or $\omega_t(x_0)\in\mathbb R$ belongs to the 
domain of analyticity of $H_t$, $H_t'(\omega_t(x_0))=0$, and $H'(w)>0$ on one side of $\omega_t(x_0)$. (Generally, points $x\in\mathbb R\setminus\supp(\lambda_t\boxplus
\gamma_{t(1-t)})$ satisfy the condition $G'_{\lambda_t}(\omega_t(x))\in\left(\frac{-1}{t(1-t)},0\right)$, or, equivalently, $H'_t(\omega_t(x))>0$ - see \cite[Lemma 4 and Corollary 2]{Biane}.)

A careful analysis of these two conditions, using the tools and methods of \cite{Biane}, allows us to conclude regarding the boundary of $\Omega$ that

\begin{remark}\label{curve}
Each connected component of $\partial\Omega\cap((0,1)\times\mathbb R)$ is a locally analytic or a locally Lipschitz curve.
\end{remark}
We leave the details of the proof of this remark to the reader.

Another interesting fact about $\partial\Omega$ is that it cannot create horizontal waves: 
\begin{remark}
When viewed as a curve in $\{z\in\mathbb C\colon\Im z\in(-1,0)\}$, no component of $\partial\Omega$ can have a local maximum of the imaginary part of its coordinate followed by a 
local minimum or vice-versa $($we call such a succession a ``wave''$)$, and neither can it contain a horizontal segment.
\end{remark}
Indeed, let us discard first the possibility of a horizontal segment $\{t_0\}\times[m,n],t_0\in(0,1),$ in the boundary of $\Omega$: based on the above description of the boundary
of the support of the distribution of $\mathsf x_{t_0}$ and the boundedness of the variables $\mathsf{a,b}$, this is a trivial consequence of the form of 
$G_{\lambda_t\boxplus\gamma_{t(1-t)}}(z)=\tau\left((z-(1-t)\sf a-t\sf b-\sqrt{t(1-t)}\sf s)^{-1}\right)$. Thus, we may discard the 
possibility of the existence of a horizontal segment in $(\partial\Omega)\cap\{z\colon\Im z\in(-1,0)\}$. The lack of waves is a consequence of the conservation of mass (Lemma 
\ref{l:mass}) and the simple connectedness of the components of $\Omega$ (Lemma \ref{simply}): it follows from Lemma \ref{simply} that any connected component of 
$\partial\Omega$ must touch at least one of $\{0\}\times\mathbb R$, $\{1\}\times \mathbb R$. Let us say that $(t_M,x_M)\in\partial\Omega$ is such that the horizontal 
line drawn through it separates some Euclidean ball $B_\iota^M=\{(t,x)\in(0,1)\times\mathbb R\colon(t-t_M)^2+(x-x_M)^2<\iota^2\}$ into two parts such that 
$\{(t,x)\in B_\iota^M\colon t>t_M\}\cap\Omega=\varnothing,\{(t,x)\in B_\iota^M\colon t<t_M\}\cap\Omega\neq\varnothing$. Having a wave simply means that there 
exist $(0,1)\ni t_m<t_M,x_M,x_m\in\mathbb R$ such that $(t_M,x_M),(t_m,x_m)$ belong to the same component of $(\partial\Omega)\cap\{z\colon\Im z\in(-1,0)\}$ and in 
addition $\{(t,x)\in B_\iota^m\colon t>t_m\}\cap\Omega^{c}\neq\varnothing,\{(t,x)\in B_\iota^m\colon t<t_m\}\cap\Omega\subset\Omega$.
Pick $0<\iota<(t_M-t_m)/2$. Consider the subset $\Omega'=\{(t,x)\in\Omega\colon1>t>t_M-\iota\}$. By our choice of $\iota$, this forces $(t_m,x_m)\not\in\Omega'$.
Thus, by taking $\iota$ sufficiently small, we insure that the boundary of the connected component of $\Omega'$ in whose boundary $(t_M,x_M)$ lays contains only a segment
$\{t_M-\iota\}\times[p,r]$ and a Lipschitz arc uniting $p$ and $r$, touching the line $\{t_M\}\times\mathbb R$ only in $(t_M,x_M)$, and otherwise staying entirely under this line.
This is a direct violation of the conservation of mass (Lemma \ref{l:mass}).

Finally, let us state a crucial result for our purposes, the Stoilow factorization, as it applies to our context (see any of \cite[Section 5.5]{MR2472875}, \cite[Page 138]{Letho},
or \cite{MR0082545}).

\begin{proposition}\label{Stoilow}
Let $\Omega_0$ be a connected component of $\Omega$. Then there exists a homeomorphism $W\colon\Omega_0\to\mathbb C^+$ and analytic functions $F,H\colon\mathbb C^+
\to\mathbb C^+,G\colon\mathbb C^+\to\mathbb C^-$ such that $f=F\circ W,g=G\circ W,h=H\circ W$ when restricted to $\Omega_0$. Moreover, $W$ is almost everywhere
real analytic.
\end{proposition}
\begin{proof}
We have established in items \ref{Open} and \ref{Light} at the beginning of the subsection that $f,g,h$ are light, open (i.e. interior) mappings. Thus, according to
any of the references listed just above our proposition, there exists a homeomorphism $W$ defined on $\Omega_0$ such that $f$ factorizes as $f=F\circ W$ for an analytic
function $F$ defined on $W(\Omega_0)$, and similar for $g,h$. Since $\Omega_0$ has been shown in Lemma \ref{simply} to be simply connected and $W$ is a homeomorphism, 
$W(\Omega_0)$ must also be simply connected. If $W(\Omega_0)=\mathbb C^+$, then the proposition is proved. If not, the Riemann mapping theorem 
guarantees the existence of a conformal map $\mathcal C\colon W(\Omega_0)\to\mathbb C^+.$ Thus, the statement of the proposition follows by replacing $W$ with 
$\mathcal C\circ W$ and $F$ with $F\circ\mathcal C^{\langle-1\rangle}$. Similar statements for $g,h$ follow the same way. The a.e. real analyticity of $W$ follows easily:
for any point $z_0$ such that, say, $F'(W(z_0))\neq0$, {\em locally} we have $W=F^{\langle-1\rangle}\circ f$. The statement follows from Item 1 of Proposition \ref{propbur}
and the discreteness of the zero set of the analytic function $F'$.
\end{proof}

\subsection{An application of the Stoilow factorization}

The great advantage of $g$ over $f$ is that one can guarantee that $g$ is a homeomorphism close to $\{0\}\times\mathbb R$. We will
make this statement precise in Proposition \ref{subord0}. While unfortunately $g$ is usually not a homeomorphism on all of $\Omega$, we will be able to write a Stoilow-like
factorization for $f$ in terms of $g$ on a large enough subdomain of $\Omega$. According to 
Item 4 in Proposition \ref{propbur}, we have $|g(t,x)-x|=|tf(t,x)|<\fK\sqrt{t}/\sqrt{1-t}$, so that $\lim_{n\to\infty}g(t_n,x_n)=x$ whenever $\Omega\ni(t_n,x_n)\to(0,x)$
as $n\to\infty$. Thus,

\begin{lemma}\label{segment}
If $\Omega_0$ is a connected component of $\Omega$, then either 
$\partial\Omega_0\cap(\{0\}\times\mathbb R)$ is one point, and 
then $\nu$ has an isolated atom at that point, or there exist 
$-\infty<a=\min\partial\Omega_0\cap(\{0\}\times\mathbb R)<\max\partial\Omega_0\cap(\{0\}\times\mathbb R)=b<+\infty$ 
such that $\mathbb R\supseteq g(\partial\Omega_0\cap([0,1)\times\mathbb R))\supseteq[a,b]$.
\end{lemma}

\begin{proof}
The first statement of the lemma is obvious in light of Lemma \ref{l:mass}. Thus, assume that $a<b$. By the definition of the topological boundary, 
there exists a sequence $\{(t_n,x_n)\}_{n\in\mathbb N}\subset\Omega_0$ such that $t_n\to0$ and $x_n\to a$ as 
$n\to\infty$. As seen just above the statement of the lemma, it follows that $|g(t_n,x_n)-a|<|x_n-a|+\fK\sqrt{t_n}/\sqrt{1-t_n}$, 
and so $\lim_{n\to\infty}g(t_n,x_n)=a$. A similar statement holds for $b$ and a sequence $(t_n',x_n')$. For any $n\in\mathbb N$, we 
may draw a path $q_n$ starting at $(t_n,x_n)$ and continuing left along $\{t_n\}\times\mathbb R$ until it hits $\partial\Omega_0$, 
and  a path $q_n'$ starting at $(t'_n,x'_n)$ and continuing right along $\{t'_n\}\times\mathbb R$ until it hits $\partial\Omega_0$. 
We may consider a simple path $p_n$ starting at  $(t_n,x_n)$ and ending at $(t_n',x_n')$, and 
completely included in $\Omega_0\setminus (q_n\sqcup q_n')$. Clearly $g(q_n\sqcup p_n\sqcup q_n')$ 
(the image via $g$ of the concatenation of the three paths) is included in $\mathbb C^-$, except for 
the images of the beginning  and of the end of $q_n\sqcup p_n\sqcup q_n'$, which are mapped inside 
$\mathbb R$. The left endpoint (beginning) of this path is, by construction, at some point 
$(\xi_n,t_n)$ for a $\xi_n<x_n$; a similar statement - with the obvious modifications - holds for 
the right endpoint (end) of $q_n\sqcup p_n\sqcup q_n'$. It follows from the above that $g$ maps the 
beginning of this path into a real number which is no larger than $a+|x_n-a|+\fK\sqrt{t_n}/\sqrt{1-t_n}$ 
and the end point into a real number which is no smaller than $b-|x_n'-b|-\fK\sqrt{t_n'}/\sqrt{1-t'_n}$. Recalling 
that $f(\partial\Omega\cap((0,1)\times\mathbb R))\subseteq\mathbb R$ and that $g(0,x)=x$ for any 
$(0,x)\in\partial\Omega$, we obtain that the segment $\left[a+|x_n-a|+\fK\sqrt{t_n}/\sqrt{1-t_n},b-|x_n'-b|-\fK\sqrt{t_n'}/\sqrt{1-t'_n}\right]$ 
is included in $g(\partial\Omega_0\cap([0,1)\times\mathbb R))$ for all $n\in\mathbb N$. 
By letting $n \rightarrow\infty$, we conclude that $g(\partial\Omega_0\cap([0,1)\times\mathbb R))\supseteq[a,b]$. 
\end{proof}

The next proposition addresses the second case in Lemma \ref{segment}.

\begin{proposition}\label{subord0}
Consider a connected component $\Omega_0$ of $\Omega$ and points $a<b$
as in Lemma \ref{segment}. Then there exist a domain ${\bf K}\subset\mathbb C^-$ such that 
\begin{itemize}
\item $[a,b]\subset\partial{\bf K};$
\item for any $w\in(a,b)$ and $0<\varepsilon<\mathrm{dist}(w,\{a,b\})/2$ there exists $v>0$ such that
$\{x-{\rm i}y\colon w-\varepsilon<x<w+\varepsilon,0<y<v\}\subset{\bf K}$;
\item ${\bf K}\subset g(\Omega_0);$
\end{itemize}
and an analytic function $\Phi\colon{\bf K}\to\mathbb C^+$
such that $\Phi\circ g=f$. $g$ maps a simply connected open set $O
\subset\Omega_0$ satisfying the conditions that $\bar\Omega_0\cap\bar{O}$ contains the connected 
component of $\partial\Omega_0\setminus\{(0,a),(0,b)\}$ containing points from the set $\{(t,x)\in
\partial\Omega_0\colon a\leq x\leq b,t=\min\{s\in[0,1)\colon(s,x)\in\bar\Omega_0\}\}$ and $\bar{O}
\cap(\{1\}\times\mathbb R)=\varnothing$, bijectively onto ${\bf K}$.
\end{proposition}

\begin{proof}
While we believe there should be a direct argument guaranteeing the injectivity of $g$ close to the ``lower'' part of the boundary, we are not aware of it, and will show this indirectly. 

We introduce the following notations, besides the ones introduced in Lemma \ref{segment}: 
\begin{itemize}
\item $-\infty<a'=\min\partial\Omega_0\cap(\{1\}\times\mathbb R)
\leq\max\partial\Omega_0\cap(\{1\}\times\mathbb R)=b'<+\infty$. 
\item $\partial_1\Omega_0$ is the part of $\partial\Omega_0$ between $(0,a)$ and $(0,b),$ and away
from $\{1\}\times\mathbb R$ (the ``lower part'' of $\partial\Omega_0$): $\partial_1\Omega_0$ is the closure of the connected component of $\partial\Omega_0
\setminus\{(0,a),(0,b)\}$ containing points from
$$
\{(t,x)\in\partial\Omega_0\colon a\leq x\leq b,t=\min\{s\in[0,1)\colon
(s,x)\in\bar\Omega_0\}\};
$$
\item $\partial_3\Omega_0$ is the ``upper part'' analogue of $\partial_1\Omega_0$ for $a',b'$: the closure of the connected 
component of $\partial\Omega_0\setminus\{(1,a'),(1,b')\}$ containing points from
$$
\{(t,x)\in\partial\Omega_0\colon a'\leq x\leq b',t=\max\{s\in(0,1]\colon(s,x)\in\bar\Omega_0\}\};
$$
\item Finally, $\partial_2\Omega_0,\partial_4\Omega_0$ are the ``left'' and ``right'' parts of $\partial\Omega_0$, that is, the closures of the two connected 
components of $\partial\Omega_0\setminus(\partial_1\Omega_0\cup\partial_3\Omega_0)$: 
$\partial_2\Omega_0$ contains points from
$$
\{(t,x)\in\partial\Omega_0\colon t\in[0,1],x=\min\{r\in\mathbb R\colon
(t,r)\in\bar\Omega_0\}\},
$$
and $\partial_4\Omega_0$ is defined the same way, but with max replacing min. 
\end{itemize}
As a consequence of Lemma \ref{segment} (or, rather, its proof), $g$ maps $\partial_1\Omega_0$ onto 
$[a,b]$, and $h$ maps $\partial_3\Omega_0$ onto $[a',b']$ (we did not exclude here the possibility that
$[a',b']$ reduces to a point). All of $f,g,h$ map $\partial_2\Omega_0
\cup\partial_4\Omega_0$ in $\mathbb R$. We do not exclude the possibility that $\partial_3\Omega_0$ is reduced to a point.

As seen in Proposition \ref{Stoilow}, there exists a homeomorphism $W\colon\Omega_0\to\mathbb C^+$ and analytic functions $F,H\colon\mathbb C^+\to\mathbb C^+,$ $G\colon 
\mathbb C^+\to\mathbb C^-$ such that $f=F\circ W,g=G\circ W,h=H\circ W$. 
Since neither of $f,g,h$ are constant, $F,-G,H$ must be non-constant self-maps of $\mathbb C^+$. $W$ being a homeomorphism,
it must send boundary to boundary, i.e. $\partial\Omega_0$ into 
$\mathbb R\cup\{\infty\}$. As seen in Lemma \ref{segment}, $g$ sends $\partial_1\Omega_0$ onto the 
segment $[a,b]$. The homeomorphism $W$ sends this same $\partial_1\Omega_0$ onto a connected 
subset of $\mathbb R\cup\{\infty\}$, that is, either an interval (possibly unbounded), or the complement 
of a bounded, open interval. As the piece $\partial_1\Omega_0$ of $\partial\Omega_0$ under consideration 
is a strict subset of $\partial\Omega_0$ (with a complement containing at least an open arc in $(0,1]\times
\mathbb R$ - see Remark \ref{curve}), $W$ cannot map it onto all of $\mathbb R\cup\{\infty\}$ (indeed, 
if that were the case, the Stoilow factorization $g=G\circ W$ would provide an analytic
function $G\colon\mathbb C^+\to\mathbb C^-$ that sends all of $\mathbb R\cup\{\infty\}$ onto 
$[a,b]$, which is absurd\footnote{To be sure, a homeomorphism can map a large boundary set to
a point: for instance, $L(x,y)=(x,(1-x)y)$ sends the square $(0,1)^2$ homeomorphically onto the 
triangle $\{(u,v)\colon0<u<1,0<v<1-u\}$, squeezes the edge $\{1\}\times[0,1]$ to the point 
$\{(1,0)\}$, and $L^{\langle-1\rangle}(u,v)=(u,\frac{v}{1-u})$. It just happens that, for the 
reasons discussed above, this isn't the case for $W$ and $\{0\}\times[a,b]$ - see \cite[Page 92]{CL}.}). 
Since it is by necessity closed, this image has a complement that is either an 
open, nonempty interval, or the union of two unbounded open intervals. By pre-composing $W$ with a 
map of the type $z\mapsto\frac{1}{d-z}$ for a $d\in\mathbb R$ outside this range, we may assume 
without loss of generality that $W$ maps $\partial_1\Omega_0$ onto a compact interval which we 
denote by $[\alpha,\beta],$ and $W^{-1}(\infty)\in\partial_3\Omega_0$. By further composing with translations and 
dilations, we may, and will, assume that $\alpha=a,\beta=b$. The same argument, with $g$ 
replaced by $h$, allows us to conclude that $W$ maps $\partial_3\Omega_0$ onto a closed connected strict 
subset of $\mathbb R\cup\{\infty\}$. 

As both $g$ and $W$ map $\partial_1\Omega_0$ onto $[a,b]$
(note that $W$ must send $a$ to $b$ and $b$ to $a$),
the Stoilow factorization $g=G\circ W$ guarantees that $G([a,b])=[a,b]$. 
A priori this must be understood in the sense of limit points at the boundary. However, we claim 
that $G$ must reflect analytically through the interval $[a,b]$, 
which it maps bijectively onto itself. Indeed, since $G$ maps 
$\mathbb C^+$ into $\mathbb C^-$, it has a Nevanlinna representation 
\begin{equation}\label{Nev}
G(z)=p-qz+\int_\mathbb R\frac{1+sz}{z-s}\,{\rm d}\rho(s),\quad z\in\mathbb C\setminus\supp\rho,
\end{equation}
for some $p\in\mathbb R,q\in[0,+\infty)$, and positive finite Borel measure $\rho$ on $\mathbb R$.
If $(a,b)\cap\supp\rho=\varnothing$, then $G$ is analytic, and takes real values, 
on $(a,b)$. In particular, if $\{(t_n,x_n)\}_n\subset\Omega_0$ is such that 
$(t_n,x_n)\to\partial\Omega_0$ and $W(t_n,x_n)\to\gamma\in(a,b)$, then $g(t_n,x_n)
=G(W(t_n,x_n))\to G(\gamma)\in(a,b)$. The function $G$ is known (and easily seen) to be strictly 
decreasing on intervals in the complement of $\supp\rho$. Thus necessarily $G([a,b])=
[a,b]$, with $b=\lim_{x\downarrow a}G(x)$. Now assume towards contradiction that $(a,
b)\cap\supp\rho\neq\varnothing$. Then there exists at least one point $\gamma\in(a,
b)$ where the nontangential limit of $G$ exists, and the nontangential limit of $\Im G$ 
belongs to $[-\infty, 0)$. As $W$ is a homeomorphism, its functional inverse $W^{\langle-1\rangle}$ is
a well-defined continuous bijective map from $\mathbb C^+$ onto $\Omega_0$. In particular,
$W^{\langle-1\rangle}(\gamma+{\rm i}(0,1])\subset\Omega_0$ is a simple path that approaches $\partial\Omega_0$.
For any sequence $\{y_n\}_{n\in\mathbb N},y_n\searrow0$, such that 
$\{(t_n,x_n)=W^{\langle-1\rangle}(\gamma+{\rm i}y_n)\}_n\subset\Omega_0$ converges (necessarily to
a point in the boundary of $\Omega_0$), we have
$$
\lim_{n\to\infty}\Im g(t_n,x_n)=\lim_{n\to\infty}\Im G(W(t_n,x_n))
=\lim_{n\to\infty}\Im G(W(W^{\langle-1\rangle}(\gamma+{\rm i}y_n)))
$$
$$
=\lim_{n\to\infty}\Im G(\gamma+{\rm i}y_n)=\lim_{\stackrel{z\to\gamma}{\sphericalangle}}\Im G(z)\in[-\infty,0).
$$
As seen just before Lemma \ref{segment}, this forces $t_n\to1$. However, points in $(a,b)$ 
are necessarily limits of sequences $W(t_n,x_n)$ with $t_n$ converging to a number in 
$[0,1-\varepsilon]$ for some $\varepsilon>0$ (see also the proof of Lemma \ref{segment}). 
This is a contradiction. Thus, $(a,b)\cap\supp\rho=\varnothing$, as claimed.

We may now write the Nevanlinna representation as
$$
G(z)=p-qz+\int_{(-\infty,a)\cup(b,\infty)}\frac{1+sz}{z-s}\,{\rm d}\rho(s),\quad 
z\in\mathbb C\setminus\supp\rho;
$$
(normally one would integrate on the closed intervals, but $G([a,b])=[a,b]\subset\mathbb R$ 
implies $\rho(\{a\})=\rho(\{b\})=0$). Its derivative is 
$$
G'(z)=-q-\int_{(-\infty,a)\cup(b,\infty)}\frac{1+s^2}{(z-s)^2}\,{\rm d}\rho(s).
$$
Thus, we have $G'(x)<0$ for any $x\in\mathbb R\setminus\supp\rho$, and in particular on $(a,b)$.
We shall next find a convenient domain in $\mathbb C$ containing $(a,b)$ on 
which we can guarantee that $-\Re G'$ is greater than zero. If $z=x+{\rm i}y$, then 
$\frac{1}{(z-s)^2}=\frac{1}{(x-s+{\rm i}y)^2}=\frac{(x-s-{\rm i}y)^2}{((x-s)^2+y^2)^2}=
\frac{(x-s)^2-y^2-2{\rm i}y(x-s)}{((x-s)^2+y^2)^2}$. Then, by recalling that $q\geq 0$, 
\begin{equation}\label{lb}
-\Re G'(x+{\rm i}y)\geq\int_{(-\infty,a)\cup(b,\infty)}
\frac{(x-s)^2-y^2}{((x-s)^2+y^2)^2}(1+s^2)\,{\rm d}\rho(s).
\end{equation}
For any $s\in(-\infty,a)$, if $|y|\leq  x-a<x-s\implies
\int_{(-\infty,a)}\frac{(x-s)^2-y^2}{((x-s)^2+y^2)^2}(1+s^2)\,{\rm d}\rho(s)\geq0$,
and for any $s\in(b,+\infty)$, if $|y|\leq b-x<s-x\implies
\int_{(b,+\infty)}\frac{(x-s)^2-y^2}{((x-s)^2+y^2)^2}(1+s^2)\,{\rm d}\rho(s)\geq 0$,
with at least one of the two integrals being strictly positive. Thus, $-\Re G'(x+{\rm i}y)>0$
on ${\bf D}=\{x+{\rm i}y\colon a\leq  x\leq b,|y|\leq\min\{x-a,b-x\}\}$, a square
with diagonal $[a,b]$.
If there exists some $0<\eta<+\infty$ such that $\rho([a-\eta,a])=0$ and/or
$\rho([b,b+\eta])=0$ then we may increase the size of ${\bf D}$ accordingly
(this is relevant particularly when $a=b$ --- see note after the proof).

Since $-\Re G'$ is strictly positive on the convex set ${\bf D}$, it 
follows that $G$ is injective on ${\bf D}$ and $G(\partial{\bf D})$ is a simple closed curve in 
$\mathbb C$, symmetric with respect to $\mathbb R$. We have $G(a)=b,G(b)=a$. Note that 
$0>\Re G'(x)=\lim_{y\to0}\frac{\Im G(x+{\rm i}y)-\Im G(x)}{y}=\lim_{y\to0}\frac{\Im G(x+{\rm i}y)}{y}$,
$x\in[a,b]$, so that $\Re G'(x)$ being bounded away from zero provides a lower bound for the vertical
thickness of $G({\bf D})$ at any given $x\in(a,b)$. More specifically, $G$ is conformal on 
${\bf D}$ so that, by Koebe's distortion theorem (see, for instance, Kari Astala, Tadeusz Iwaniec, and Gaven Martin \cite[Theorem 2.10.6]{MR2472875}), we have 
\begin{equation}\label{Koebe}
\frac{|G'(x)|}{4}\text{dist}(x,\partial{\bf D})\leq\text{dist}(G(x),\partial G({\bf D}))\leq|G'(x)|\text{dist}(x,\partial{\bf D}).
\end{equation}
The shape of our domain ${\bf D}$ guarantees that $\text{dist}(x,\partial{\bf D})=\min\{x-a,b-x\}/\sqrt{2}$,
thus allowing us to conclude with ${\bf K}=G({\bf D}\cap\mathbb C^+)\subset\mathbb C^-$.

Denote ${\bf D}^\pm={\bf D}\cap\mathbb C^\pm$. The relation $g=G\circ W$ and 
the fact that $W\colon\Omega_0\to\mathbb C^+$ is a homeomorphism provide us 
with a set $O=W^{-1}({\bf D}^+)$ as claimed in our proposition.

Recall that $g=G\circ W,f=F\circ W$. Trying to find a map $\Phi$ such that $\Phi\circ g=f$
is equivalent to finding $\Phi$ such that $\Phi\circ G\circ W=F\circ W$ on some subset of $\Omega_0$.
Since $W$ is a homeomorphism, it is enough to find $\Phi$ such that $\Phi\circ G=F$ on
some relevant domain inside $\mathbb C^+$. We simply define $\Phi\colon G({\bf D}^+)\to
\mathbb C^+$, $\Phi(z)=F( G^{\langle-1\rangle}(z))$, where the inverse is the one taking values in
${\bf D}$. This completes the proof.
\end{proof}

Proposition \ref{subord0} holds true as well when $a=b$ (i.e. under the first case of Lemma \ref{segment}), with the set $O$ containing in its boundary
parts of $\partial_2\Omega_0$ and $\partial_4\Omega_0$ adjacent to the point $a=b$: this is a direct application of the description
of $\partial\Omega_0$ provided in the arguments justifying Remark \ref{curve} together with the Schur complement formula applied in Proposition \ref{propbur}, Item 6. 
This case is not relevant for the purposes of Theorem \ref{main}, but nevertheless of some interest. However, it will be important later to notice that
Koebe's Theorem \eqref{Koebe} together with the estimate \eqref{lb} imply that the simple curve $\partial G(\mathbf D)$ creates angles of non-zero
measure (possibly more than $\pi$) at both endpoints $a$ and $b$. Moreover, if one of $a$ or $b$ is an atom for $\nu$, then
that point does not belong to $\supp\rho$, and so it belongs to the interior in $\mathbb R$ of $\mathbb R\cap\partial\mathbf K$.

\subsection{The boundary values of $f$}

We recall that, unlike quasiregular functions, ``most'' analytic functions are determined by their values at 
the frontier. In particular, a function defined on a domain whose boundary contains an interval from $\mathbb R$ 
and with values in a half-plane is determined by its (known to exist 
a.e.) nontangential limits on any set of nonzero measure (the Fatou and Riesz-Privalov Theorems - see 
\cite[Theorems 2.5 and 8.1]{CL}).  In this section, we prove Theorem \ref{main}, namely that $f(0,x)=\lim_{t\to0}\Phi(g(t,x))=
\displaystyle\lim_{\stackrel{z\to x}{\sphericalangle}}\Phi(z)$, where, as before, $\displaystyle
\lim_{\stackrel{x\to x}{\sphericalangle}}\Phi(z)$ denotes the nontangential limit of $\Phi$ at $x$. This 
holds of course $\nu^{ac}$-almost everywhere on the interval $[a,b]$ in question. In order for our result to be 
non-vacuous, we need to assume that $\nu$ has a nonzero absolutely continuous part in $[a,b]$, and,
in particular, that $a<b$.

\begin{lemma}\label{limit}
Let ${D}\subseteq\mathbb C^+$ be a rectangle such that $\partial{D}\cap\mathbb R$ is an 
interval whose interior contains zero, and consider a non-constant analytic function $\omega\colon{
D}\to\mathbb C^+$. Suppose that the nontangential limit of $\omega$ at zero exists and
belongs to $\mathbb C^+.$ Then there exists $1>\varepsilon>0$ and a
smooth path $\gamma\colon(0,\varepsilon]\to \mathbb C^+$ such that:
\begin{enumerate}
\item $\gamma(t)=t\omega(\gamma(t)),$ $t\in(0,\varepsilon];$
\item $\lim_{t\to0}\gamma(t)=0$ and $\lim_{t\to0}\gamma(t)/t\displaystyle=\lim_{\stackrel{z\to 0}{\sphericalangle}}\omega(z);$
\item $\lim_{t\to0}\omega(\gamma(t))$ exists and equals the nontangential
limit of $\omega$ at zero.
\end{enumerate}
Moreover, for $\varepsilon$ small enough, the path $\gamma$
satisfying properties 1,2, and 3 above is unique.
\end{lemma}

\begin{proof}
Let $l\in\mathbb C^+$ denote the nontangential limit of $\omega$ at zero. Consider a cone
$$
\Gamma_{c}=\{z\in\mathbb C^+\colon|\Re z|<c\Im z\},
$$
choose $c=1+2|\Re l|/\Im l$ (so that $l\in\Gamma_c$), and denote 
$\Gamma_c(\eta)=\{z\in\Gamma_c\colon\Im z<\eta\}$. Since zero 
belongs to the {\it interior} of the interval which is the intersection of 
the boundary of ${D}$ with the real line, there exists an $\eta>0$ 
sufficiently small such that $\Gamma_c(\eta)\subset{D}$.

By definition of nontangential limit, there exists $\eta>0$ such
that $\omega(\Gamma_c(\eta))\subset\Gamma_c,$ and from the continuity
of $\omega$ on the closure of $\Gamma_c(\eta)$ we conclude that the set
$\omega(\Gamma_c(\eta))$ is bounded. Thus, there exists $\varepsilon>0$
such that $t\omega(\Gamma_c(\eta))\subsetneq\Gamma_c(\eta)$ for all
$t\in(0,\varepsilon].$ Fix now such a $t$. The analytic function
$\varphi_t\colon\Gamma_c(\eta)\to\Gamma_c(\eta)$ defined by
$\varphi_t(z)=t\omega(z)$ has, according to the Denjoy-Wolff Theorem
\cite{MR9197,wolff}, a unique interior fixed point, which is also attracting (observe that the point must 
indeed be interior, since zero is not a fixed point, and $t\omega(\Gamma_c(\eta))$ 
is a proper subset of $\Gamma_c(\eta)\cup\{0\}$). Denote this point by $\gamma(t).$ 
The implicit function theorem guarantees the smoothness of the
correspondence $t\mapsto\gamma(t):$ indeed, according to the Schwarz-Pick
lemma, $|\varphi_t'(\gamma(t))|<1$. This proves the first part of the lemma.

The first part of item 2 follows from the facts that $\gamma(t)\in t\omega(\Gamma_c(\eta))$
and the set $t\omega(\Gamma_c(\eta))$ tends to zero uniformly as $t\to0.$

Item 3 follows from the first part of item 2, the fact that $\gamma(t)\!\in\!\Gamma_c$ for all
$t\!\in\!(0,\varepsilon]$, and Fatou's Theorem.

The second part of item 2 follows directly from item 3 and the equality $\gamma(t)/t=\omega(\gamma(t))$.

Assume towards contradiction that there exists another path
$\delta\colon(0,\varepsilon]\to{D}$ satisfying conditions 1-3 in the
Lemma. Observe that both $\gamma$ and $\delta$ are right inverses
for the function $\Psi\colon{D}\to\mathbb C\setminus(-\infty,0]$
defined by $\Psi(z)=z/\omega(z),$ and thus they are injective.
Moreover,
$\gamma((0,\varepsilon])\cap\delta((0,\varepsilon])=\varnothing.$
Indeed, assume that $\gamma(t_1)=\delta(t_2).$ Then
$t_1=\Psi(\gamma(t_1))=\Psi(\delta(t_2))=t_2$. Denote $s=t_1=t_2.$
We have
$$
\Psi'(\gamma(s))=s^2\frac{\omega(\gamma(s))-\gamma(s)\omega'(\gamma(s))}{\gamma(s)^2}=
s^2\frac{\gamma(s)(1/s-\omega'(\gamma(s)))}{\gamma(s)^2},
$$
so the derivative of $\Psi$ in the point $\gamma(s)$ is zero if and only
if $1-s\omega'(\gamma(s))=0,$ or, equivalently, if
$\varphi_s'(\gamma(s))=1.$ But $\gamma(s)$ is the Denjoy-Wolff
point of $\varphi_s(z)=s\omega(z).$ Since, as seen above, $z\mapsto
\varphi_s(z)$ sends $\Gamma_c(\eta)$ strictly inside itself, we obtain a
contradiction with the Schwarz-Pick Lemma. We conclude
that the derivative of $\Psi$ in the point $\gamma(s)$ cannot be
zero, so that $\Psi$ must be injective on some neighborhood of
$\gamma(s),$ and hence $\gamma$ and $\delta$ must coincide on a
whole subinterval of $(0,\varepsilon]$ centered at $s$, and hence on
all $(0,s].$ This is a contradiction. So indeed
$\gamma((0,\varepsilon])\cap\delta((0,\varepsilon])=\varnothing.$

Now consider the open, connected and simply connected set
$\bD_0\subset{D}$ delimited by $\gamma,\delta,$ and a third
simple smooth curve $\beta$ included in ${D}$ which has its
endpoints at $\gamma(\varepsilon/2)$ and $\delta(\varepsilon/2),$
intersects $\gamma((0,\varepsilon])\cup\delta((0,\varepsilon])$ in no
other point, and such that zero belongs to the closure of
$\bD_0$. Observe that $\Psi(\gamma(t))=\Psi(\delta(t))=t,$ so
by a theorem of Lindel\"{o}f \cite[Theorem 2.3.1]{CL} applied to $\Psi$ on $D$, and a corollary of the 
Iversen Theorem \cite[Theorem 5.2]{CL} applied to $\Psi$ on $\bD_0$, we have $\lim_{z\to0,z\in\bD_0}\Psi(z)=0.$
Of course, $\Psi(\bD_0)$ is in its own turn an open connected
set. Since $\Psi(\gamma((0,\varepsilon/2]))=\Psi(\delta((0,\varepsilon/2]))=(0,\varepsilon/2]$,
and so, as seen above, $\lim_{z\to0,z\in\bD_0}\Psi(z)=0$,
we must have that $0\in\overline{\Psi(\bD_0)}.$ 
Since $\Psi$ is an analytic map, hence open, $\partial\Psi(\mathbb D_0)\subseteq\Psi(\partial\mathbb D_0)=\Psi(\delta\sqcup\gamma\sqcup\beta)$. $\Psi(\mathbb D_0)$ is a bounded 
open connected set, so its topological boundary is a compact set in $\mathbb C$. This compact set must thus be included in the continuous curve $[0,\varepsilon/2]\cup\Psi(\beta)$. 
Thus, $\Psi(\beta)$ must describe a curve in $\mathbb C$ which, together with $[0,\varepsilon/2]$, surrounds an open connected set (it may enter it too, but must surround it entirely). 
By construction, $\Psi$ maps the two ends of $\beta$ in $\varepsilon/2$, so $\Psi(\beta)$ is a closed curve in $\mathbb C\setminus(-\infty,0]$. But this closed curve cannot be included 
in $\mathbb C\setminus(-\infty,0]$ because $\Psi$ is an open mapping on $D$ and $\beta\subset{D}$. This is a contradiction. So indeed $\gamma$ is unique.
\end{proof}
As an aside, we record three free probability consequences of the above lemma, which have no direct application to this work.
\begin{corollary}
Assume that $\mu$ is a Borel probability measure on $\mathbb R$, $\gamma_t$ is the centered semicircular of variance $t$, and $\nu$ is a freely infinitely divisible probability
measure on $\mathbb R$ which is not a point mass.
\begin{enumerate}
\item For $\mu^{\rm ac}$-almost every $x\in\mathbb R$, we have $\lim_{t\to0}\frac{{\rm d}(\mu\boxplus\gamma_t)(x)}{{\rm d}x}=\frac{{\rm d}\mu(x)}{{\rm d}x}$;
\item For $\mu^{\rm ac}$-almost every $x\in\mathbb R$, we have $\lim_{t\to1}\frac{{\rm d}\mu^{\boxplus t}(x)}{{\rm d}x}=\frac{{\rm d}\mu(x)}{{\rm d}x}$;
\item For $\mu^{\rm ac}$-almost every $x\in\mathbb R$, we have $\lim_{t\to0}\frac{{\rm d}(\mu\boxplus\nu^{\boxplus t})(x)}{{\rm d}x}=\frac{{\rm d}\mu(x)}{{\rm d}x}$;
\end{enumerate}
\end{corollary}
\begin{proof}
Each item follows from the functional equation satisfied by the Cauchy-Stieltjes transform of the corresponding measure, which we record below:
\begin{eqnarray*}
G_{\mu\boxplus\gamma_t}(z) & = & G_\mu(z-tG_{\mu\boxplus\gamma_t}(z)),\\
G_{\mu^{\boxplus t}}(z) & = & G_\mu\left(\frac1tz+\left(1-\frac1t\right)\left(G_{\mu^{\boxplus t}}(z)\right)^{-1}\right),\\
G_{\mu\boxplus\nu^{\boxplus t}}(z) & = & G_\mu\left(z-tR_\nu(G_{\mu\boxplus\nu^{\boxplus t}}(z))\right).
\end{eqnarray*}
As before, $R_\nu$ denotes Voiculescu's $R$-transform of $\nu$, which is known to extend to the upper/lower complex half-plane if and only if $\nu$ is freely infinitely divisible
(see \cite[Theorem 5.10(i)]{BV93}). The reader will find details on the proofs of these relations in \cite{BV93,Biane} as well as in \cite[Theorem 2.5]{belinschi2004atoms}.
\end{proof}
Yet another free probability consequence of Lemma \ref{limit} is recorded in \cite[Corollary 5.2]{belinschi2008remarkable}.

\begin{proof}[Proof of Theorem \ref{main}]
Pick an arbitrary connected component $\Omega_0$ of $\Omega$. With the notations from Proposition \ref{subord0},
consider a point $x\in\mathbb R$ such that $(0,x)\in\overline{\Omega_0}\setminus\{(0,a),(0,b)\}$ 
and $\displaystyle\lim_{\stackrel{z\to x}{\sphericalangle}}\Phi(z)\in\mathbb C^+.$ We let 
$\omega(z)=-\Phi(x+z)$. By Proposition \ref{subord0}, this function is defined on a small 
rectangle included in the lower half-plane and having zero at the middle of its upper edge.
According to Lemma \ref{limit}, there exists a unique path $\gamma$ in this 
rectangle such that $\gamma(t)=t\omega(\gamma(t))=-t\Phi(x+\gamma(t))$.
However, at the same time $-tf(t,x)=-t\Phi(x-tf(t,x))$. The uniqueness part of 
Lemma \ref{limit} guarantees that $\gamma(t)=-tf(t,x)$, and items 1--3
of the same lemma allow us to conclude.
\end{proof}

Now we can finally state that 
$$
f(0,x)=\lim_{t\to0}\Phi(g(t,x))=\displaystyle\lim_{\stackrel{z\to x}{\sphericalangle}}\Phi(z).
$$
This relation holds for $\nu^{\rm ac}$-almost all $x\in\mathbb R$. 
That is, for each connected component of $\Omega
$, we find a set $O$ on which the above can be 
written. From this relation and Item 6 of Proposition \ref{propbur}, it follows
immediately that
$$
z\mapsto\Phi(z)-\int_\mathbb R\frac{{\rm d}\nu(s)}{z-s}
$$
is an analytic function on the intersection of the domain of $\Phi$ with $\mathbb C^-$. 
Its nontangential limits are real a.e. on $\mathbb R\cap\partial{\bf K}$. Thus, as seen 
before, since $\Phi(z)-\int_\mathbb R1/(z-s)\,{\rm d}\nu(s)$ has real nontangential limits 
a.e. on the relevant subset of $\mathbb R$, either the Schwarz reflection principle 
applies at a given point $x$, or the cluster set of the function at $x$ is equal to 
$\overline{\mathbb C^+}$ or to $\overline{\mathbb C}$ (see \cite[Theorem 5.4]{CL}).

We focus next on the issue of proving Corollary \ref{c:uniquemu} (the uniqueness of the Brownian bridge
when provided with the {\em complete} initial data, i.e. with $f(0,\cdot)$). This will be done in the next section. For now, we 
record here for future use yet another boundary property of $\Phi$.
Specifically, the following lemma clarifies the relation between $\Phi$ and $f(0,\cdot)$ if $\nu$
contains a non-singular part.
\begin{lemma}\label{dt}
With the notations from Proposition \ref{subord0}, assume that $\supp\nu^{ac}\cap\overline{\Omega_0}\neq\varnothing$.
Then $f(0,\cdot)$ determines uniquely $\Phi$ on $g(\Omega_0)$.
\end{lemma}
\begin{proof}
We continue to use the notations from Proposition \ref{subord0} and its proof.
It has been noted just after the proof of Theorem \ref{main} that if $\supp\nu^{ac}\cap\partial\Omega_0\neq\varnothing,$
then $\displaystyle\lim_{\stackrel{z\to x}{\sphericalangle}}\Phi(z)=f(0,x)$ for $\nu^{ac}$-almost 
all $x$, that is for all $x$ in a set of nonzero Lebesgue measure. An application of \cite[Theorem 8.1]{CL}
guarantees that $\Phi$ is determined by these values.
\end{proof}

\subsection{Uniqueness of the Brownian bridge under complete initial data}

We start this section with the proof of Corollary \ref{c:uniquemu} under the extra assumption that $\nu^{\rm ac}$ puts strictly positive mass on each connected 
component of $\supp\nu$. This hypothesis will be weakened later.

\begin{proof}[Proof of Corollary \ref{c:uniquemu}, restricted case]
For simplicity of notations we write $f(t,x)=f^{\nu\rightarrow \mu}(t,x)$ and $f'(t,x)=f^{\nu\rightarrow \mu'}(t,x)$, and let
\begin{align*}
g(t,x)=x-tf(t,x),\quad g'(t,x)=x-tf'(t,x).
\end{align*}

We restrict ourselves to two connected subsets 
$\Omega_0\subset\Omega$ and $\Omega'_0\subset\Omega'$ such that $(\{0\}\times\mathbb R)\cap(\partial_1\Omega_0)\cap(\partial_1\Omega'_0)$ contains more than one point
(see notations in the proof of Proposition \ref{subord0}, which we shall use throughout this proof as well). According to Proposition \ref{propbur} Items 2 and 3,
in order to prove our corollary, it is enough to show that $f=f'$ (and thus, in particular, $\Omega_0=\Omega_0'$). By Proposition \ref{subord0}, there exist sets ${\bf K}$,
${\bf K}'\subset\mathbb C^-$ and maps $\Phi,\Phi'$ such that $\partial{\bf K}\cap\partial{\bf K}'\supseteq(\{0\}\times\mathbb R)\cap(\partial_1\Omega_0)\cap(\partial_1\Omega'_0)$, 
$\Phi\colon{\bf K}\to\mathbb C^+,\Phi'\colon{\bf K}'\to\mathbb C^+$ are analytic and satisfy $\Phi\circ g=f,\Phi'\circ g'=f'$.
According to our hypothesis, the restriction of $\nu^{\rm ac}$ to the connected component(s) of $\supp\nu$ included in the intersection 
$(\{0\}\times\mathbb R)\cap(\partial_1\Omega_0)\cap(\partial_1\Omega'_0)$ is non-zero, so that, by Theorem \ref{main}, there exists a subset of nonzero Lebesgue measure of points 
$c$ in this set such that $\lim_{t\to0}f(t,c),\lim_{t\to0}f'(t,c)$ exist, are equal, and $\lim_{t\to0}\Im f(t,c)=\lim_{t\to0}\Im f'(t,c)\in(0,+\infty).$
Lemma \ref{dt} guarantees that $\Phi=\Phi'$ on their (nonempty) joint domain, which means that they are extensions of each other to the respective domains.
Lemma \ref{limit} forces $f(t,c)=f'(t,c)$ for all points $c$ as above, and $t\in(0,1)$ sufficiently small. As seen in the proof of Lemma \ref{limit},
the existence and uniqueness of the solution $f(t,c)$ to the equation $\Phi\circ g=f$ is provided via Denjoy-Wolff Theorem applied to a properly chosen domain
(in order to prove the existence of a solution) and the implicit function theorem (in order to prove analyticity of the correspondence in $t$). 
This last result however provides, via the relation $\Phi'(x-tf(t,x))(1-t\partial_x f(t,x))=\partial_xf(t,x)$ and the analytic implicit function theorem, 
the analyticity of the correspondence $x\mapsto f(t,x)$ on a neighborhood (in $\mathbb C$) of a given $c\in \mathbb R$ as above, for fixed $t>0$ such that $|t\Phi'(x-tf(t,x))|<1$.
Thus, $f,f'$ have analytic extensions as functions of two complex variables to an open set in $\mathbb C^2$. Thanks to their equality
on an open subset of $\mathbb R^2$ showed above, we have $f=f'$ on this whole open set, and necessarily on all of their common 
domain of analyticity.

The existence and differentiability (established independently - see Proposition \ref{propbur}) of the solution $f$ guarantees that the set 
$\{(t,x)\in O\colon\Phi'(g(t,x))=-\frac1t\}$ is empty (recall from Proposition \ref{subord0} that $g(O)={\bf K}$). Indeed, otherwise
the equality $\Phi'(x-tf(t,x))(1-t\partial_x f(t,x))=\partial_xf(t,x)$ would imply $0=\frac1t$, an obvious contradiction. In particular, 
this shows that the implicit function theorem argument from the above applies on {\em all} of $O\cap O'$ in order to conclude that $f=f'$ on this set.

Proposition \ref{subord0} guarantees that the upper part of the boundary of the union of the two sets ${\bf K,K}'$ (i.e. $(\mathbb R\cap{\bf K})\cup(\mathbb R\cap\partial{\bf K}')$)
covers entirely the convex hull of the components of $\supp\nu$ included in the union $\partial\Omega_0\cup\partial\Omega'_0$.
That allows us to conclude that $\Phi=\Phi'$ extends to an open, simply connected set that contains this convex hull in its boundary.
Thus, the domain of analyticity of each of $f,f'$ extends to a simply connected subset $C\supseteq O\cup O'$ of $\Omega_0\cup\Omega_0'$
which contains in its boundary all of the components of $\supp\nu$ included in the union $\partial\Omega_0\cup\partial\Omega'_0$,
and on which, as seen above, $f=f'$. 
Given the definition of $\Omega=\{\Im f>0\},\Omega'=\{\Im f'>0\}$, it follows that $C\subseteq\Omega_0\cap\Omega_0'$ so, in particular,
$\Omega_0\cap(\{0\}\times\mathbb R)=\Omega_0'\cap(\{0\}\times\mathbb R).$

Consider an arbitrary $s\in(0,1)$. Define $g_s(t,x)=x-(t-s)f(t,x)=g(t,x)+sf(t,x)$, 
which is automatically a solution of the Beltrami equation \eqref{Bel} on $\Omega_0$. Moreover, $g_s(s,x)=x,\Im g_s(t,x)<0$
for $1>t>s$, $\Im g_s(t,x)>0$ for $0<t<s$. Lemma \ref{segment} applies to $g_s$ on $\{s\}\times\mathbb R$,
and thus there exists a set\footnote{It may be that $\{(t,x)\in\Omega_0\colon1>t>s\}$ and/or $\{(t,x)\in\Omega_0\colon s>t>0\}$
are not connected anymore -- examples can be easily found --, but our arguments apply as well to each connected component of these sets.} 
${\bf K}_s\subseteq\mathbb C^-$, a set $O_s\subseteq\{(t,x)\in\Omega_0\colon1>t>s\}$, and an analytic function $\Phi_s$ as
in Proposition \ref{subord0}, with the only difference that $\Phi_s(g_s(t,x))=f(t,x),(t,x)\in O_s$. 
In addition, the proof of Proposition \ref{subord0} applies without modification to show that there exist
${\bf L}_s\subseteq\mathbb C^+$,  $P_s\subseteq\{(t,x)\in\Omega_0\colon s>t>0\}$, and an analytic function $\Psi_s$
such that $\Psi_s(g_s(t,x))=f(t,x),(t,x)\in P_s$. The similar objects derived from $\Omega_0'$ and $f'$
will be denoted the same way, except that each will receive a $'$.

Returning now to the question of the equalities $f=f',\Omega_0=\Omega_0'$, consider an $(s,x)\in\Omega_0
\cap\Omega_0'$. We claim that $f(s,x)=f'(s,x).$ (Note again that trivially if $f=f'$ on $\Omega_0
\cap\Omega_0'$, then $\Omega_0=\Omega_0'$.) If $(s,x)\in C$, then there is nothing to prove. If $(s,x)\in\{s\}\times I$
for some interval $I\subseteq\mathbb R$ such that $\{s\}\times I\cap C\neq\varnothing,$ then we apply the 
considerations above to find sets ${\bf L}_s,{\bf L}_s'$, $P_s,P_s'$, and functions $\Psi_s,\Psi'_s$ such that 
$\Psi_s(g_s(t,x))=f(t,x),\Psi_s'(g_s'(t,x))=f'(t,x)$ for $(t,x)\in P_s\cap P_s'$. The arguments above
yield the existence of a connected set $C_s$ in $P_s\cup P_s'$ containing $\{s\}\times I$ in its
boundary such that $f=f'$ on all of $C_s$, and hence in $(s,x)$. We immediately observe that 
this guarantees the equality $f=f'$ on the whole subset below $\{s\}\times I$. More precisely,
we look at the segments $(\{r\}\times\mathbb R)\cap\Omega_0\cap\Omega_0'$ for each $r\in(0,1)$. 
If $\tilde{s}$ is such that $(\{\tilde{s}\}\times\mathbb R)\cap C\neq\varnothing$, then there are disjoint intervals
$I_1,I_2,\dots,I_k$ such that $\{\tilde{s}\}\times\mathbb R\cap\Omega_0\cap\Omega_0'=
\{\tilde{s}\}\times(I_1\cup\cdots\cup I_k)$ and a (possibly smaller) subfamily $\{J_1,\dots,J_l\}\subseteq
\{I_1,I_2,\dots,I_k\}$ such that $\{\tilde{s}\}\times J_i\cap C\neq\varnothing,1\le i\le k$. Then 
the set of points $(r,x)$ that can be connected to $\{\tilde{s}\}\times J_i$ for some $i\in\{1,\dots,k\}$
by a path in $\Omega_0$ starting at $(r,x)$ and whose first coordinate does not decrease satisfies 
the condition that $f(r,x)=f'(r,x)$ and implicitly that $(r,x)\in\Omega_0\cap\Omega_0'$.
Thus, we have succeeded in proving that the set of points on which $f=f'$ is bounded 
(in $(0,1)\times\mathbb R$) by a family of segments $\{s_j\}\times I_j$. Moreover, below
(in the sense just described) these segments, $\Omega_0$ and $\Omega_0'$ coincide.
For simplicity, re-denote this set by $C$.

Finally, assume that $(s,x)$ is above $C$, meaning that, with the notation $I$ from the above,
$(\{s\}\times I)\cap C=\varnothing$. For any segment $\{r\}\times J$ bordering $C$, we may perform 
the construction described in the first part of the proof in order to increase $C$ {\em strictly} above
the ``level'' $r$ (i.e. to find $r'>r$ and an interval $J'\subset\mathbb R$ such that points from
$\{r\}\times J$ can be united to points from $\{r'\}\times J'$ by smooth paths whose first coordinate
in $(0,1)\times\mathbb R$ does not decrease). Continuity of $f$ guarantees that this process will reach
any ``level'' which is strictly less than $1$, and in particular level $s$.

Recalling that the definition of $\Omega$ is given as the set in 
$(0,1)\times\mathbb R$ on which $\Im f>0$ allows us to conclude that $\Omega_0=\Omega_0'$, and thus 
complete the proof of this first version of our corollary. 
\end{proof}

To prove Corollary \ref{c:uniquemu} in full generality, let us make explicit how boundary values determine $\Phi$ in the absence of an absolutely continuous part for $\nu$.
According to Items 2 and 3 of Proposition \ref{propbur}, we have weak convergence of $\Im f(t,\cdot)$ and of $\Re f(t,\cdot)$ as $t\to0$.
Pick a connected component $\Omega_0$ of $\Omega$. If $(\{0\}\times\supp\nu^{\rm ac})\cap\partial\Omega_0\neq\varnothing,$ then the already proven restricted case
applies. Thus, we assume that $(\{0\}\times\supp\nu^{\rm ac})\cap\partial\Omega_0=\varnothing,$ and recall the hypothesis that each component of the closed set 
$(\{0\}\times\supp\nu)\cap\partial\Omega_0=(\{0\}\times\supp\nu^{\rm s})\cap\partial\Omega_0$ is infinite. In particular (with the notations from Proposition \ref{subord0} and
Lemma \ref{segment}) it is automatic that not only $a<b$, but also that implicitly $\nu^{\rm s}((a,b))=\nu((a,b))>0$, so $\supp\nu\cap(a,b)$ is infinite. Recall the constructions 
$\mathbf K$ and $\Phi=F\circ G^{\langle-1\rangle}$ from Proposition \ref{subord0} and its proof (whose notations we continue using). The set $\mathbf K\subseteq\mathbb C^-$ 
contains in its boundary the interval $[a,b]$, which contains the (infinite by hypothesis) set $(\{0\}\times\mathbb R)\cap\partial\Omega_0=(\{0\}\times\supp\nu)\cap\partial\Omega_0$. 
The relation $\Phi\circ g=f$ together with Lemma \ref{segment} forces the set $\supp\nu\cap[a,b]$ inside the set of singularities of $\Phi$ (the set where $\Phi$ does not reflect
meromorphically through $(a,b)$) -- as it is apparent from the definition of $\Omega$, the two sets in fact coincide. As $F$ is a Nevanlinna map (see proof of Proposition \ref{subord0}), 
it has a representation $F(u)=m+nu+\int_\mathbb R\frac{1+tu}{t-u}\,{\rm d}\lambda(t)$. This forces the bound $0<\Im u\Im F(u)<n(\Im u)^2+\int\frac{(\Im u)^2}{(\Re u-t)^2+(\Im 
u)^2}(1+t^2)\,{\rm d}\lambda(t)$. If $\Re u$ is restricted to a compact interval (in our case $[a,b]$) and $0<|\Im u|<M$, then this last quantity is bounded by a constant $C$ 
depending only on $a,b,M,n,$ and $\lambda$, not on $u$. The function $G$ has been shown in the proof of Proposition \ref{subord0} to be bijective on $[a,b]$ and $-G'(x)>0$ has 
been shown to be bounded away from zero, so that $G^{\langle-1\rangle}$ is differentiable on $(a,b)$, and the absolute value of the derivative of $G^{\langle-1\rangle}$ is uniformly 
bounded on $[a,b]$. Thus, there exists a constant $C>0$ such that $|\Phi(u)|<\frac{C}{|\Im u|},u\in\mathbf K$. We define an extension of $\Phi$ to the set $\{z\in\mathbb C\colon
\bar{z}\in \mathbf K\}$ by $z\mapsto\overline{\Phi(\bar{z})}$, a correspondence which is analytic. For any $\varepsilon>0$, one may define the Schwartz distribution 
\begin{equation}\label{Distr}
T^\varepsilon(\psi)
=(2{\rm i})^{-1}\lim_{y\nearrow0}\int_{[a+\varepsilon,b-\varepsilon]}\Phi(x+{\rm i}y)\psi(x)\,{\rm d}x-\int_{[a+\varepsilon,b-\varepsilon]}\Phi(x-{\rm i}y)\psi(x)\,{\rm d}x,
\end{equation}
for any compactly supported test function $\psi$ on $\mathbb R$. 
Since the rate of growth at the boundary of $\Phi$ is uniformly of order $\frac{1}{|\Im u|}$, it follows that $T^\varepsilon$ is a distribution of order zero,
i.e. a measure, for all $\varepsilon>0$. Since $(2{\rm i})^{-1}(\Phi(x+{\rm i}y)-\Phi(x-{\rm i}y))>0$ for all $y<0,x\in(a,b)$, the measure is positive. We write its Cauchy transform as 
$$
G_{T^\varepsilon}(u)=\frac{1}{\pi}\left\langle T_t^\varepsilon,\frac{1}{u-t}\right\rangle,\quad z\in\mathbb C\setminus[a+\varepsilon,b-\varepsilon],
$$
where $T_t^\varepsilon$ refers to the distribution $T^\varepsilon$ acting on functions in the variable $t$ (specifically, in our case, on the function $\mathbb R\ni t\mapsto\frac{1}{u-t}
\in\mathbb C$), and the large brackets refer to the duality action (see, for instance \cite{mitrovic1971plemelj}). The Plemelj formula implies that
$$
T^\varepsilon(\psi)=(2{\rm i})^{-1}\lim_{y\nearrow0}\int_{[a+\varepsilon,b-\varepsilon]}G_{T^\varepsilon}(x+{\rm i}y)\psi(x)\,{\rm d}x-\int_{[a+\varepsilon,b-\varepsilon]}
G_{T^\varepsilon}(x-{\rm i}y)\psi(x)\,{\rm d}x,
$$
so that the analytic function $\Phi-G_{T^\varepsilon}$ extends analytically to $\partial\mathbf K\cap\mathbb R$ with real values (see also \cite{CL}).
 Moreover,
$$
\lim_{y\to0}\int_{[a+\varepsilon,b-\varepsilon]}G_{T^\varepsilon}(x+{\rm i}y)\psi(x)\,{\rm d}x+\int_{[a+\varepsilon,b-\varepsilon]}G_{T^\varepsilon}(x-{\rm i}y)\psi(x)\,{\rm d}x
=\frac{\rm 2}{\pi}\left\langle T^\varepsilon_t*{\rm vp}\frac1t,\psi\right\rangle,
$$
where $T^\varepsilon_t*{\rm vp}\frac1t$ denotes the convolution of the distribution $T^\varepsilon$ with the principal value. This extension allows us to write the formula 
\begin{equation}\label{78}
\lim_{t\to0}f(t,x)-G_{T^\varepsilon}(g(t,x))=\lim_{t\to0}\Phi(g(t,x))-G_{T^\varepsilon}(g(t,x))=\lim_{z\to x}\Phi(z)-G_{T^\varepsilon}(z).
\end{equation}
(Recall that $\Im g(t,x)\le0$.) That is, the difference between $f$ and $G_{T^\varepsilon}\circ g$ extends continuously to $t=0$ as a function, not as a distribution
(the above equalities immediately imply that $\lim_{t\to0}\Im f(t,\cdot)=\lim_{t\to0}\Im G_{T^\varepsilon}(g(t,\cdot))$ in the sense of distributions).

Since this holds for all small $\varepsilon>0$ and we have assumed that each connected component of $\supp\nu$ is infinite, it follows via the identity principle for analytic functions
that for some $\varepsilon >0$ small enough, the analytic function $\Phi(z)-G_{T^\varepsilon}(z)$, $z\in G(\mathbf D)$, is uniquely determined by its values on $\supp\nu$.
(When $z$ is real, $\Phi(z)-G_{T^\varepsilon}(z)$ is understood in the sense of analytic continuation.) 


{\bf Aside:} The reader may be legitimately concerned by the possibility that $T^\varepsilon$ is not well-defined as a Schwartz distribution, as $\Phi$ is {\em not} defined on all of 
$\mathbb C^-$. One may easily argue that this possibility does not occur by considering the restriction of $\Phi$ to a rectangle in $\mathbb C^-$ with upper edge $(a+\varepsilon,b-
\varepsilon)$, mapping this rectangle conformally onto $\mathbb C^-$ and then arguing that the distribution $T^\varepsilon$ is the push-forward of part of the Nevanlinna 
representation measure corresponding to the composition of $\Phi$ with the above-mentioned conformal mapping (or, more precisely, its inverse). However, thanks to the expression 
$\Phi=F\circ G^{\langle-1\rangle}$, we may be more precise as to the identity of $T^\varepsilon$. Indeed, 
$$
\int_{[a+\varepsilon,b-\varepsilon]}\!\!\!\!\Phi(x+{\rm i}y)\psi(x)\,{\rm d}x=\!\int_{[a+\varepsilon,b-\varepsilon]}\!\!\!\!F(G^{\langle-1\rangle}(x+{\rm i}y))\psi(x)\,{\rm d}x
=\!\int_{G^{\langle-1\rangle}([a+\varepsilon,b-\varepsilon]+{\rm i}y)}\!\!\!\!\!F(v)\psi(G(v)-{\rm i}y)G'(v)\,{\rm d}v,
$$
with the smooth change of variable (and path of integration) $G(v)=x+{\rm i}y$ (since $G$ is bijective, the path is simple and smooth). One may suppose without loss of generality 
that $\psi$ is analytic on a neighborhood of $[a+\varepsilon,b-\varepsilon]$ in $\mathbb C$, so that the above is a path integral of analytic functions in $\mathbb C$. 
For a small $\varepsilon>0$ fixed, one may replace the path $G^{\langle-1\rangle}([a+\varepsilon,b-\varepsilon]+{\rm i}y)$ with a straight path starting from the endpoint 
of this path that is closest to $\mathbb R$ and ending under the opposite endpoint, and then closing with a vertical line. By an arbitrarily small change of $\varepsilon$ we may 
assume without loss of generality that $F$ has a finite nontangential limit at $G^{\langle-1\rangle}(a+\varepsilon)$ and $G^{\langle-1\rangle}(b-\varepsilon)$, so that the integral 
of $F(v)\psi(G(v)-{\rm i}y)G'(v)$ along the vertical path tends to zero as $y\to0$. Thus, we may replace the right hand side in \eqref{Distr} with
$$
T^\varepsilon(\psi)\!=\!(2{\rm i})^{-1}\!\!
\lim_{y\to0}\int_{[G^{\!\langle\!-\!1\rangle\!}(a+\varepsilon),G^{\!\langle\!-\!1\rangle\!}(b-\varepsilon)]-{\rm i}y}\!\!\!\!\!\!\!\!\!\!\!\!\!\!\!F(v)\psi(G(v)-{\rm i}y)G'(v)\,{\rm d}v-
\!\!\int_{[G^{\!\langle\!-\!1\rangle\!}(a+\varepsilon),G^{\!\langle\!-\!1\rangle\!}(b-\varepsilon)]+{\rm i}y}\!\!\!\!\!\!\!\!\!\!\!\!\!\!\!F(v)\psi(G(v)+{\rm i}y)G'(v)\,{\rm d}v.
$$
Since the integral is on a finite-length path and $G,G^{\langle-1\rangle}$ are analytic diffeomorphisms, the Nevanlinna representation of $F$ guarantees that $T^\varepsilon$ 
is indeed a finite measure for all $\varepsilon>0$ in an infinite set which has zero as a limit point. Of course, this in particular informs us that the positive measure $\lambda$ in the 
Nevanlinna representation of $F$ puts nonzero mass, and is singular with respect to the Lebesgue measure, on $[a,b]=[G^{\langle-1\rangle}(b),G^{\langle-1\rangle}(a)]$.


It is now of course quite clear that $T^\varepsilon=\nu|_{[a+\varepsilon,b-\varepsilon]}$.

Using the above results, we may now prove  Corollary \ref{c:uniquemu} in its full generality.

\begin{proof}[Proof of Corollary \ref{c:uniquemu}, full generality]
We use the notations from the proof of the restricted version of  Corollary \ref{c:uniquemu} and from Proposition \ref{subord0}. 
Choose connected components $\Omega_0$ of $\Omega$ and $\Omega_0'$ of $\Omega'$ so that $\partial_1\Omega_0\cap\partial_1\Omega_0'\cap(\{0\}\times\mathbb R)$
contains more than one point. According to the hypothesis, this forces $\partial_1\Omega_0\cap\partial_1\Omega_0'\cap(\{0\}\times\supp\nu$) to be infinite.
We have shown in Proposition \ref{subord0} that there exist sets $\mathbf{K,K}'\subset\mathbb C^-$ and functions $\Phi,\Phi'$ defined on them such that
$\partial\mathbf K\cap\partial\mathbf K'$ contains the same component(s) of $\supp\nu$ that are present in $\partial_1\Omega_0\cap\partial_1\Omega_0'$.
The set $\partial\mathbf K\cap\partial\mathbf K'\cap\supp\nu$ must contain a component of $\supp\nu$, which by hypothesis is an infinite set. 
Let us pick such a component and call it $[w,r]$ $(w<r)$. We also pick an $(r-w)/9>\varepsilon>0$. As established above, there exist Schwartz distributions 
$T^\varepsilon$, ${T^\varepsilon}'$ corresponding to $\Phi,\Phi'$, respectively, defined via \eqref{Distr}. As established just after \eqref{78}, we have
\begin{align*}
&\nu|_{[w+\varepsilon,r-\varepsilon]}=\lim_{t\to0}\Im G_{T^\varepsilon}(g(t,\cdot))=\lim_{t\to0}\Im\Phi(g(t,\cdot))=\lim_{t\to0}\Im f(t,\cdot)\\
&\quad\quad\quad\quad\quad\quad\quad\quad\quad=\lim_{t\to0}\Im f'(t,\cdot)=\lim_{t\to0}\Im\Phi'(g'(t,\cdot))=\lim_{t\to0}\Im G_{{T^\varepsilon}'}(g'(t,\cdot))\neq0,
\end{align*}
in the sense of distributions on $[w+\varepsilon,r-\varepsilon]$. Thus, $T^\varepsilon$ and ${T^\varepsilon}'$ coincide on $[w+\varepsilon,r-\varepsilon]$.
In particular, this shows that $\Phi-\Phi'$ has an analytic extension (with real values) to $(w+\varepsilon,r-\varepsilon)$. However, we have assumed that
$f(0,\cdot)=f'(0,\cdot)$ and we have shown that $f(0,\cdot)-G_{T^\varepsilon}(\cdot)$ is real and analytic on 
a neighborhood of $[w+\varepsilon,r-\varepsilon]$, with a similar statement for $f'$ and ${T^\varepsilon}'$. Since ${T^\varepsilon}={T^\varepsilon}'$
implies $G_{T^\varepsilon}=G_{{T^\varepsilon}'}$, the infiniteness of the support of $\nu|_{[w+\varepsilon,r-\varepsilon]}$ guarantees via the identity principle for
analytic functions that $\Phi=\Phi'$ on their common domain, and thus they are extensions of each other. From here on, the proof runs identically with the proof
of the restricted case.
\end{proof}

We record next a fact about the free Brownian bridge which might be of some independent interest.

\begin{corollary}\label{maxi}
The function $f$ reaches its supremum on $\partial\Omega$.
Moreover, $\Im [f(t, x)]$ reaches its supremum on 
$\{0,1\}\times\mathbb R$, $-\Im[ g(t,x)]$ on $\{1\}\times\mathbb R$,
and $\Im [h(t,x)]$ on $\{0\}\times\mathbb R$.
\end{corollary}

\begin{proof}
Let us recall that $f$ satisfies Beltrami's equation \eqref{Bel} and
thus, $f$ is an open mapping. In particular, as an open mapping, it
cannot have a local maximum inside $\Omega$. We shall argue that on 
each simply connected component $\Omega_0$ of $\Omega$, the 
imaginary part $\Im[ f]$ can reach its maximum only on $\partial\Omega_0$. 
Since on $\partial \Omega\cap\{(0,1)\times\mathbb R\}$ we know that 
$\Im [f]$ is zero, it remains that this maximum is reached at a point of 
either $\{0\}\times\mathbb R$ or $\{1\}\times\mathbb R$. 

Thus, assume towards contradiction that there exists a component
$\Omega_0$ of $\Omega$ (simply connected by Lemma \ref{simply}),
a point $(t_0,x_0)\in\Omega_0$ and a neighbourhood $V_0\subseteq
\Omega_0$ of it so that $\Im[ f(t, x)]\leq\Im [f(t_0,x_0)]$ for all $(t,x)\in V_0$
(denote for simplicity $c=\Im f(t_0,x_0)\in(0,+\infty)$). By shrinking
$V_0$ if necessary, we may assume that $\overline{V_0}\subset\Omega_0.$
This means that 
$$
f({V_0})\subset\{z\in\mathbb C^+|\Im [z]\leq c\},
$$
and in addition that $f({V_0})\cap(\mathbb R+\ri c)\neq\varnothing,$ 
as it contains the point $f(t_0, x_0)$. But any neighbourhood
of $f(t_0,x_0)$ contains elements from $\{z\in\mathbb C^+
\colon\Im z> c\},$ so the point $f(t_0, x_0)$ is in 
the boundary of the set $f(V_0)$ while $(t_0,x_0)$ belongs to the
open set $V_0$. This contradicts the openness of $f$ at $(t_0,x_0)$.
The last two statements are obvious consequences of the previous.
\end{proof}

\appendix
\section{Proofs from Section \ref{s:L1spherical} }\label{a:proofL1}

\begin{proof}[Proof of Proposition \ref{p:continuity}]
The first point is proven in \cite[Lemma 5.1]{GZ3} in the case where $\int x^2 d\nu(x)+\int x^2 d\mu(x)\leq  \fK$ and  $\int x^2 d\nu'(x)+\int x^2 d\mu'(x)\leq  \fK$. However it is straightforward to extend this estimate to our setting up to remark that we approximate $\nu$  in $I(\nu,\mu)$ by $\tilde \nu=\nu(\bm1_{|x|\geq \delta^{-1/2}})\delta_0+\nu \bm1_{|x|< \delta^{-1/2}}$. We can construct diagonal matrices $A_N=\diag\{a_1, a_2,\cdots, a_N\}, \tilde A_N=\diag\{\tilde a_1,\tilde a_2,\cdots, \tilde a_N\}, B_N=\diag\{b_1, b_2,\cdots, b_N\}$ with empirical measures converging to $\nu, \tilde \nu, \mu$ respectively. We can construct them such that $\sum_{i}|a_i-\tilde a_i|\leq 2N\int_{|x|\geq \delta^{-1/2}} |x|\rd \nu $ and $|b_i|\leq \fK$.
Then we approximate $I(\nu, \mu), I(\tilde\nu, \mu)$ by spherical integrals $I_N(A_N, B_N), I_N(\tilde A_N, B_N)$. Since 
\begin{align}
|N\Tr(A_N U B_N U^*)-N\Tr(\tilde A_N U B_N U^*)|
\leq N\sum_{ij}|a_i-\tilde a_i||b_j||U_{ij}|^2\leq 2N^2\fK  \nu(|x|\bm1_{|x|\geq \delta^{-1/2}}),
\end{align}
we conclude that $I(\tilde\nu, \mu)$ approximates $I(\nu, \mu)$ up to an error of order $\fK \nu(|x|\bm1_{|x|\geq \delta^{-1/2}})$. This finishes the first point in Proposition \ref{p:continuity}. For the second estimate, 
we have 
\begin{align*}
&\phantom{{}={}}\left|\int T_{\nu}T_{\mu} \rd x-\int T_{\nu'}T_{\mu'} \rd x\right|\\
&\leq\int |T_{\nu}-T_{\nu'}||T_{\mu}| \rd x
+\int |T_{\nu'}| |T_{\mu} -T_{\mu'}| \rd x\\
&\leq \int \fK|T_{\nu} -T_{\nu'}| \rd x+\int |T_{\mu}-T_{\mu'}|\delta^{-1/2} \rd x+\int_{|T_{\nu'}|\geq \delta^{-1/2}} 2\fK |T_{\nu'} |\rd x\\
&\leq \fK \rd_W(\nu, \nu')+\delta^{-1/2} \rd_W(\mu, \mu')+2\fK \oo_\delta(1)=C_\fK\oo_\delta(1),
\end{align*}
where in the third line we used that $|T_\mu|, |T_{\mu'}|\leq \fK$, and in the last line we used the definition of Wasserstein distance \eqref{e:wd2}.
\end{proof}

\begin{proof}[Proof of Proposition \ref{p:spbound}]
Let $A_{N}=\diag(a_1,a_2,\cdots, a_N), B_{N}=\diag(b_1, b_2,\cdots, b_N)$ be two sequences of deterministic self-adjoint matrices, with $a_1\geq a_2\geq \cdots \geq a_N$, and $b_1\geq b_2\geq \cdots\geq b_N$, such that they are $N$-quantiles of the measures $\nu$ and $\mu$ respectively. The upper bound in \eqref{e:spbound} follows directly from the convexity of the spherical integral,
\begin{align*}
I(\nu,\mu)
&=\lim_{N\rightarrow\infty}\frac{1}{\beta N^2}\log\int e^{\frac{\beta N}{2} \Tr(A_NUB_NU^{*})} \rd U\\
&\leq  \limsup_{N\rightarrow\infty}\frac{1}{2 N}\sum_{i=1}^N a_i b_i
=\frac{1}{2}\int T_{\nu}T_{\mu}\rd x\,.
\end{align*}
For the lower bound, we denote $ \bB_{\varepsilon} $ the set of unitary matrices  when $\beta=2$, or orthogonal matrices when $\beta=1$ , such that
\begin{align*}
\bB_\varepsilon=\{U: |U_{ii}-1|\leq \varepsilon \text{ for } 1\leq i\leq N\}.
\end{align*}
On the set $\bB_\varepsilon$, we have for all $i$, $|U_{ii}-1|\leq \varepsilon$ and $\sum_j|U_{ij}-\delta_{ij}|^2\leq 2\varepsilon$. It follows that  
\begin{align}\begin{split}\label{e:AUBUbound}
\Tr(A_N U B_N U^*)
&=\sum_{i,j}a_i b_j |U_{ij}|^2
=\sum_i a_i b_i+\sum_{i}a_i b_i (|U_{ii}|^2-1)+\sum_{i\neq j}a_i b_j |U_{ij}|^2\\
&\geq \sum_i a_i b_i-\fK(2\varepsilon+\varepsilon^2)\sum_{i}|a_i|-2\fK\varepsilon \sum_i|a_i|
\geq \sum_i a_i b_i-5\fK\varepsilon \sum_i|a_i|,
\end{split}\end{align}
provided $\varepsilon<1$.
 Moreover,  notice that $U$ is normal with complex eigenvalues $\{z_1,\ldots,z_N\}$ so that
$$B_\varepsilon=\cap_{1\leq i\leq N} \left\{ |z_i-1| \leq\varepsilon \right\} =\left\{\max_{\|v\|_2=1 }\left| \langle v, (U-I)v\rangle \right| \leq \varepsilon \right\}\subset  \bB_\varepsilon.$$
The joint law of the eigenvalues is well known to be given by a Coulomb gas law from which classical large deviation estimates, see \cite{AGZ, BAZ},
show that  there exists a constant $C(\varepsilon)>0$ such that
$B_{\varepsilon}$ holds with probability at least $e^{-C(\varepsilon)N^2}$.
As a consequence, we also have $\bP(\bB_\varepsilon)\geq e^{-C(\varepsilon)N^2}$.  The lower bound in \eqref{e:spbound} follows 
\begin{align*}
I(\nu,\mu)&=\lim_{N\rightarrow\infty}\frac{1}{\beta N^2}\log\int e^{\frac{\beta N}{2}\Tr(A_NUB_NU^{*})} \rd U\\
&\geq \frac{1}{2}\int T_{\nu}T_{\mu}\rd x
+\liminf_{N\rightarrow\infty}\frac{1}{\beta N^2}\log \bP(\bB_\varepsilon)
-\OO(\varepsilon) \nu(|x|) \\
&\geq \frac{1}{2}\int T_{\nu}T_{\mu}\rd x
-\OO(\varepsilon) \nu(|x|)-C(\varepsilon)\,.
\end{align*}
\end{proof}
\begin{proof}[Proof of Proposition \ref{p:spbound2}]
Let $A_{N}=\diag(a_1,a_2,\cdots, a_N), B_{N}=\diag(b_1, b_2,\cdots, b_N)$ be two sequences of deterministic self-adjoint matrices, with $a_1\geq a_2\geq \cdots \geq a_N$, and $b_1\geq b_2\geq \cdots\geq b_N$, such that they are $N$-quantiles of the measures $\nu$ and $\mu$ respectively.
We denote the truncated diagonal matrix $A_N^\varepsilon=\diag(a_1\bm1(|a_1|\leq 1/\varepsilon),a_2\bm1(|a_2|\leq 1/\varepsilon),\cdots, a_N\bm1(|a_N|\leq 1/\varepsilon))$, which is obtained by removing large entries of $A_N$.
The upper bound in \eqref{e:spbound} follows directly from the spherical integral,
\begin{align}\begin{split}\label{e:Iupp}
I(\nu,\mu)
&=\lim_{N\rightarrow\infty}\frac{1}{\beta N^2}\log\int e^{\frac{\beta N}{2} (\Tr((A_N-A_N^{\varepsilon})UB_NU^{*})+\Tr(A_N^{\varepsilon}UB_NU^{*}))} \rd U\\
&\leq\lim_{N\rightarrow\infty}\frac{1}{\beta N^2}\log\int e^{\sum_{i: |a_i|>1/\varepsilon}a_ib_i+\Tr(A_N^{\varepsilon}UB_NU^{*}))} \rd U\\
&=I(\nu^\varepsilon, \mu)+\frac{1}{2}\int_{|T_\nu|> 1/\varepsilon} T_{\nu}T_{\mu}\rd x,
\end{split}\end{align}
where we used that the empirical eigenvalue density of $A_N^\varepsilon$ converges to $\nu^\varepsilon$.

In the following we prove the lower bound, which is more involved. Let $N_1=|\{i: a_i>1/\varepsilon\}|$, $N_2=|\{i: |a_i|\leq 1/\varepsilon\}|$ and $N_3=|\{i: a_i<-1/\varepsilon\}|$. Since by our assumption that $\nu(|x|)\leq \fK$, it follows that $N_1, N_3\leq \varepsilon \fK N$. We rewrite $A_N, U$ as block matrices
\begin{align*}
A_N=\left[
\begin{array}{ccc}
A_1& 0 & 0\\
0 & A_2 & 0\\
0 & 0 & A_3
\end{array}
\right],\quad 
U=\left[
\begin{array}{c}
U_1\\
U_2\\
U_3
\end{array}
\right],\quad U_i\in \bR^{N_i\times N}\text{ for } i=1,2,3.
\end{align*}
With these new notations, we rewrite the exponent as
\begin{align*}
\Tr(A_NUB_NU^*)
=\Tr(A_1 U_1 B_NU_1^*)+\Tr(A_2 U_2 B_NU_2^*)+\Tr(A_3 U_3 B_NU_3^*).
\end{align*}
We denote $\mathbb B_{\delta}$  the set of unitary matrices when $\beta=2$, or orthogonal matrices when $\beta=1$ 
\begin{align*}
\bB_\delta=\{U:|U_{ii}-1|\leq \delta \text{ for } 1\leq i\leq N_1, N-N_3< i\leq N\}.
\end{align*}
Notice that $P=N_1+N_3\leq 2\varepsilon\fK N$. We show that there exists a constant $C(\delta)<\infty$ such that $\bP(\bB_\delta)\geq e^{-C(\delta)\varepsilon\fK N^2}$. { We prove the case for orthogonal matrices, the unitary case can be proven in the same way. 
We can represent the joint distribution of the vectors  $\{U_i, i\leq N_1, i>N-N_3\}$ as the array of vectors obtained by applying Gram-Schmidt orthonormalization procedure to independent Gaussian vectors $(g_1,g_2,\cdots, g_P)$ with $P=N_1+N_3$. 
We denote the event 
\begin{align}
B_{\delta}=\left\{g_{ii}\geq \sqrt{2N/\delta}, 1\leq i\leq N_1; g_{i(N-P+i)}\geq \sqrt{2N/\delta}, N_1+1\leq i\leq P\right\}.
\end{align}
On this event the entries of $g_1, g_2, \cdots, g_P$ are still independent Gaussian random variables. Moreover, we have
$\bP(B_\delta)\geq e^{-C(\delta)P N}$. The Gram-Schmidt orthonormalization procedure sends $g_i$ to 
\begin{align}
U_i= \frac{g_i-P_{i-1}g_i}{\|g_i-P_{i-1}g_i\|_2},
\end{align}
where $P_{i-1}$ is the projection on the span of $g_{1},g_2,\cdots, g_{i-1}$. In the following we show that conditioning on $B_\delta$, with high probability (larger than $1-e^{-CN\ln N}$) :\begin{enumerate}
\item $\|g_i\|^2=g_{ii}^2+N+\OO(\sqrt N\log N)$, for $1\leq i \leq N_1$; $\|g_i\|^2=g_{i(N-P+i)}^2+N+\OO(\sqrt N\log N)$, for $N_1+1\leq i\leq P$.
\item  $\|P_{i-1}g_i\|_2=g_{ii}(\sqrt{i-1}+\OO(\log N))/\sqrt {N-i+1}+\sqrt{i-1}+\OO(\log N)$, for $1\leq i \leq N_1$; $\|P_{i-1}g_i\|_2=g_{i(N-P+i)}(\sqrt{i-1}+\OO(\log N))/\sqrt {N-i+1}+\sqrt{i-1}+\OO(\log N)$, for $N_1+1\leq i\leq P$.
\end{enumerate}
The first item follows easily from the concentration of $\chi^2$ distributions. For the second item we prove the case that $1\leq i\leq N_1$, the case for $N_2+1\leq i\leq P$ follows from the same argument. By the triangle inequality, we have
\begin{align}\label{e:pp}
\|P_{i-1}g_i\|_2\leq g_{ii}\|P_{i-1}e_i\|_2+\|P_{i-1}(g_i-g_{ii}e_i)\|_2.
\end{align}
For the first projection on the right hand side of \eqref{e:pp}, 
we can upper bound it by replacing the projection to the span of $g_1, g_2, \cdots, g_{i-1}, e_1, e_2, \cdots, e_{i-1}$. Since $e_i$ is orthogonal to $e_1, e_2, \cdots, e_{i-1}$, for the projection to the span of $g_1, g_2, \cdots, g_{i-1}, e_1, e_2, \cdots, e_{i-1}$, we can ignore the $1,2,\cdots (i-1)$-th coordinates. More precisely, we denote $\tilde g_1, \tilde g_2,\cdots, \tilde g_{i-1}\in \bR^{N-i+1}$ from restricting $g_1, g_2, \cdots, g_{i-1}$ to the $i, i+1,\cdots, N$-th coordinate. In this way $\tilde g_1, \tilde g_2,\cdots, \tilde g_{i-1}$ are independent standard Gaussian vectors in $\bR^{N-i+1}$, and the first term on the righthand side of \eqref{e:pp} is bounded by projecting $e_i$ to $\tilde g_1, \tilde g_2,\cdots, \tilde g_{i-1}$. The length of the projection is  $(\sqrt{i-1}+\OO(\log N))/\sqrt {N-i+1}$ with high probability.
For the second projection on the right hand side of \eqref{e:pp}, we can replace the projection to the span of $g_1, g_2, \cdots, g_{i-1}$ ignoring the $i$-th coordinate. In this way $g_i-g_{ii}e_i$ is a standard Gaussian vector. The length of its projection to a $(i-1)$-dim subspace is $\sqrt{i-1}+\OO(\log N)$ with high probability. The second item follows. Combining the arguments above, with high probability
\begin{align}
U_{ii}
&= \frac{g_{ii}-(P_{i-1}g_i)_{i}}{\|g_i-P_{i-1}g_i\|_2}=\frac{g_{ii}(1-\OO(\sqrt \varepsilon))+\OO(\sqrt{\varepsilon N})}{\sqrt{(1-\OO(\varepsilon))g_{ii}^2+N+\OO(\sqrt N \log N)}}=\frac{1+\OO(\sqrt \varepsilon)}{\sqrt{1+N/g_{ii}^2}}.
\end{align}
We recall that from our construction of the set $B_\delta$, $g_{ii}^2\geq 2N/\delta$.
It follows that $|U_{ii}-1|\leq \delta$, provided that $\varepsilon$ is small enough. By a union bound, conditioning on the event $B_\delta$, with high probability it holds that $|U_{ii}-1|\leq \delta \text{ for } 1\leq i\leq N_1, N-N_3< i\leq N$. Therefore, $\bB_\delta$ holds with probability at least $e^{-C(\delta)PN}=e^{-C(\delta)\varepsilon \fK N^2}$.
}
%

On the set $\bB_\delta$, we have $|U_{ii}-1|\leq \delta$ and $\sum_j|U_{ij}-\delta_{ij}|^2\leq 2\delta$ for $1\leq i\leq N_1, N-N_3< i\leq N$. 
Similarly to \eqref{e:AUBUbound}, on the set $\bB_\delta$, we have
\begin{align}\begin{split}\label{e:lowb0}
&\Tr(A_1 U_1 B_NU_1^*)\geq \sum_{i\leq N_1}a_ib_i-5\fK^2\delta,\mbox{ and }\quad \Tr(A_3 U_3 B_NU_3^*)\geq \sum_{i> N-N_3}a_ib_i-5\fK^2 \delta\,.
\end{split}\end{align}
Since $U_2$ is unitary/orthogonal invariant, with \eqref{e:lowb0}, we have a lower bound for the spherical integral
\begin{align}\begin{split}\label{e:lowbb}
&\phantom{{}={}}\frac{1}{\beta N^2}\log\int e^{\frac{\beta N}{2}\Tr(A_NUB_NU^{*})} \rd U\\
&\geq \frac{1}{\beta N^2}\log\int_{\bB_\delta} e^{\frac{\beta N}{2}(\Tr(A_1 U_1 B_NU_1^*)+\Tr(A_2 U_2 B_NU_2^*)+\Tr(A_3 U_3 B_NU_3^*))} \rd U\\
&\geq \frac{1}{2N}\sum_{i\leq N_1\atop \text{or }i>N-N_3}a_ib_i+\frac{1}{\beta N^2}\log\int_{\bB_\delta} e^{\frac{\beta N}{2}\Tr(A_2WU_2B_NU_2^{*}W^*)} \rd W\rd U_2
-2\fK^2(3\delta+\delta^2),
\end{split}\end{align}
where $W$ is an $N_2\times N_2$ unitary/orthogonal matrix following Haar measure.
By our construction, the spectral measure of $A_2$ converges to $\nu\bm1(|x|\leq 1/\varepsilon)/\int_{|x|\leq1/\varepsilon} \rd \nu$, and  
\begin{align*}
\rd_W\left(\frac{\nu\bm1(|x|\leq 1/\varepsilon)}{\int_{|x|\leq1/\varepsilon} \rd \nu}, \nu^\varepsilon\right)=\OO(\varepsilon \fK).
\end{align*}
Thanks to Cauchy's Interlacing Theorem, the eigenvalues of $U_2B_NU_2^{*}$ and $B_N$ are interlaced. Moreover, we have that $N_2\geq N-2\varepsilon \fK N$. The spectral measure of $(N/N_2)U_2B_NU_2^*$ is close to the spectral measure of $\mu$ in Wasserstein distance as defined in \eqref{e:wd}, 
\begin{align*}
\rd_W\left(\mu_{(N/N_2)U_2B_NU_2^*}, \mu\right)=\OO(\varepsilon \fK).
\end{align*}
For the integral on the right hand side of \eqref{e:lowbb}, we can first integrate out $W$, and use Proposition \ref{p:continuity},
\begin{align}\begin{split}\label{e:intbound}
&\phantom{{}={}}\frac{1}{\beta N^2}\log\int_{\bB_\delta} e^{\frac{\beta N}{2}\Tr(A_2WU_2B_NU_2^{*}W^*)} \rd W\rd U_2\\
&=
\frac{1}{\beta N^2}\log\int_{\bB_\delta} e^{\frac{\beta N_2}{2}\Tr(A_2W((N/N_2)U_2B_NU_2^{*})W^*)} \rd W\rd U_2\\
&\geq 
\frac{1}{\beta N^2}\log\int_{\bB_\delta} e^{\beta N_2^2(I(\nu^\varepsilon, \mu)+C_\fK\oo_\varepsilon(1)+\oo_N(1))}\rd U_2\\
&=I(\nu^\varepsilon, \mu)+C_\fK\oo_\varepsilon(1)+\oo_N(1)+\frac{1}{\beta N^2}\log \bP(\bB_\delta)\\
&\geq I(\nu^\varepsilon, \mu)+C_\fK\oo_\varepsilon(1)+\oo_N(1)-C(\delta)\fK\varepsilon 
\end{split}\end{align}
Therefore \eqref{e:lowbb} and \eqref{e:intbound} together implies that
\begin{align}\begin{split}\label{e:lowpp}
I(\nu,\mu)
&=\lim_{N\rightarrow \infty}\frac{1}{\beta N^2}\log\int e^{\frac{\beta N}{2}\Tr(A_NUB_NU^{*})} \rd U\\
&\geq \frac{1}{2}\int_{|T_\nu|>1/\varepsilon}T_\nu T_\mu\rd x+I(\nu^\varepsilon, \mu)-C_\fK\oo_\varepsilon(1),
\end{split}\end{align}
provided we take $\varepsilon$ much smaller than $\delta$.
The estimates \eqref{e:Iupp} and \eqref{e:lowpp} together conclude the proof of Proposition \ref{p:spbound2}.
\end{proof}

\bibliographystyle{plain}
\bibliography{bibAl-1}

\end{document}